\documentclass{article}[11pt]

\usepackage{amsmath,amssymb,amsthm}
\usepackage[retainorgcmds]{IEEEtrantools}
\usepackage{enumerate}
\usepackage{tikz}
\usepackage{tikz-cd}
\usepackage{pifont}

\newtheorem{theorem}{Theorem}[section]
\newtheorem{corollary}[theorem]{Corollary}
\newtheorem{lemma}[theorem]{Lemma}
\newtheorem{proposition}[theorem]{Proposition}

\newtheorem*{claim*}{Claim}

\theoremstyle{remark}
\newtheorem*{remark*}{Remark}

\theoremstyle{definition}
\newtheorem*{definition}{Definition}

\newcommand\R{\mathbb{R}}
\newcommand\C{\mathbb{C}}
\newcommand\Z{\mathbb{Z}}
\newcommand\N{\mathbb{N}}
\newcommand\D{\mathbb{D}}
\newcommand\h{\mathbb{H}}
\newcommand\F{\mathbb{F}}

\newcommand\ip[1]{\left \langle #1 \right \rangle}

\DeclareMathOperator{\id}{id}
\DeclareMathOperator{\im}{im}
\DeclareMathOperator{\Scaf}{Scaf}
\DeclareMathOperator{\Mid}{Mid}
\DeclareMathOperator{\Beg}{Beg}
\DeclareMathOperator{\End}{End}

\DeclareMathOperator{\Hom}{Hom}

\DeclareMathOperator{\rank}{rank}
\DeclareMathOperator{\Perm}{Perm}

	\newcommand\rgvertex[2]{ %
		\fill[red] (#1,{#2+0.15}) arc (90:270:0.15) -- (#1,{#2+0.15});
		\fill[green!50!black!50] (#1,{#2+0.15}) arc (90:-90:0.15) -- (#1,{#2+0.15});
	}
	\newcommand\rvertex[2]{ \fill[red] (#1,#2) circle (0.15); }
	\newcommand\gvertex[2]{ \fill[green!50!black!50] (#1,#2) circle (0.15); }
	\newcommand\ivertex[2]{ \path[draw = black, fill = white] (#1,#2) circle (0.15); }

\usetikzlibrary{decorations.markings}
\tikzset{
    set arrow inside/.code={\pgfqkeys{/tikz/arrow inside}{#1}},
    set arrow inside={end/.initial=>, opt/.initial=},
    /pgf/decoration/Mark/.style={
        mark/.expanded=at position #1 with
        {
            \noexpand\arrow[\pgfkeysvalueof{/tikz/arrow inside/opt}]{\pgfkeysvalueof{/tikz/arrow inside/end}}
        }
    },
    arrow inside/.style 2 args={
        set arrow inside={#1},
        postaction={
            decorate,decoration={
                markings,Mark/.list={#2}
            }
        }
    },
}

\begin{document}

\author{David Jekel}
\title{Layering $\partial$-Graphs and Networks: $\partial$-Graph Transformations, Harmonic Continuation, and the Electrical Inverse Problem}

\begin{center}
{\LARGE Layering $\partial$-Graphs and Networks:}

\Large $\partial$-Graph Transformations, Harmonic Continuation, and the Electrical Inverse Problem

~

David Jekel\footnote{Ph. D. Student at University of California Los Angeles.  B.S. in Math from University of Washington, 2015.}

~

{January 2, 2016, revised May 17, 2016.}
\end{center}

\begin{abstract}
We consider the inverse problem for countable, locally finite electrical networks with edge weights in an arbitrary field.  The electrical inverse problem seeks to determine the weights of the edges knowing only the potential and current data of harmonic functions on a set of boundary nodes.  Motivated by the results of Curtis-Ingerman-Morrow and de-Verdiere-Gitler-Vertigan and others, we formalize the idea of using layer-stripping and harmonic continuation to solve the inverse problem.  Our strategy is to iteratively recover ``vulnerable'' edges near the boundary, then remove them by deletion or contraction.  To recover the vulnerable edge, we set up a clever boundary value problem and solve it using discrete harmonic continuation.

We define ``scaffolds,'' a set of oriented edges that models the flow of information in harmonic continuation.  We formulate a sufficient geometric condition (``recoverability by scaffolds'') for the inverse problem to be solvable using the layer-stripping strategy.  Recoverability by scaffolds is preserved under box products, harmonic subgraphs, covering graphs, and more generally under preimages by unramified harmonic morphisms.  For critical circular planar graphs, we prove recoverability by scaffolds using the medial graph.

We also connect the harmonic continuation process to Baez-Fong's compositional framework for networks and Lam-Pylyavskyy's electrical linear group.  We use this to generalize results of Curtis-Ingerman-Morrow and de-Verdiere-Gitler-Vertigan relating the size of connections through the graph and the rank of submatrices of the response matrix.  We give a symplectic characterization of the boundary behavior for networks and the electrical linear group, valid for fields other than $\mathbb{F}_2$.  Many of our results also generalize to the nonlinear networks such as those of Johnson.
\end{abstract}

\subsection*{Approach and Prerequisites:}

The main thrust of this paper is concrete and geometric--it is about cutting networks apart and gluing them together, stripping away a network layer by layer, and propagating potential and current information step-by-step through a network.  The results are elementary and self-contained enough to be accessible to advanced undergraduates familiar with linear algebra, set theory, basic graph theory, and basic category theory.

Network theory is a multi-faceted subject reaching out to graph theory, physics, probability, algebraic topology, and symplectic Lie theory.  I therefore make passing references to many branches of mathematics, yet none of the other results are essential to understand the main proofs here.

The first major section of the paper is devoted to explaining the main ideas without going into the technical details.  It is meant to serve as a summary to those who do not have time to read the whole thing and as preparation for those who do.  Most of the insights are simple and it is only a matter of choosing the correct definitions to make the proofs work in the best generality.

Familiarity with the results of Curtis-Ingerman-Morrow \cite{CIM} or de Verdiere-Gitler-Vertigan \cite{dVGV} on electrical networks is very helpful in understanding the motivation, though some of the important ideas will be explained here anyway.

\subsection*{Acknowledgements}

James Morrow organized the 2013 math REU at the University of Washington which first got me interested in electrical networks.  Both he and Ian Zemke spent a lot of time listening to and critiquing the ideas which would eventually evolve into this paper.  Jim Morrow and I continued to meet over the next two years and had many useful conversations.

The ideas and examples of the UW REU catalyzed my creative processes.  This paper continues the work of Curtis-Ingerman-Morrow, Will Johnson, Konrad Schr\o der, Ian Zemke, and others.  The REU student papers are sometimes flawed, but I list them to give credit to the originators of the ideas I used.

I thank Pavlo Pylyavskyy for inviting me to present these ideas at the University of Minnesota combinatorics seminar.  The presentation and conversations helped me improve the exposition.

My collaborator Avi Levy from the UW pointed out many useful ideas and references.  Our joint paper ``Torsion of the Graph Laplacian'' \cite{torsion} cross-fertilized this paper in several places.  I also thank the REU students Will Dana, Austin Stromme, and Collin Litterell, with whom we collaborated in the earlier stages of that paper.

Several times, I had the start of an idea that had already been invented by others, and their perspective and terminology ultimately became an integral part of this paper.  If I have copied anyone else's results, it is unintentional, and I will insert proper citations when I become aware.

The papers Curtis-Ingerman-Morrow \cite{CIM} and de-Verdiere-Gitler-Vertigan \cite{dVGV} contain similar results.  However, specific citations will be given from \cite{CIM} since I am more familiar with that paper.

The pictures were produced using Till Tantau's package Tikz.  I also used code posted on Stack Exchange by ``Qrrbrbirlbels'' and ``Henri Menke''
\newline (http://tex.stackexchange.com/questions/163689/add-arrows-to-a-smooth-tikz-function).

I worked on this paper while at the University of Washington and at UCLA, and was in part supported by the NSF grant DMS-1460937t for the REU and James Morrow's RTG.

Soli Deo Gloria.

\newpage

\tableofcontents

\newpage

\section{Background and Results}

The electrical inverse problem seeks to probe the interior of an electrical network from boundary measurements.  We have an electrical network with resistors of unknown properties and we want to figure out what they are from the testing potential and net current at boundary nodes.  The graph-based inverse problem we shall study is a discrete analogue of a continuous problem in PDE, and was motivated by electrical engineering and electrical impedance tomography.

The seminal papers of Curtis-Ingerman-Morrow \cite{CIM} and de Verdiere-Gitler-Vertigan \cite{dVGV} solved the inverse problem for networks embedded in the disk, and proved many related results concerning spanning-tree-determinant formulas, connections through the graph, and medial graphs.

We take as our starting point their idea of using \emph{layer-stripping} and \emph{harmonic continuation} to solve the inverse problem.  Our strategy is to iteratively recover ``vulnerable'' edges near the boundary, then remove them by \emph{deletion} or \emph{contraction} (as in \cite{CIM}, \S 11).  To recover the vulnerable edge, we set up a clever \emph{boundary value problem} and solve it using \emph{discrete harmonic continuation} (as in \cite{CMresistor}, \S 3).

This is perhaps the simplest possible approach to the inverse problem, but it is quite powerful, especially when done systematically.  Will Johnson used this strategy to solve the inverse problem for nonlinear networks in the disk in the arXiv paper \cite{WJ}, and his techniques were adapted to infinite networks in the half-plane in the undergraduate thesis of Ian Zemke \cite{IZ}.  In the spirit of Johnson and Zemke's work, we will formalize the layer-stripping approach and describe sufficient geometric conditions to make each step work.  We will define a class of graphs \emph{recoverable by scaffolds} for which the layer-stripping approach is guaranteed to solve the inverse problem (Theorem \ref{thm:solvablerecoverable}, \S \ref{subsec:solvable}).

Our approach is more general in that we do not assume the graph is embedded in any surface.  However, if we are given an embedded graph with a medial graph, we can still use this to prove recoverability by scaffolds (\S \ref{sec:surfaces}).  We will show that the \emph{critical circular planar graphs} studied by \cite{CIM} and \cite{dVGV} are recoverable by scaffolds (Theorem \ref{thm:CCPTL}, \S \ref{subsec:circularscaf}).  Moreover, the harmonic continuation process does not rely on any special properties of the edge weights, and thus works for networks over arbitrary fields and certain types of nonlinear networks (\S \ref{subsec:nonlinear}).  It adapts to infinite networks and applies to ``supercritical'' networks in the half-plane (\S \ref{subsec:halfplanar}).  We thus reprove the results of \cite{CIM}, \cite{dVGV}, \cite{WJ}, and \cite{IZ} about the inverse problem.

One of the main advantages of our framework is that recoverability by scaffolds can be ``pulled back'' using an adaptation of Urakawa's \emph{harmonic morphisms} \cite{urakawa}:  If $f: G \to G'$ is an \emph{unramified harmonic morphism} of graphs with boundary (see \S \ref{subsec:harmonicmorphisms}), and $G'$ is recoverable by scaffolds, then $G$ is also recoverable by scaffolds (Theorem \ref{thm:solvablepullback}, \S \ref{subsec:solvable}).  In particular, this shows that \emph{covering graphs}, \emph{subgraphs}, and \emph{box products} of graphs recoverable by scaffolds are also recoverable by scaffolds, which greatly expands the list of known recoverable networks.

The key geometric construction in our sufficient condition is called a \emph{scaffold} (\S \ref{sec:scaffolds}).  It is a set of oriented edges of a scaffold, roughly speaking, show the direction of harmonic continuation.  It turns out that scaffolds and layer-stripping itself are related and describe the same fundamental structure (\S \ref{subsec:scaflayerability}).  Yet a third perspective is furnished by Baez-Fong's compositional framework \cite{BF}, in which a graph is viewed as a morphism from a set of input vertices to a set of output vertices, and composition of morphisms glues the outputs of the first graph to the inputs of the second graph (see \S \ref{sec:IO}).  Scaffolds are equivalent to certain \emph{elementary factorizations} in this category, which express a morphism as a concatenation of very simple networks corresponding to individual steps in the harmonic continuation process (\S \ref{subsec:elementary}).

We refer to these related ideas collectively as ``layering theory.''  We will apply layering theory not only to solve the inverse problem, but also to generalize the \emph{rank-connection principle} observed by \cite{CIM} and \cite{dVGV} relating the ranks of submatrices of the response matrix and the size of connections through the graph (see \cite{CIM} Theorem 4.2).  We formulate a version of the rank-connection principle that makes sense for arbitrary fields, even when the response matrix is not defined, and holds for generic edge weights.  Using layering theory, we describe necessary and sufficient conditions on the graph for the rank-connection principle to hold for all edge weights (Theorem \ref{thm:RC3}, \S \ref{subsec:semielementary}).

We give a description of Lam-Pylyavksyy \cite{LP}'s electrical linear group in terms of layering theory (\S \ref{subsec:IOlayerstripping}).  Motivated by results of Baez-Fong and Lam-Pylyavksyy, we use layering theory to prove a symplectic characterization of the possible boundary behaviors of electrical networks over any field (Theorem \ref{thm:symplectic}, \S \ref{subsec:symplecticchar}).  We also characterize the electrical linear group for fields other than $\F_2$ (Theorem \ref{thm:ELsymplectic}, \S \ref{subsec:ELchar}).

\section{Overview of Main Ideas}

This section motivates and describes the main constructions of this paper, leaving out some of the technicalities and applications.  This overview is meant to summarize the main ideas for those who do not have time to read the whole paper, and to make the later technical developments more digestible for those who will keep reading.

After giving the main definitions, we explain 1) the layer-stripping strategy for the inverse problem, 2) recovering boundary spikes and boundary edges using harmonic continuation, 3) using scaffolds as a geometric model for the harmonic continuation process, 4) application to mixed-data boundary value problems, 5) another model for harmonic continuation using a category where the morphisms are graphs and composition glues them together.

\subsection{Definitions}

In this paper, a {\bf graph} $G$ is a countable, locally finite, undirected 
multi-graph with self-loops allowed. We write $V$ for the vertex set and $E$ 
for the set of {\bf oriented} edges.  If $e$ is an oriented edge, $e_-$ and 
$e_+$ refer to its starting and ending vertices, and $\overline{e}$ refers to 
its reverse orientation.  The {\bf degree} of a vertex $p$ is the number of oriented edges with $e_+ = p$.

A {\bf graph with boundary} (abbreviated to {\bf 
$\partial$-graph}) is a graph with a specified partition of $V$ into two sets 
$V^\circ$ and $\partial V$, called the {\bf interior} and {\bf boundary vertices} 
respectively.  

For a field $\mathbb{F}$, we define an {\bf $\mathbb{F}$-network} $\Gamma = (G,w)$ as a $\partial$-graph $G$ together with a weight function $w: E \to \mathbb{F} \setminus 0$ with $w(e) = w(\overline{e})$.  Traditionally, the weights are in $\R_+$, but most of our results hold for general fields.

A {\bf potential} is a function $u: V \to \mathbb{F}$.  For a potential $u$, we define $du(e) = u(e_+) - u(e_-)$.  The {\bf current} on an oriented edge $e$ induced by the potential $u$ is
\[
w(e) du(\overline{e}) = -w(e) du(e) = w(e) (u(e_-) - u(e_+)).
\]
The reason for the negative sign is that ``current flow goes in the opposite direction of the gradient.''  The {\bf net current} at a vertex $p$ is given by the {\bf weighted Laplacian}
\[
\Delta u(p) = \sum_{e: e_+ = p} w(e)du(e).
\]
We say $u$ is {\bf harmonic} if $\Delta u(p) = 0$ for all \emph{interior} vertices $p$.  We denote the {\bf vector space of harmonic functions} by $\mathcal{U}(\Gamma)$.  Physically, harmonic functions represent valid electrical potentials that satisfy Ohm's law that the current on an edge is $w(e) du(e)$ and Kirchhoff's law that the net current at an interior vertex is zero.

The {\bf boundary data} of a harmonic function $u$ is the pair $(u|_{\partial V}, \Delta u|_{\partial V})$.  The {\bf boundary behavior}
\[
\Lambda(\Gamma) = \{(u|_{\partial V}, \Delta u|_{\partial V}): u \in \mathcal{U}(\Gamma)\}
\]
is the set of all pairs $(\phi,\psi)$ which are the boundary data of harmonic functions.  It is a linear subspace of $\mathbb{F}^{\partial V} \times \mathbb{F}^{\partial V}$.

We consider the following version of the {\bf inverse problem}: For a fixed graph $G$ and field $\mathbb{F}$, are the edge weights uniquely determined by the boundary behavior?  That is, is $w \mapsto \Lambda(G,w)$ injective?  If the answer is yes, then we say $G$ is {\bf recoverable} (over $\mathbb{F}$).

For positive real edge-weights and finite networks, there is a unique harmonic function with any prescribed potentials on $\partial V$ (that is, the \emph{Dirichlet problem} has a unique solution) (see \cite{CIM}).  Thus, the role of $\Lambda(\Gamma)$ is traditionally played by the {\bf Dirichlet-to-Neumann map} or {\bf response matrix}, a linear transformation $\mathbb{F}^{\partial V} \to \mathbb{F}^{\partial V}$ that sends $\phi \in \F^{\partial V}$ to the net current vector of the harmonic function with $u|_{\partial V} = \phi$.  In this situation, the boundary behavior is the \emph{graph} of the Dirichlet-to-Neumann map.  In general, there might not be a Dirichlet-to-Neumann map, so we must use the boundary behavior instead.

\subsection{Layer-Stripping} \label{subsec:strategy}

We want to recover our network through the following iterative procedure (the {\bf layer-stripping strategy}):  As long as there are edges left in the network
\begin{enumerate}
	\item Locate some edge $e$ which is ``near the boundary,'' and find $w(e)$ from $\Lambda(\Gamma)$.
	\item Delete or contract $e$ to obtain a smaller network $\Gamma'$.
	\item Use $w(e)$ and $\Lambda(\Gamma)$ to compute the boundary behavior $\Lambda(\Gamma')$.
	\item Repeat with $\Gamma'$ instead of $\Gamma$.
\end{enumerate}

There are two types of ``near-boundary'' edges:  A {\bf boundary spike} is an edge with one interior endpoint and one boundary endpoint of degree $1$.  A {\bf boundary edge} is an edge where both endpoints are boundary vertices.  A boundary spike and a boundary edge are pictured below; the boundary vertices are black and the interior vertices are white: %See Figure \ref{bdspikebdedge}.

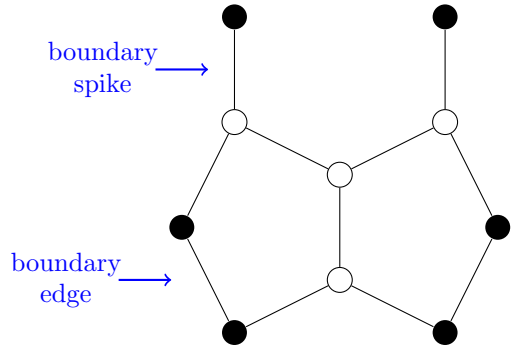
\begin{center}
\begin{tikzpicture}[scale=0.7]
	\node[circle,fill] (1) at (-2,4) {};
	\node[circle,fill] (2) at (2,4) {};
	\node[circle,fill] (3) at (-3,0) {};
	\node[circle,draw] (4) at (-2,2) {};
	\node[circle,draw] (5) at (0,1) {};
	\node[circle,draw] (6) at (2,2) {};
	\node[circle,fill] (7) at (3,0) {};
	\node[circle,fill] (8) at (-2,-2) {};
	\node[circle,draw] (9) at (0,-1) {};
	\node[circle,fill] (10) at (2,-2) {};
	
	\draw (1) to (4) to (3) to (8) to (9) to (10) to (7) to (6) to (2);
	\draw (4) to (5) to (6);
	\draw (5) to (9);
	
	\draw[->,thick,blue] (-3.5,3) -- (-2.5,3);
	\node[blue,align=center] at (-4.5,3) {boundary \\ spike};
	
	\draw[->,thick,blue] (-4.2,-1) -- (-3.2,-1);
	\node[blue,align=center] at (-5.2,-1) {boundary \\ edge};
	
\end{tikzpicture}
\end{center}

Boundary spikes are removed by {\bf contraction} and boundary edges are removed by {\bf deletion}.  When a boundary spike is contracted, the two endpoints are identified.  The new vertex occupies the position of the interior endpoint, but becomes a boundary vertex: 

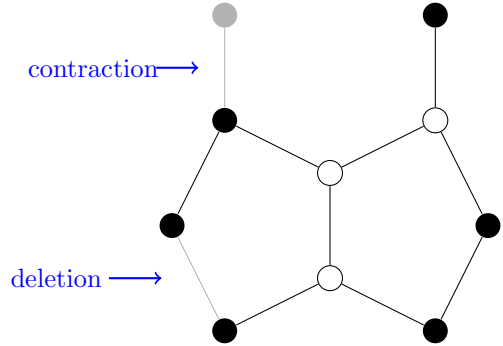
\begin{center}
\begin{tikzpicture}[scale=0.7]
	\node[circle,fill,black!30] (1) at (-2,4) {};
	\node[circle,fill] (2) at (2,4) {};
	\node[circle,fill] (3) at (-3,0) {};
	\node[circle,fill] (4) at (-2,2) {};
	\node[circle,draw] (5) at (0,1) {};
	\node[circle,draw] (6) at (2,2) {};
	\node[circle,fill] (7) at (3,0) {};
	\node[circle,fill] (8) at (-2,-2) {};
	\node[circle,draw] (9) at (0,-1) {};
	\node[circle,fill] (10) at (2,-2) {};
	
	\draw[black!30] (1) to (4);
	\draw (4) to (3);
	\draw[black!30] (3) to (8);
	\draw (8) to (9) to (10) to (7) to (6) to (2);
	\draw (4) to (5) to (6);
	\draw (5) to (9);
	
	\draw[->,thick,blue] (-3.3,3) -- (-2.5,3);
	\node[blue,align=center] at (-4.5,3) {contraction};
	
	\draw[->,thick,blue] (-4.2,-1) -- (-3.2,-1);
	\node[blue,align=center] at (-5.2,-1) {deletion};
	
\end{tikzpicture}
\end{center}

If $\Gamma'$ is obtained from $\Gamma$ by contracting a boundary spike or deleting a boundary edge, then computing $\Lambda(\Gamma')$ from $\Lambda(\Gamma)$ and vice versa (step 3) is straightforward.  First suppose $\Gamma'$ is obtained by contracting a boundary spike $e$ with boundary endpoint $p$ and interior endpoint $q$.  Note that any harmonic function $u'$ on $\Gamma'$ extends uniquely to a harmonic function $u$ on $\Gamma$.  We simply choose the potential at $p$ to make the net current at $q$ be zero, that is, set $u(p) = u'(q) + w(e)^{-1} \Delta u'(q)$, where the Laplacian is computed in $\Gamma'$.

To find the boundary data of $u$ from that of $u'$, replace $q$ with $p$ in the list of boundary vertices, replace $u'(q)$ with $u(p)$, and leave everything else the same.  Note that all the current at $q$ from the edges in $\Gamma'$ must flow from $q$ to $p$ in $\Gamma$, and thus, $\Delta u'(q)$ in $\Gamma'$ is the same as $\Delta u(p)$ in $\Gamma$.  The boundary data of $u$ is found by a single row operation from the boundary data of $u'$, which adjusts the potential at $p$ or $q$ based on the net current there.  If
\[
\Xi: \mathbb{F}^{\partial V'} \times \mathbb{F}^{\partial V'} \to \mathbb{F}^{\partial V} \times \mathbb{F}^{\partial V}
\]
is the corresponding linear transformation, then $\Lambda(\Gamma) = \Xi \Lambda(\Gamma')$ and $\Lambda(\Gamma') = \Xi^{-1} \Lambda(\Gamma)$.

For a boundary edge, there is a similar transformation | simply change the net current on the two endpoints by $\pm w(e)(u(e_+) - u(e_-))$.

As we shall see in \S \ref{sec:symplectic}, for finite graphs, $\Lambda(\Gamma)$ is a Lagrangian subspace of $\mathbb{F}^{\partial V} \times \mathbb{F}^{\partial V}$ with respect to the standard symplectic form, and the linear transformations for adding boundary spikes and boundary edges are symplectic matrices.  In fact, for any field with more than two elements, these special matrices generate the group of symplectic matrices $\Xi$ which map $(1,\dots,1,0,\dots,0)$ to itself.  This provides another perspective on the \emph{electrical linear group} of \cite{LP}.

\subsection{Recovery of Boundary Spikes and Boundary Edges} \label{subsec:recoveryexample}

The hardest step of the layer-stripping strategy is the recovery of boundary spikes and boundary edges.  Let us first handle the case of a boundary spike $e$ with boundary endpoint $p$ and interior endpoint $q$.  Our goal will be to find $P, Q \subset \partial V$ such that:
\begin{itemize}
	\item \emph{Existence:} For any possible choice of edge weights, there exists a harmonic function $u$ with $u|_P = 0$, $\Delta u|_Q = 0$, and $u(p) = 1$.
	\item \emph{Uniqueness:} For any possible choice edge weights, a harmonic function $u$ with $u|_P = 0$ and $\Delta u|_Q = 0$ is forced to have $u(q) = 0$.
\end{itemize}
Note that if $u$ is any such harmonic function, then the net current on $p$ is
\[
\Delta u(p) = w(e)(u(p) - u(q)) = w(e)(1 - 0) = w(e).
\]
If we establish our two claims, that will show that $w(e)$ is uniquely determined by $\Lambda(\Gamma)$.  Indeed, by solving some linear equations, we can find a pair $(\phi,\psi) \in \Lambda(\Gamma)$ such that $\phi|_P = 0$, $\psi|_Q = 0$, and $\phi(e_+) = 1$.  If we pick any such pair, $\psi(p)$ is guaranteed to be $\Delta u(p) = w(e)$.

For a boundary edge $e$ with endpoints $p$ and $q$, the strategy is the same, except that this time we force $u$ to be zero on $p$ and all its neighbors other than $q$, and we force $u(q) = 1$.  This guarantees that $\Delta u(p) = -w(e)$.

We demonstrate the two claims about $P$ and $Q$ using \emph{discrete harmonic continuation}, which is best explained by example.  We will recover the boundary spike in the earlier example by imposing the boundary conditions pictured below.  The conditions in parentheses denote the net current and the ones not in parentheses denote the potential.

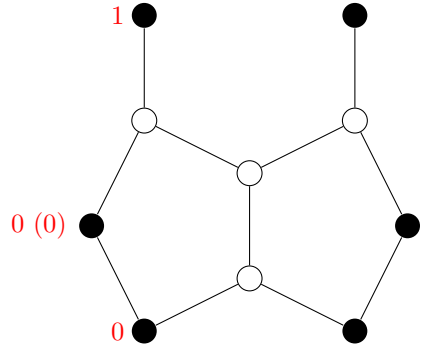
\begin{center}
\begin{tikzpicture}[scale=0.7]
	\node[circle,fill] (1) at (-2,4) {};
	\node[circle,fill] (2) at (2,4) {};
	\node[circle,fill] (3) at (-3,0) {};
	\node[circle,draw] (4) at (-2,2) {};
	\node[circle,draw] (5) at (0,1) {};
	\node[circle,draw] (6) at (2,2) {};
	\node[circle,fill] (7) at (3,0) {};
	\node[circle,fill] (8) at (-2,-2) {};
	\node[circle,draw] (9) at (0,-1) {};
	\node[circle,fill] (10) at (2,-2) {};
	
	\begin{scope}[text = red]
		\node at (-2.5,4) {$1$};
		\node at (-4, 0) {$0$ ($0$)};
		\node at (-2.5, -2) {$0$};
	\end{scope}
	
	\draw (1) to (4);
	\draw (4) to (3);
	\draw (3) to (8);
	\draw (8) to (9);
	\draw (9) to (10);
	\draw (10) to (7);
	\draw (7) to (6);
	\draw (6) to (2);
	\draw (4) to (5);
	\draw (5) to (6);
	\draw (5) to (9);
	
\end{tikzpicture}
\end{center}

In this example, $P$ comprises the two lower left boundary vertices, and $Q$ is the single vertex where the net current is zero.  At this point, we see an edge where $u = 0$ on both endpoints, and deduce that the current on the edge is zero and color the edge blue (below, left).  Next, the vertex where the net current is declared to be zero has only two edges incident to it and one of them has zero current already.  This implies the current is zero on the other edge, and hence we conclude that the potential is zero on the interior vertex of the spike (below, right).

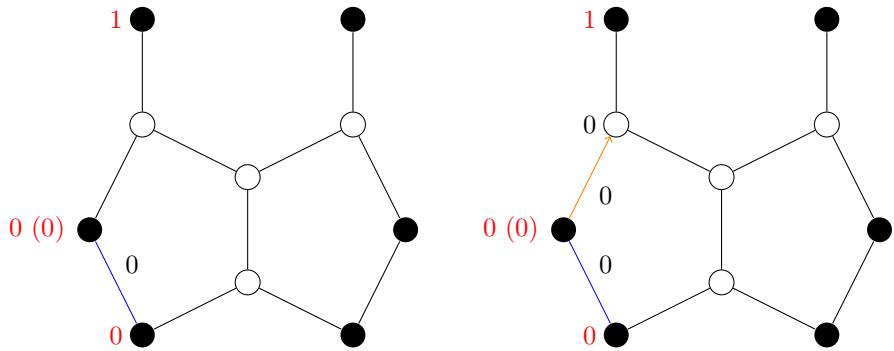
\begin{center}
\begin{tikzpicture}[scale=0.7]
	\node[circle,fill] (1) at (-2,4) {};
	\node[circle,fill] (2) at (2,4) {};
	\node[circle,fill] (3) at (-3,0) {};
	\node[circle,draw] (4) at (-2,2) {};
	\node[circle,draw] (5) at (0,1) {};
	\node[circle,draw] (6) at (2,2) {};
	\node[circle,fill] (7) at (3,0) {};
	\node[circle,fill] (8) at (-2,-2) {};
	\node[circle,draw] (9) at (0,-1) {};
	\node[circle,fill] (10) at (2,-2) {};
	
	\begin{scope}[text = red]
		\node at (-2.5,4) {$1$};
		\node at (-4, 0) {$0$ ($0$)};
		\node at (-2.5, -2) {$0$};
	\end{scope}
	
	\draw (1) to (4);
	\draw (4) to (3);
	\draw[blue] (3) to node[auto,black] {$0$} (8);
	\draw (8) to (9);
	\draw (9) to (10);
	\draw (10) to (7);
	\draw (7) to (6);
	\draw (6) to (2);
	\draw (4) to (5);
	\draw (5) to (6);
	\draw (5) to (9);
	
	\begin{scope}[shift = {(9,0)}]
	
	\node[circle,fill] (1) at (-2,4) {};
	\node[circle,fill] (2) at (2,4) {};
	\node[circle,fill] (3) at (-3,0) {};
	\node[circle,draw] (4) at (-2,2) {};
	\node[circle,draw] (5) at (0,1) {};
	\node[circle,draw] (6) at (2,2) {};
	\node[circle,fill] (7) at (3,0) {};
	\node[circle,fill] (8) at (-2,-2) {};
	\node[circle,draw] (9) at (0,-1) {};
	\node[circle,fill] (10) at (2,-2) {};
	
	\begin{scope}[text = red]
		\node at (-2.5,4) {$1$};
		\node at (-4, 0) {$0$ ($0$)};
		\node at (-2.5, -2) {$0$};
	\end{scope}
	
	\node at (-2.5,2) {$0$};
	
	\draw (1) to (4);
	\draw[->,orange] (3) to node[auto,swap,black] {$0$} (4);
	\draw[blue] (3) to node[auto,black] {$0$} (8);
	\draw (8) to (9);
	\draw (9) to (10);
	\draw (10) to (7);
	\draw (7) to (6);
	\draw (6) to (2);
	\draw (4) to (5);
	\draw (5) to (6);
	\draw (5) to (9);
	\end{scope}
	
\end{tikzpicture}
\end{center}

Next, the current on the boundary spike is determined since the potentials on the endpoints are determined (below, left).  As remarked above, the current on the spike is $w(e)$ from the boundary vertex to the interior vertex, and thus the net current on the boundary vertex is $w(e)$.

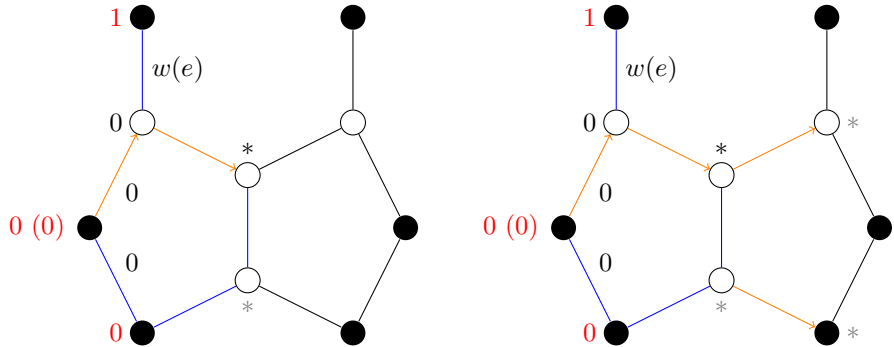
\begin{center}
\begin{tikzpicture}[scale=0.7]
	\node[circle,fill] (1) at (-2,4) {};
	\node[circle,fill] (2) at (2,4) {};
	\node[circle,fill] (3) at (-3,0) {};
	\node[circle,draw] (4) at (-2,2) {};
	\node[circle,draw] (5) at (0,1) {};
	\node[circle,draw] (6) at (2,2) {};
	\node[circle,fill] (7) at (3,0) {};
	\node[circle,fill] (8) at (-2,-2) {};
	\node[circle,draw] (9) at (0,-1) {};
	\node[circle,fill] (10) at (2,-2) {};
	
	\begin{scope}[text = red]
		\node at (-2.5,4) {$1$};
		\node at (-4, 0) {$0$ ($0$)};
		\node at (-2.5, -2) {$0$};
	\end{scope}

	\node at (-2.5,2) {$0$};

	\draw[blue] (1) to node[auto,black] {$w(e)$} (4);
	\draw[->,orange] (3) to node[auto,swap,black] {$0$} (4);
	\draw[blue] (3) to node[auto,black] {$0$} (8);
	\draw (8) to (9);
	\draw (9) to (10);
	\draw (10) to (7);
	\draw (7) to (6);
	\draw (6) to (2);
	\draw (4) to (5);
	\draw (5) to (6);
	\draw (5) to (9);
\end{tikzpicture}
\end{center}

We have now completed the ``uniqueness'' step, showing that our boundary conditions force $u$ to be zero on the interior vertex of the spike.  It remains to show that our partially defined function extends to \emph{some} harmonic function on the whole network.  Note that we do not care about uniqueness any more since the behavior of $u$ near the spike is under control.

We start at the interior vertex $q$ of the spike.  There is only one edge at $q$ where the current is not yet determined, but we can choose the current on this edge to make the net current at $q$ zero (below, left).  At this point, the data we have on the network does not determine any more values of $u$.  To continue with our harmonic extension, we assign a potential $*$ arbitrarily at the vertex indicated in gray in the picture (below, right).

\begin{center}
\begin{tikzpicture}[scale = 0.7]

	\node[circle,fill] (1) at (-2,4) {};
	\node[circle,fill] (2) at (2,4) {};
	\node[circle,fill] (3) at (-3,0) {};
	\node[circle,draw] (4) at (-2,2) {};
	\node[circle,draw] (5) at (0,1) {};
	\node[circle,draw] (6) at (2,2) {};
	\node[circle,fill] (7) at (3,0) {};
	\node[circle,fill] (8) at (-2,-2) {};
	\node[circle,draw] (9) at (0,-1) {};
	\node[circle,fill] (10) at (2,-2) {};
	
	\begin{scope}[text = red]
		\node at (-2.5,4) {$1$};
		\node at (-4, 0) {$0$ ($0$)};
		\node at (-2.5, -2) {$0$};
	\end{scope}

	\node at (-2.5,2) {$0$};
	\node at (0,1.5) {$*$};

	\draw[blue] (1) to node[auto,black] {$w(e)$} (4);
	\draw[->,orange] (3) to node[auto,swap,black] {$0$} (4);
	\draw[blue] (3) to node[auto,black] {$0$} (8);
	\draw (8) to (9);
	\draw (9) to (10);
	\draw (10) to (7);
	\draw (7) to (6);
	\draw (6) to (2);
	\draw[->,orange] (4) to (5);
	\draw (5) to (6);
	\draw (5) to (9);
	
	\begin{scope}[shift = {(9,0)}]

	\node[circle,fill] (1) at (-2,4) {};
	\node[circle,fill] (2) at (2,4) {};
	\node[circle,fill] (3) at (-3,0) {};
	\node[circle,draw] (4) at (-2,2) {};
	\node[circle,draw] (5) at (0,1) {};
	\node[circle,draw] (6) at (2,2) {};
	\node[circle,fill] (7) at (3,0) {};
	\node[circle,fill] (8) at (-2,-2) {};
	\node[circle,draw] (9) at (0,-1) {};
	\node[circle,fill] (10) at (2,-2) {};
	
	\begin{scope}[text = red]
		\node at (-2.5,4) {$1$};
		\node at (-4, 0) {$0$ ($0$)};
		\node at (-2.5, -2) {$0$};
	\end{scope}

	\node at (-2.5,2) {$0$};
	\node at (0,1.5) {$*$};

	\begin{scope}[text=gray]
		\node at (0,-1.5) {$*$};
	\end{scope}

	\draw[blue] (1) to node[auto,black] {$w(e)$} (4);
	\draw[->,orange] (3) to node[auto,swap,black] {$0$} (4);
	\draw[blue] (3) to node[auto,black] {$0$} (8);
	\draw (8) to (9);
	\draw (9) to (10);
	\draw (10) to (7);
	\draw (7) to (6);
	\draw (6) to (2);
	\draw[->,orange] (4) to (5);
	\draw (5) to (6);
	\draw (5) to (9);
	
	\end{scope}
\end{tikzpicture}
\end{center}

Once again, we see two edges where the potential on the endpoints is known.  We color them blue to indicate that the current on these edges is known (below, left).  Then we see interior vertices with only one underdetermined edge each, and we must choose the current on these edges to make the net current at the interior vertices zero; they are indicated in orange (below, right).

\begin{center}
\begin{tikzpicture}[scale=0.7]
	\node[circle,fill] (1) at (-2,4) {};
	\node[circle,fill] (2) at (2,4) {};
	\node[circle,fill] (3) at (-3,0) {};
	\node[circle,draw] (4) at (-2,2) {};
	\node[circle,draw] (5) at (0,1) {};
	\node[circle,draw] (6) at (2,2) {};
	\node[circle,fill] (7) at (3,0) {};
	\node[circle,fill] (8) at (-2,-2) {};
	\node[circle,draw] (9) at (0,-1) {};
	\node[circle,fill] (10) at (2,-2) {};
	
	\begin{scope}[text = red]
		\node at (-2.5,4) {$1$};
		\node at (-4, 0) {$0$ ($0$)};
		\node at (-2.5, -2) {$0$};
	\end{scope}

	\node at (-2.5,2) {$0$};
	\node at (0,1.5) {$*$};

	\begin{scope}[text=gray]
		\node at (0,-1.5) {$*$};
	\end{scope}

	\draw[blue] (1) to node[auto,black] {$w(e)$} (4);
	\draw[->,orange] (3) to node[auto,swap,black] {$0$} (4);
	\draw[blue] (3) to node[auto,black] {$0$} (8);
	\draw[blue] (8) to (9);
	\draw (9) to (10);
	\draw (10) to (7);
	\draw (7) to (6);
	\draw (6) to (2);
	\draw[->,orange] (4) to (5);
	\draw (5) to (6);
	\draw[blue] (5) to (9);
	
\begin{scope}[shift = {(9,0)}]

	\node[circle,fill] (1) at (-2,4) {};
	\node[circle,fill] (2) at (2,4) {};
	\node[circle,fill] (3) at (-3,0) {};
	\node[circle,draw] (4) at (-2,2) {};
	\node[circle,draw] (5) at (0,1) {};
	\node[circle,draw] (6) at (2,2) {};
	\node[circle,fill] (7) at (3,0) {};
	\node[circle,fill] (8) at (-2,-2) {};
	\node[circle,draw] (9) at (0,-1) {};
	\node[circle,fill] (10) at (2,-2) {};
	
	\begin{scope}[text = red]
		\node at (-2.5,4) {$1$};
		\node at (-4, 0) {$0$ ($0$)};
		\node at (-2.5, -2) {$0$};
	\end{scope}

	\node at (-2.5,2) {$0$};
	\node at (0,1.5) {$*$};

	\begin{scope}[text=gray]
		\node at (0,-1.5) {$*$};
		\node at (2.5,2) {$*$};
		\node at (2.5,-2) {$*$};
	\end{scope}

	\draw[blue] (1) to node[auto,black] {$w(e)$} (4);
	\draw[->,orange] (3) to node[auto,swap,black] {$0$} (4);
	\draw[blue] (3) to node[auto,black] {$0$} (8);
	\draw[blue] (8) to (9);
	\draw[->,orange] (9) to (10);
	\draw (10) to (7);
	\draw (7) to (6);
	\draw (6) to (2);
	\draw[->,orange] (4) to (5);
	\draw[->,orange] (5) to (6);
	\draw[blue] (5) to (9);
\end{scope}
	
\end{tikzpicture}
\end{center}

Finally, we declare an arbitrary parameter at the right middle boundary vertex, color two edges blue, and then one edge orange:

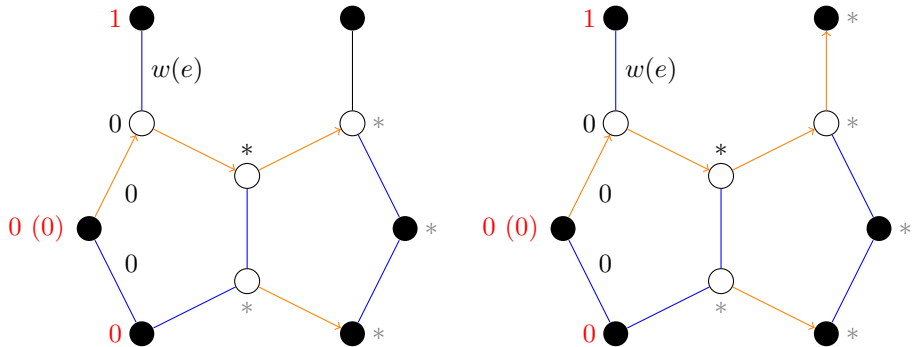
\begin{center}
\begin{tikzpicture}[scale=0.7]
	\node[circle,fill] (1) at (-2,4) {};
	\node[circle,fill] (2) at (2,4) {};
	\node[circle,fill] (3) at (-3,0) {};
	\node[circle,draw] (4) at (-2,2) {};
	\node[circle,draw] (5) at (0,1) {};
	\node[circle,draw] (6) at (2,2) {};
	\node[circle,fill] (7) at (3,0) {};
	\node[circle,fill] (8) at (-2,-2) {};
	\node[circle,draw] (9) at (0,-1) {};
	\node[circle,fill] (10) at (2,-2) {};
	
	\begin{scope}[text = red]
		\node at (-2.5,4) {$1$};
		\node at (-4, 0) {$0$ ($0$)};
		\node at (-2.5, -2) {$0$};
	\end{scope}

	\node at (-2.5,2) {$0$};
	\node at (0,1.5) {$*$};

	\begin{scope}[text=gray]
		\node at (0,-1.5) {$*$};
		\node at (2.5,2) {$*$};
		\node at (2.5,-2) {$*$};
		\node at (3.5,0) {$*$};
	\end{scope}

	\draw[blue] (1) to node[auto,black] {$w(e)$} (4);
	\draw[->,orange] (3) to node[auto,swap,black] {$0$} (4);
	\draw[blue] (3) to node[auto,black] {$0$} (8);
	\draw[blue] (8) to (9);
	\draw[->,orange] (9) to (10);
	\draw[blue] (10) to (7);
	\draw[blue] (7) to (6);
	\draw (6) to (2);
	\draw[->,orange] (4) to (5);
	\draw[->,orange] (5) to (6);
	\draw[blue] (5) to (9);
	
\begin{scope}[shift = {(9,0)}]
	\node[circle,fill] (1) at (-2,4) {};
	\node[circle,fill] (2) at (2,4) {};
	\node[circle,fill] (3) at (-3,0) {};
	\node[circle,draw] (4) at (-2,2) {};
	\node[circle,draw] (5) at (0,1) {};
	\node[circle,draw] (6) at (2,2) {};
	\node[circle,fill] (7) at (3,0) {};
	\node[circle,fill] (8) at (-2,-2) {};
	\node[circle,draw] (9) at (0,-1) {};
	\node[circle,fill] (10) at (2,-2) {};
	
	\begin{scope}[text = red]
		\node at (-2.5,4) {$1$};
		\node at (-4, 0) {$0$ ($0$)};
		\node at (-2.5, -2) {$0$};
	\end{scope}

	\node at (-2.5,2) {$0$};
	\node at (0,1.5) {$*$};

	\begin{scope}[text=gray]
		\node at (0,-1.5) {$*$};
		\node at (2.5,2) {$*$};
		\node at (2.5,-2) {$*$};
		\node at (3.5,0) {$*$};
		\node at (2.5,4) {$*$};
	\end{scope}

	\draw[blue] (1) to node[auto,black] {$w(e)$} (4);
	\draw[->,orange] (3) to node[auto,swap,black] {$0$} (4);
	\draw[blue] (3) to node[auto,black] {$0$} (8);
	\draw[blue] (8) to (9);
	\draw[->,orange] (9) to (10);
	\draw[blue] (10) to (7);
	\draw[blue] (7) to (6);
	\draw[->,orange] (6) to (2);
	\draw[->,orange] (4) to (5);
	\draw[->,orange] (5) to (6);
	\draw[blue] (5) to (9);
\end{scope}

\end{tikzpicture}
\end{center}

We have now shown that our partially defined harmonic function extends to the whole network.  By construction, the net current at each interior vertex is zero.  Indeed, each interior vertex has an orange edge exiting it, and the current on the orange edge was chosen to make the net current at the starting point zero.  Therefore, we have proved that the weight of the boundary spike is uniquely determined by $\Lambda(\Gamma)$.

\subsection{Formalizing Harmonic Continuation with Scaffolds} \label{subsec:formalizing}

Although the process of defining a harmonic function in the last example was ostensibly algebraic, it can be represented purely geometrically by the orange edges, blue edges, and relationship between them.  The set of orange edges is an example of a \emph{scaffold}, a set $S$ of oriented edges satisfying certain conditions, designed as an auxiliary framework to build a harmonic function, or as a geometric model of the flow of information.  To motivate the definition of scaffolds, let us try to formalize the process in the last example.

Discrete harmonic continuation has two types of moves, represented by the blue and orange edges.  For each blue edge, the values of $u$ on the endpoints are determined first, and that defines the current on the edge.  For each orange edge, we know $u$ at one endpoint and the current on the edge, and use that to find the potential at the other endpoint.  The orange edges are \emph{parallel} to the flow of information, but the blue edges are \emph{transverse} to it.

The harmonic continuation process is broken into two stages:  In the first stage, we are concerned about \emph{uniqueness}.  We do not have to worry about consistency since we just want $u$ to be zero everywhere.  In the second stage, we are concerned about \emph{existence}.  We do not care what the values of $u$ are away from the spike $e$ so long as there is some consistent extension.  The boundary spike comes in the \emph{middle} between the two stages.  The value on the interior endpoint had to be uniquely determined, but we also needed to have $u = 1$ at the boundary endpoint, and needed to have a consistent harmonic extension even after we put nonzero data on the network.

In the second stage, there can be some obstacles to uniqueness of extensions.  In the example, these obstacles were represented by the interior vertices where we assigned an arbitrary parameter rather than deducing the potential from previous information.  These were precisely the interior vertices with no orange edge \emph{entering} them.  When there is an orange edge entering a vertex $p$, we can use the values of $u$ already defined to determine the current on the orange edge and hence $u(p)$.  But if there is no orange edge entering $p$, we have one parameter of freedom in choosing $u(p)$.

In the first stage, obstacles to existence are permissible; since we want $u$ to be zero in this region, there is no problem achieving harmonicity.  The obstacles to existence are interior vertices with no orange edge \emph{exiting} them.  (These did not occur in the example, and if they did we would have barely noticed them, since we were focused on forcing $u$ to be zero.)  When there is an exiting orange edge, then the current on the edge is chosen so as to make the net current at the starting vertex $0$.  But if there is no exiting orange edge, there is nowhere for the current to escape to.

We are going to use the set of oriented orange edges as our model for the flow of information.  What requirements did the orange edges have to satisfy?  First, each interior vertex can have at most one orange edge entering it and at most one orange edge exiting it.  This implies that the oriented edges form disjoint paths.

The choice of orange edges must also be consistent with the \emph{order} of harmonic continuation.  Note that each orange edge was used \emph{after} the other edges incident to its starting point and \emph{before} the edges incident to its ending point.  To capture this idea of order directly from the properties of the orange edges, we define an \emph{increasing path} to be a path that uses only oriented orange edges and blue edges, with no two blue edges in a row.  These are the paths that are forced to be increasing with respect to the order of harmonic continuation.

In order for such increasing paths to reflect an underlying order of the edges, we need to require that \emph{no increasing path forms a cycle}.  Another requirement is that any obstacle to existence must come \emph{after} any obstacle to uniqueness.  If $S$ is the set of orange edges, then the obstacles to existence are interior vertices not in $S_-$ and the obstacles to uniqueness are interior vertices not in $S_+$.  Thus, we make the requirement, that \emph{there is no increasing path from a vertex in $V^\circ \setminus S_+$ to a vertex in $V^\circ \setminus S_-$}.  Finally, for our harmonic continuation process to work for infinite graphs, we anticipate some use of Zorn's lemma, and we require that \emph{there is no infinite decreasing path}.

If a set $S$ of oriented edges with $S \cap \overline{S} = \varnothing$ satisfies these conditions, we will call it a {\bf scaffold}.  For a scaffold $S$, we can partition the vertices and edges into three sets, corresponding to the beginning, middle, and end of the harmonic continuation process:
\begin{itemize}
	\item The {\bf Beginning} consists of anything that can be reached by a decreasing path from an interior vertex not in $S_-$.
	\item The {\bf End} consists of anything that can be reached by an increasing path from an interior vertex not in $S_+$.
	\item The {\bf Middle} consists of everything else.
\end{itemize}
The ``first stage'' takes place in the Beginning and part of the Middle, and the ``second stage'' takes place in part of the Middle and the End.

Using essentially the same argument as in the example, we will show that we can recover a boundary spike $e$ if there is a scaffold $S$, where $e \not \in S \cup \overline{S}$ and $e$ is in the Middle.  We can recover a boundary edge $e$ if there is a scaffold $S$, where $e \in S \cup \overline{S}$ and $e$ is in the Middle.

We say a $\partial$-graph is {\bf recoverable by scaffolds} if there is a sequence of boundary spike contractions and boundary edge deletions that exhausts the edges in the graph, and at each step, the edges removed can be recovered using a scaffold.  We will show that the inverse problem can always be solved for such $\partial$-graphs.

Recoverability by scaffolds has the virtue of being a purely geometric condition | we no longer have to pretend to do algebra while performing harmonic continuation.  However, recoverability by scaffolds is hard to check because it is inductive; it requires choosing the sequence of layer-stripping operations and constructing a scaffold at each stage.  But we will establish some easier-to-check sufficient conditions.

The main advantage of defining recoverability by scaffolds is that we can use a scaffolds on one graph to produce scaffolds on other graphs.  We will define an \emph{unramified harmonic morphism} of $\partial$-graphs later, and show that if $f: G \to H$ is a UHM and $S$ is a scaffold on $H$, then $f^{-1}(S)$ is a scaffold on $G$.  The reason for this is basically that increasing paths push forward to increasing paths.

Moreover, layer-stripping operations also produce layer-stripping operations by taking preimages.  This will enable us to show that if $f: G \to H$ is a UHM and $H$ is recoverable by scaffolds, then so is $G$.  Thus, when one checks that $H$ is recoverable by scaffolds, that automatically shows that a host of other graphs are also recoverable by scaffolds.

\subsection{Application to Mixed-Data Boundary Value Problems} \label{subsec:mixeddata}

Our recovery strategy used harmonic continuation to solve certain mixed-data boundary value problems.  Motivated by results of \cite{CIM}, we apply harmonic continuation to understand existence and uniqueness question for mixed-data boundary value problems in general.  Partition $\partial V$ into two sets $P$ and $Q$, and consider the following questions:
\begin{itemize}
	\item For which $(\phi,\psi) \in \mathbb{F}^P \times \mathbb{F}^P$ does there exist a harmonic function with $u|_P = \phi$ and $\Delta u|_P = \psi$?
	\item If there is such a harmonic function, how uniquely do the values on $P$ determine the values on $Q$?
	\item The same questions with $P$ and $Q$ reversed.
\end{itemize}
If existence and uniqueness occur for the first question, then we have a well-defined map $\mathbb{F}^P \times \mathbb{F}^P \to \mathbb{F}^Q \times \mathbb{F}^Q$.  But in general, we only have a linear relation
\[
X: \mathbb{F}^P \times \mathbb{F}^P \rightsquigarrow \mathbb{F}^Q \times \mathbb{F}^Q,
\]
that is, a linear subspace
\[
X \subset (\mathbb{F}^P \times \mathbb{F}^P) \times (\mathbb{F}^Q \times \mathbb{F}^Q)
\]
describing what boundary data on $P$ is compatible with what boundary data on $Q$.  To avoid clumsy notation, we write $x$ instead of $(\phi,\psi)$ for an element of $\F^P \times \F^P$.  Then we define $X$ by saying that $(x, y) \in X$ if and only if there is a harmonic function $u$ with boundary data $x$ on $P$ and $y$ on $Q$.

If $\pi_P$ and $\pi_Q$ are the projections of $(\mathbb{F}^P \times \mathbb{F}^P) \times (\mathbb{F}^Q \times \mathbb{F}^Q)$ onto the first and second factors, then the space of valid boundary data on $P$ for which the mixed-data problem has a solution is given by $V_P = \pi_P(X)$, and likewise for $Q$ it is $V_Q = \pi_Q(X)$.  The failure of uniqueness for our problem is described by $Z_Q$, the subspace of $V_Q$ consisting of all $y$ which are compatible with $0 \in \mathbb{F}^P \times \mathbb{F}^P$.  We can similarly define $Z_P = \pi_P(\pi_Q^{-1}(0))$.

There is a linear bijection $V_P / Z_P \to V_Q / Z_Q$ and hence $\dim V_P - \dim Z_P = \dim V_Q - \dim Z_Q$.  We will call this number $\rank X$ and think of it as the ``amount of algebraic connection'' between data on $P$ and data on $Q$.

Suppose that there is a scaffold where the vertices of $P$ are the ``inputs'' and the vertices of $Q$ are the ``outputs'' as in Figure \ref{fig:scaffactorization}.  Then we can use harmonic continuation to find the dimensions of the fundamental subspaces $V_P$, $Z_P$, $V_Q$, and $Z_Q$.  Here is an intuitive description of the process (which on the surface is rather different than the formal proof we will give later).

We start with the boundary data on $P$ and harmonically continue, using the edges in the scaffold in order.  As before, there are two stages.  In the first stage, there are some obstacles to existence of harmonic extensions when vertices in $P \cup V^\circ$ are not the input of some orange edge.  Say there are $k$ of them.  When we are forced to determine the net current on such vertices, we may reach an inconsistency, which forces us to throw out some choices of initial data on $P$.  Equivalently, each obstacle imposes a linear relationship that $x \in V_P$ must satisfy.

For the other choices of initial data that survive the first stage, we keep going.  In the second stage, there may be some obstacles to uniqueness when vertices in $V^\circ \cup Q$ are not the output of some orange edge, at which we must assign an arbitrary parameter.  Say there are $\ell$ of them.

Once we have gone through all the edges, we have eliminated all the invalid data on $P$ and parametrized the data on $Q$ that is compatible with each element of $V_P$.  We see that $\dim V_P = 2|P| - k$ since one dimension was eliminated by each obstacle.  Moreover, for each element of $V_P$, the compatible elements of $V_Q$ form an affine subspace of dimension $\ell$, since that is how many arbitrary parameters we used.  Thus, $\dim Z_Q = \ell$.

If we reverse the edges in the scaffold, we can say the same thing switching $P$ and $Q$ and switching $k$ and $\ell$.  In particular,
\[
\rank X = (2|P| - k) - k = (2|Q| - \ell) - \ell = 2(|P| - k) = 2(|Q| - \ell).
\]

This number has a simple geometric meaning as well.  There are $|P|$ paths of orange edges starting at $P$, and $k$ of them end at a vertex in $V^\circ \cup P$, so $|P| - k$ of them make it to $Q$.  Symmetrically, there are $|Q|$ paths going backwards from $Q$ and $|Q| - \ell$ of them reach $P$.  Thus, $|P| - k = |Q| - \ell$ is the size of the connection through the graph from $P$ to $Q$.  Thus, we have

~

{\bf Rank-connection principle:} \emph{The amount of algebraic connection between $V_P$ and $V_Q$ (that is, $\rank X$) is equal to twice the size of connection through the graph between $P$ and $Q$.}

~

The rank-connection principle (or an equivalent formulation using the response matrix) was observed in \cite{CIM} for circular planar networks as a consequence of the determinant-connection formula for determinants of submatrices of the Laplacian.  In a similar spirit, one can deduce from the grove-determinant formula (\cite{RF}) that the rank-connection principle holds for any $\partial$-graphs for \emph{generic} edge weights.  However, provided we can find a scaffold and do harmonic continuation, the rank-connection principle holds for \emph{all} edge-weights.

In particular, if there is a scaffold where all the orange paths connect $P$ and $Q$ (hence a full-size connection between $P$ and $Q$), then the relation $X$ actually defines a bijective function $\mathbb{F}^P \times \mathbb{F}^P \to \mathbb{F}^Q \times \mathbb{F}^Q$ for all edge weights.  Amazingly, the converse is also true:  If $X$ is a bijective function for all nonzero edge weights in $\R$, then such a scaffold exists (Theorem \ref{thm:uniquefull}).  Scaffolds thus provide a geometric characterization of situations when existence and uniqueness occur for all edge weights.

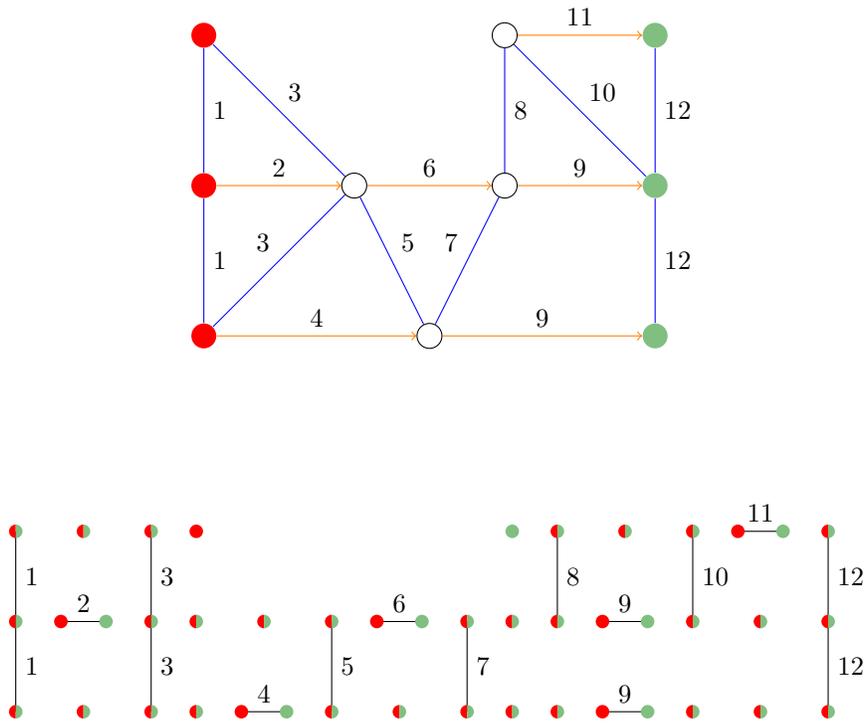
\begin{figure}

\caption{A scaffold and a corresponding elementary factorization.  The numbers on the edges show one possible partial order compatible with the scaffold.  In each IO-network the inputs are red and the outputs are green.  The green vertices of each IO-network are joined to the red vertices of the next one.}  \label{fig:scaffactorization}

\begin{center}
\begin{tikzpicture}
	\node[circle,fill,red] (1) at (0,4) {};
	\node[circle,fill,red] (2) at (0,2) {};
	\node[circle,fill,red] (3) at (0,0) {};
	\node[circle,draw] (4) at (2,2) {};
	\node[circle,draw] (5) at (3,0) {};
	\node[circle,draw] (6) at (4,4) {};
	\node[circle,draw] (7) at (4,2) {};
	\node[circle,fill,green!50!black!50] (8) at (6,4) {};
	\node[circle,fill,green!50!black!50] (9) at (6,2) {};
	\node[circle,fill,green!50!black!50] (10) at (6,0) {};
	
	\draw[blue] (1) to node[auto,black] {$1$} (2);
	\draw[blue] (2) to node[auto,black] {$1$} (3);
	\draw[->,orange] (2) to node[auto,black] {$2$} (4);
	\draw[blue] (1) to node[auto,black] {$3$} (4);
	\draw[blue] (3) to node[auto,black] {$3$} (4);
	\draw[->,orange] (3) to node[auto,black] {$4$} (5);
	\draw[blue] (4) to node[auto,black] {$5$} (5);
	\draw[->,orange] (4) to node[auto,black] {$6$} (7);
	\draw[blue] (5) to node[auto,black] {$7$} (7);
	\draw[blue] (6) to node[auto,black] {$8$} (7);
	\draw[->,orange] (7) to node[auto,black] {$9$} (9);
	\draw[->,orange] (5) to node[auto,black] {$9$} (10);
	\draw[blue] (6) to node[auto,black] {$10$} (9);
	\draw[->,orange] (6) to node[auto,black] {$11$} (8);
	\draw[blue] (8) to node[auto,black] {$12$} (9);
	\draw[blue] (9) to node[auto,black] {$12$} (10);
	
	\begin{scope}[shift = {(-2.5,-5)}, scale = 0.6]
		\draw (0,4) to node[auto,black] {$1$} (0,2) to node[auto,black] {$1$} (0,0);
		\rgvertex{0}{4}
		\rgvertex{0}{2}
		\rgvertex{0}{0}
		
		\draw (1,2) to node[auto] {$2$} (2,2);
		\rgvertex{1.5}{4}
		\rvertex{1}{2}
		\gvertex{2}{2}
		\rgvertex{1.5}{0}
		
		\draw (3,4) to node[auto] {$3$} (3,2) to node[auto] {$3$} (3,0);
		\rgvertex{3}{4}
		\rgvertex{3}{2}
		\rgvertex{3}{0}
		
		\rvertex{4}{4}
		\rgvertex{4}{2}
		\rgvertex{4}{0}
		
		\draw (5,0) to node[auto] {$4$} (6,0);
		\rgvertex{5.5}{2}
		\rvertex{5}{0}
		\gvertex{6}{0}
	
		\draw (7,2) to node[auto,black] {$5$} (7,0);
		\rgvertex{7}{2}
		\rgvertex{7}{0}
	
		\draw (8,2) to node[auto] {$6$} (9,2);
		\rvertex{8}{2}
		\gvertex{9}{2}
		\rgvertex{8.5}{0}
	
		\draw (10,2) to node[auto] {$7$} (10,0);
		\rgvertex{10}{2}
		\rgvertex{10}{0}
		
		\gvertex{11}{4}
		\rgvertex{11}{2}
		\rgvertex{11}{0}
		
		\draw (12,4) to node[auto] {$8$} (12,2);
		\rgvertex{12}{4}
		\rgvertex{12}{2}
		\rgvertex{12}{0}
		
		\draw (13,2) to node[auto] {$9$} (14,2);
		\draw (13,0) to node[auto] {$9$} (14,0);
		\rgvertex{13.5}{4}
		\rvertex{13}{2}
		\gvertex{14}{2}
		\rvertex{13}{0}
		\gvertex{14}{0}
	
		\draw (15,4) to node[auto] {$10$} (15,2);
		\rgvertex{15}{4}
		\rgvertex{15}{2}
		\rgvertex{15}{0}
	
		\draw (16,4) to node[auto] {$11$} (17,4);
		\rvertex{16}{4}
		\gvertex{17}{4}
		\rgvertex{16.5}{2}
		\rgvertex{16.5}{0}
	
		\draw (18,4) to node[auto] {$12$} (18,2) to node[auto] {$12$} (18,0);
		\rgvertex{18}{4}
		\rgvertex{18}{2}
		\rgvertex{18}{0}
	
	\end{scope}
\end{tikzpicture}
\end{center}

\end{figure}

\subsection{Another Viewpoint: Gluing Networks as Composition} \label{subsec:gluing}

Scaffolds are flexible, functorial, and adaptable to the infinite situation, but these virtues arose from \emph{discarding certain information} about the harmonic continuation process:  The scaffold does not specify what boundary vertices were the overall ``inputs'' or ``outputs.''  Nor does it specify in what exact order the edges were used during harmonic continuation; rather, it gives us a \emph{partial} order expressing many different options for ordering the edges.  Thus, in order to establish the rank-connection principle, it will be convenient to have another geometric object which describes the details of the harmonic continuation process more explicitly.

We can view the linear relation $X$ described in the last section as a transformation or morphism from $(\F^P)^2$ to $(\F^Q)^2$ in the category of linear relations.  This is the category where the objects are vector spaces and a morphism $T: V \rightsquigarrow W$ is a subspace $T \subset V \times W$.  The composition of $S: U \rightsquigarrow V$ and $T: V \rightsquigarrow W$ is defined by
\[
(x,z) \in T \circ S \Leftrightarrow \exists y \in V \text{ such that } (x,y) \in S \text{ and } (y,z) \in T.
\]
In the case where the relations are bona fide functions, this reduces to composition of functions.

We will factorize our relation $X: (\F^P)^2 \to (\F^Q)^2$ as a composition of very simple relations corresponding to the individual steps in the harmonic continuation process.  But this algebraic factorization will be modeled by a factorization in a more geometric category.

Using the language of Baez-Fong \cite{BF}, we define a category where a $\partial$-graph with $\partial V = P \cup Q$ is viewed as a morphism from $P$ to $Q$.  As will become clear later, it is useful to allow $P$ and $Q$ to overlap, and to view $P$ and $Q$ as \emph{labels} on the vertices rather than the vertices themselves.  We thus define the category of input-output graphs (or IO-graphs) as follows: The objects are finite sets.  A morphism $\mathcal{G}: P \to Q$ is a graph $G$ together with labelling functions $i: P \to V$ and $j: Q \to V$, which we assume to be injective.\footnote{Technically, a morphism is an equivalence class of graphs where two graphs are considered the same if they are isomorphic by a graph isomorphism that commutes with the labelling functions.}    It is unnecessary to specify $\partial V$ for $G$ since we can simply call it $i(P) \cup j(Q)$.

Two morphisms $\mathcal{G}: P \to Q$ and $\mathcal{G}': Q \to R$ are composed by ``gluing the graphs together along $Q$'':\ taking the disjoint union of $G$ and $G'$ and then identifying the vertices labelled by $Q$ in $G$ with the vertices labelled by $Q$ in $G'$.  In other words, we glue the outputs of the first morphism to the inputs of the second morphism.  The category of IO-networks is defined the same way, but with the extra information of edge weights.

The relation $X$ (properly defined) is a functor from the category of IO-networks to the category of linear relations (called the ``black box functor'' in \cite{BF}).  Suppose $\Gamma: P \to Q$ and $\Gamma': Q \to R$ are IO-network morphisms where the sets of vertices labelled by $P$, $Q$, and $R$ are disjoint.  Suppose $(x,y) \in X(\Gamma)$ and $(y,z) \in X(\Gamma')$ represent the boundary data of harmonic functions $u$ and $u'$.  In $\Gamma' \circ \Gamma$, the vertices labelled by $Q$ are interior.  Then $u$ and $u'$ will glue together to a harmonic function on $\Gamma' \circ \Gamma$ if and only if their potentials agree on these vertices and their net currents cancel.

Thus, to make composition in the category of relations work, we change our sign convention.  We say $(x,y) \in X(\Gamma)$ if $x = (u|_P, -\Delta u|_P)$ and $y = (u|_Q, \Delta u|_Q)$.  In other words, $x$ tells us the current flowing \emph{out of} the network at $P$ and $y$ tells us the current flowing \emph{into} the network at $Q$.\footnote{We put the minus sign on $P$ rather than $Q$ to make the formulas in \S \ref{subsec:elementarymorphisms} and \S \ref{subsec:IOlayerstripping} neater.} In the new convention, if $(x,y) \in X(\Gamma)$ and $(y,z) \in X(\Gamma')$, then the harmonic functions $u$ and $u'$ will glue together to a harmonic function on $\Gamma \circ \Gamma'$, and conversely, any $(x,z) \in X(\Gamma' \circ \Gamma)$ arises this way.  We thus have $X(\Gamma' \circ \Gamma) = X(\Gamma') \circ X(\Gamma)$.

The only question is how to define $X(\Gamma)$ when the vertices $i(P)$ and $j(Q)$ overlap.  If $p \in i(P) \cap j(Q)$, then we can think of current is flowing into $p$ on the $P$ side and out of $p$ on the $Q$ side.  Thus for $((x_1,x_2), (y_1,y_2))$ to be in $X(\Gamma)$ we want the associated harmonic function $u$ to satisfy $x_1(p) = y_1(p) = u(p)$ and $y_2(p) - x_2(p) = \Delta u(p)$.  This convention makes $X$ a functor in the general case.

If we have a scaffold on $\Gamma$ modeling harmonic continuation from $P$ to $Q$, then we obtain an {\bf elementary factorization} of $\Gamma: P \to Q$ as $\Gamma_n \circ \dots \circ \Gamma_1$ as shown in Figure \ref{fig:scaffactorization}.  Here each IO-graph is pictured with the input vertices red and the output vertices green.  The elementary IO-networks come in four types: Type 1 corresponds to an orange edge, type 2 corresponds to a blue edge, type 3 corresponds to an obstacle to existence, and type 4 corresponds to an obstacle to uniqueness.

Given an elementary factorization, it is easy to prove the rank-connection principle:  All we have to do is compute $X$ on the elementary IO-networks $\Gamma_j$ and figure out what happens when we compose them.  The details are carried out in \S \ref{sec:IO}.  We will also describe how to obtain elementary factorizations from scaffolds and vice versa (\S \ref{subsec:IOscaffolds}).

\subsection{Organization}

We have now given a rough description of the ideas that form the backbone of the paper.  The rest of the paper works out the technicalities of these constructions, relates them to other things, and applies them to prove old and new results.

\S \ref{sec:operations} lays out the definitions and basic properties of various $\partial$-graph constructions.  We describe harmonic morphisms, and in particular, the unramified harmonic morphisms (UHMs) which we will use to pull back scaffolds.  Next, we define various operations with harmonic subgraphs play a central (though sometimes hidden) role in the paper and in network theory in general.  We also define layer-stripping operations and show that they pull back under UHMs.

\S \ref{sec:scaffolds} defines \emph{scaffolds} and \emph{recoverability by scaffolds}.  We show that recoverability by scaffolds is a sufficient condition for solving the inverse problem, and that it pulls back under UHMs.  We also define a stronger but more symmetrical condition called \emph{total layerability}.

\S \ref{sec:IO} describes the category of IO-networks (as in \cite{BF}) and elementary factorizations in this category.  We show that the rank-connection principle holds for any network with such an elementary factorization (\S \ref{subsec:elementary}).  We relate elementary factorizations, layer-stripping, and scaffolds.  The IO-graph category provides a more concrete motivation and definition for (a variant of) the electrical linear group from \cite{LP}.

\S \ref{sec:surfaces} describes how to construct scaffolds and elementary factorizations for $\partial$-graphs embedded on surfaces.   We give a strategy that is potentially applicable to general surfaces and execute it for the critical circular planar $\partial$-graphs studied in \cite{CIM}, \cite{dVGV}, \cite{WJ}.  We show critical circular planar $\partial$-graphs are totally layerable and that elementary factorizations exist for any circular pair.  We also prove recoverability of the supercritical half-planar graphs of \cite{IZ}.

\S \ref{sec:characterization} gives a geometric characterization of the situations when the rank-connection principle holds universally for all edge weights, using a generalization of elementary factorizations.  Along the way, we prove a max-flow min-cut principle for connections.

\S \ref{sec:symplectic} characterizes the possible boundary behaviors of electrical networks using symplectic vector spaces, drawing on \cite{BF} and \cite{LP}.  We show that the boundary behavior of a network is a Lagrangian subspace of $\F^{\partial V} \times \F^{\partial V}$ containing $(1,\dots,1,0,\dots,0)$ and conversely, any such subspace can be realized as the boundary behavior of a network (which we explicitly construct).  Similarly, we show that the electrical linear group is the group of symplectic matrices that preserve $(1,\dots,1,0,\dots,0)$ and explicitly construct networks for each matrix.  Together with local network equivalences, this can be used to show that any network is equivalent to a circular planar network (if we don't require the edge weights to be positive).  This result holds for any field other than $\F_2$.

\S \ref{sec:concluding} generalizes much of the theory to nonlinear networks, and suggests further generalizations and open problems.

\S \ref{sec:operations} through \S \ref{sec:IO} build on each other and are prerequisites for the later sections.  However, \S \ref{sec:surfaces} through \S \ref{sec:symplectic} are independent of each other and can be read in any order.  \S \ref{sec:concluding} comments on the results of all the preceding sections.

\section{Operations with $\partial$-Graphs and Networks} \label{sec:operations}

\subsection{Harmonic Morphisms} \label{subsec:harmonicmorphisms}

The correct notion of graph morphism for our theory is neither a continuous map of graphs viewed as topological spaces, nor a graph homomorphism.  Since we are interested in harmonic functions, we need a type of graph morphism that preserves the Laplacian.  Adapting the construction of Urakawa \cite{urakawa}, we will define harmonic morphisms of $\partial$-graphs and networks.  Harmonic morphisms include many standard classes of maps between graphs, and they are loosely analogous to analytic functions between Riemann surfaces in the sense that if $f: \Gamma_1 \to \Gamma_2$ is a harmonic morphism and $u$ is harmonic on $\Gamma_2$, then $u \circ f$ is harmonic on $\Gamma_1$.

A {\bf harmonic morphism of $\partial$-graphs} $f: G_1 \to G_2$ is a map ${f: V_1 \sqcup E_1 \to V_2 \sqcup E_2}$ such that
\begin{enumerate}
	\item $f$ maps vertices to vertices.
	\item If $f(e)$ is an oriented edge, then $f(e_{\pm}) = f(e)_{\pm}$ and $f(\overline{e}) = \overline{f(e)}$.
	\item If $f(e)$ is a vertex, then $f(\overline{e}) = f(e)$ and $f(e_{\pm})= f(e)$.
	\item $f$ maps interior vertices to interior vertices.
	\item For any $p \in V_1^\circ$, the restricted map
	\[
	\{e\in E_1\colon e_+=p\text{ and $f(e)$ is an edge}\} \to \{e \in E_2: e_+ = f(p)\}
	\]
	has constant fiber size.  In other words, it is $n$-to-$1$ for some $n \geq 0$ (which may depend on $p$).
\end{enumerate}
A {\bf harmonic morphism of $\mathbb{F}$-networks} is given by a harmonic morphism of the underlying $\partial$-graphs which preserves the edge weights in the sense that $w(f(e)) = w(e)$ whenever $f(e)$ is an edge.

The reader may verify that using harmonic morphisms, $\partial$-graphs and $\mathbb{F}$-networks form categories.

In brief, (1), (2), and (3) state that $f$ preserves reverse orientations and endpoints of edges.  Unlike a graph homomorphism (see \cite{GodsilRoyle}), $f$ is allowed to ``collapse'' an edge into a vertex, and in that case the endpoints of the edge must map to the same vertex.  (1), (2), (3) imply that in the language of topology, $f$ is a continuous cellular map.

Condition (5) says that if we ignore collapsed edges, then $f$ maps the ``star'' $\{e \in E_1: e_+ = p\}$ in an $n$-to-$1$ way onto the ``star'' $\{e \in E_2: e_+ = f(p)\}$ whenever $p$ is an interior vertex.  This implies that if $u: V_2 \to \mathbb{F}$, then
\[
\Delta(u \circ f)(p) = \sum_{e: e_+ = p} w(e) d(u \circ f)(e) = n \sum_{e: e_+ = f(p)} w(e) du(e) = n \Delta u(p),
\]
where for the middle equality we use the fact that $d(f \circ u) = 0$ on collapsed edges, and $f$ preserves the edge weights.  Together with (4), this implies
\begin{lemma}
If $f: \Gamma_1 \to \Gamma_2$ is a harmonic morphism and $u$ is harmonic on $\Gamma_2$, then $u \circ f$ is harmonic on $\Gamma_1$.  In other words, $\Gamma \mapsto \mathcal{U}(\Gamma)$ is a contravariant functor from the category of $\mathbb{F}$-networks to the category of $\mathbb{F}$-vector spaces.
\end{lemma}

\begin{remark*}
In fact, given conditions (1) - (4) and given any fixed edge weights, condition (5) is equivalent to saying that $f$ ``locally preserves harmonicity'' in the sense that $\Delta u(p) = 0$ implies $\Delta(u \circ f)(p) = 0$ for all $p \in V^\circ$ and any function $u: V \to \mathbb{F}$.
\end{remark*}

Harmonic morphisms include several standard types of maps.  First, we say $f: \tilde{G} \to G$ is a {\bf covering map} if it satisfies the following conditions: $f$ maps vertices to vertices and oriented edges to oriented edges, $f$ preserves the reverses and endpoints of edges, $f$ is surjective for both vertices and edges, $f(p) \in V^\circ$ if and only if $p \in \tilde{V}^\circ$, and $f$ maps $\{e \in \tilde{E}: e_+ = p\}$ bijectively onto $\{e \in E: e_+ = f(p)\}$ for all $p \in V$.  (This agrees with the topological definition of covering map in the special case of graphs.)  Any covering map is a ``local isomorphism.''

A {\bf covering $\partial$-graph} of $G$ is a $\partial$-graph $\tilde{G}$ together with a covering map $\tilde{G} \to G$.  Concretely, covering $\partial$-graphs are constructed as follows:  Take a finite or countable set $S$, and define
\[
\tilde{V} = V \times S, \qquad \tilde{E} = E \times S, \qquad \tilde{V}^\circ = V^\circ \times S.
\]
Each oriented edge $(e,s)$ has the natural ending point $(e,s)_+ = (e_+,s)$.  However, we will ``mix up'' the reverses and starting points among the edges in each fiber $f^{-1}(e)$.  For each $e \in E$, choose a permutation $\sigma_e \in \Perm(S)$ such that $\sigma_{\overline{e}} = \sigma_e^{-1}$, and then set
\[
\overline{(e,s)} = (\overline{e}, \sigma_e(s)),
\]
so that $(e,s)_- = (e_-, \sigma_e(s))$.  In other words, we take several copies of $G$, then cut the edges in half and glue them together in a different arrangement in each fiber. The covering map $\tilde{G} \to G$ is given by the projections $V \times S \to V$ and $E \times S \to E$.  Up to isomorphism, all covering $\partial$-graphs are constructed this way, assuming $G$ is connected.

{\bf Branched covering maps} are like covering maps except that they allow \emph{ramification}: $f$ is not required to map $\{e \in \tilde{E}: e_+ = p\}$ bijectively onto $\{e \in E: e_+ = f(p)\}$.  Instead, this mapping must be $n$-to-$1$ for some $n > 0$ which may depend on $p$.  This can be viewed as a discrete analogue of the behavior of the analytic function $z^n$ in a neighborhood of the origin, in keeping with the analogy between graphs and Riemann-surfaces in the literature \cite{urakawa}, \cite{bakernorine}, \cite{BobenkoGunther}.  Concretely, some branched covering spaces can be obtained from covering spaces by gluing some vertices in each fiber $f^{-1}(p)$ together.

{\bf Box products} furnish another class of harmonic morphisms.  Given two $\partial$-graphs $G_1$ and $G_2$, we define $G_1 \Box G_2$ as the graph $G$ with
\[
V = V_1 \times V_2, \qquad E = E_1 \times V_2 \sqcup V_1 \times E_2, \qquad V^\circ = V_1^\circ \times V_2^\circ.
\]
The reverses and endpoints of edges are defined by
\[
\overline{(e,p)} = (\overline{e},p), \quad (e,p)_+ = (e_+,p), \quad (e,p)_- = (e_-,p) \text{ for } (e,p) \in E_1 \times V_2
\]
and a symmetrical formula for $(p,e) \in V_1 \times E_2$.  The ``natural'' projection map $G_1 \Box G_2 \to G_1$ is a harmonic morphism.  It collapses all the edges in $V_1 \times E_2$ into vertices.

Similar to our construction of covering spaces, we can define a {\bf twisted box-product} by choosing a permutation of $V_1$ for each edge $e \in E_1$ with $\sigma_{\overline{e}} = \sigma_e^{-1}$ and defining
\[
(e,p)_+ = (e_+,p), \qquad \overline{(e,p)} = (\overline{e}, \sigma(e) p).
\]
When we twist in the first factor but leave the edges in $V_1 \times E_2$ untwisted, then the map $G \to G_1$ is a harmonic morphism, although the map $G \to G_2$ is not since the two endpoints of the collapsed edges are not mapped to the same vertex.

Finally, harmonic morphisms include the inclusion maps of {\bf harmonic subgraphs}.  We say $G'$ is a harmonic subgraph of $G$ if
\[
V' \subset V, \quad E' \subset E, \quad (V')^\circ \subset V^\circ,
\]
and for each $p \in (V')^\circ$, the star $\{e \in E: e_+ = p\}$ is contained in $E'$.  Note that a $\partial$-graph $G'$ is a harmonic subgraph of $G$ if and only if it a subgraph and the inclusion map is a harmonic morphism, provided we assume $G'$ has no isolated interior vertices.  The inclusion maps of harmonic subgraphs are characterized as harmonic morphisms which are globally injective and such that $\{e: e_+ = p\} \to \{e: e_+ = f(p)\}$ is bijective for each interior vertex $p$.

Roughly speaking, a harmonic morphism is locally some mixture of a branched covering map, the projection of a twisted box product, and the inclusion of a harmonic subgraph.

In order to pull back scaffolds and layer-stripping operations, we will have to exclude ramification.  We define an {\bf unramified harmonic morphism (UHM)} as a harmonic morphism such that $\{e: e_+ = p\} \to \{e: e_+ = f(p)\}$ is \emph{bijective} for each interior vertex $p$ and \emph{injective} for each boundary vertex $p$.  Inuitively, a scaffold is something like a foliation of a Riemann surface.  If a map between Riemann surfaces is locally injective, then a foliation can be pulled back by taking preimages, but the preimage of a foliation will not be a foliation near branching points.

Note that $\partial$-graphs and UHMs form a category.  Moreover, covering maps, the projections of twisted box products, and the inclusions of harmonic subgraphs are UHMs.  

\subsection{Operations with Subgraphs} \label{subsec:subgraphs}

Harmonic subgraphs are used implicitly or explicitly in most papers about electrical networks.  In particular, layer-stripping operations produce harmonic subgraphs (see \S \ref{subsec:layerstrippingops}).  Harmonic continuation proceeds by extending harmonic functions defined on harmonic subnetworks.  Local network equivalences such as $Y$-$\Delta$ transformations (see \cite{JR}) work by replacing one harmonic subnetwork with a different harmonic subnetwork with the same boundary behavior.  Thus, it is well worth our while to develop the language and basic properties of harmonic subgraphs and subnetworks.

If $f: G_1 \to G_2$ is a harmonic morphism and $H$ is a harmonic subgraph of $G_2$, then we can define the {\bf pullback $f^{-1}(H)$} as the harmonic subgraph of $G_1$ whose vertex and edge sets are the preimages of the vertex and edge sets of $H$, and whose interior vertices are $f^{-1}(V^\circ(H)) \cap V^\circ(G_1)$.  \textbf{Intersections and unions of harmonic subgraphs} are defined by taking the intersections and unions of the respective sets $V$, $V^\circ$, and $E$.  For instance, $V^\circ(\bigcup_\alpha \Gamma_\alpha) = \bigcup_\alpha V^\circ(\Gamma_\alpha)$.  For a harmonic subgraph $H \subset G$, we define the {\bf complement} $G \setminus H$ by
\[
V(G \setminus H) = V(G) \setminus V^\circ(H), \qquad E(G \setminus H) = E(G) \setminus E(H), \qquad V^\circ(G \setminus H) = V^\circ(G) \setminus V(H).
\]
This does not satisfy all the properties of a set-theoretic complement, but it is the best we can do in a harmonic subgraph.  We define similar operations for subnetworks.

The way that harmonic subnetworks and boundary behavior interact in general is well-known and unsurprising.  Roughly speaking,
\begin{itemize}
	\item \emph{Gluing:} If we glue together a collection of networks along boundary vertices, then the boundary behavior of the larger network depends only on the boundary behaviors of the smaller ones (see e.g. \cite{card}).
	\item \emph{Splicing:} If $\Gamma'$ is obtained by replacing some part of $\Gamma$ by another part with the same boundary behavior, then $\Gamma$ and $\Gamma'$ have the same boundary behavior (see e.g. \cite{DI}).
	\item \emph{Recoverability:} A subnetwork of a recoverable network is recoverable (see e.g. \cite{card} Theorem 2.9, \cite{JRsub}).
\end{itemize}

%\begin{remark*}
%In the case of gluing linear resistor networks, there is an explicit formula for the response matrix of the larger network based on the smaller ones (\cite{card} \S 2.1).  But these principles can be derived purely from set theory.
%\end{remark*}

To make this precise, we define a \textbf{subgraph partition of} $G$ as a collection of harmonic subgraphs $\{G_\alpha\}$ such that
\begin{itemize}
	\item $V(G) = \bigcup_\alpha V(G_\alpha)$,
	\item $E(G)$ is the disjoint union of $E(G_\alpha)$.
	\item $V^\circ(G_\alpha)$ is disjoint from $V(G_\beta)$ for any $\alpha \neq \beta$.
\end{itemize}
Note that $\bigcup_\alpha G_\alpha$ is not $G$, but $G$ is obtained from $\bigcup_\alpha G_\alpha$ by changing some boundary vertices to interior vertices.  A {\bf subnetwork partition} is defined the same way except with the extra information of edge weights.

\begin{proposition}[Gluing] \label{prop:subnetworkgluing}
If $\{\Gamma_\alpha\}$ is subnetwork partition of $\Gamma$, then $\Lambda(\Gamma)$ can be computed from $\Lambda(\Gamma_\alpha)$ and the identifications between vertices in $\partial V(\Gamma_\alpha)$ and $\partial V(\Gamma_\beta)$ in the larger network.
\end{proposition}

\begin{proof}
Let $S = \bigcup_\alpha \partial V(\Gamma_\alpha)$.  Let $T \subset \prod_\alpha \Lambda(\Gamma_\alpha)$ be the set of points $((\phi_\alpha,\psi_\alpha))$ where
\begin{enumerate}[a.]
	\item If $p \in V(\Gamma_\alpha) \cap V(\Gamma_\beta)$, then $\phi_\alpha(p) = \phi_\beta(p)$.
	\item If $p \in S \cap V^\circ(\Gamma)$, then
	\[
	\sum_{\alpha: \alpha \in \partial V(\Gamma_\alpha)} \psi_\alpha(p) = 0.
	\]
	Since $p$ is an endpoint of only finitely many edges, and each edge is in only one subnetwork, the sum has only fintitely many nonzero terms.
\end{enumerate}
Define $F: T \to \mathbb{F}^{\partial V} \times \mathbb{F}^{\partial V}$ by $\prod_\alpha (\phi_\alpha, \psi_\alpha) \mapsto (\phi,\psi)$, where
\begin{enumerate}
	\item $\phi(p) = \phi_\alpha(p)$ for $p \in \partial V(\Gamma)$.
	\item $\psi(p) = \sum_{\alpha: p \in \partial V(G_\alpha)} \psi_\alpha(p)$,
\end{enumerate}
which is well-defined by definition of $T$.  Then $\Lambda(\Gamma) = F(T)$.  Indeed, if $((\phi_\alpha, \psi_\alpha)) \in T$ and $(\phi_\alpha,\psi_\alpha)$ is the boundary data of a harmonic function $(u_\alpha,c_\alpha)$, then (a) and (b) guarantee that they paste together to a harmonic function on $\Gamma$, and (1) and (2) describe how to find its boundary data.  Conversely, given any harmonic function on $\Gamma$, the restrictions to $\Gamma_\alpha$ will be harmonic and their boundary data will be in $T$.  Since we have described how to find $\Lambda(\Gamma)$ from $\Lambda(\Gamma_\alpha)$, we are done.
\end{proof}

\begin{corollary}[Splicing]
Suppose that $\{\Gamma_\alpha\}$ and $\{\Gamma_\alpha'\}$ are subnetwork partitions of $\Gamma$ and $\Gamma'$ respectively such that $\partial V(\Gamma_\alpha) = \partial V(\Gamma_\alpha)$, $\partial V(\Gamma) = \partial V(\Gamma')$.  If $\Lambda(\Gamma_\alpha) = \Lambda(\Gamma_\alpha')$ for all $\alpha$, then $\Lambda(\Gamma) = \Lambda(\Gamma')$.
\end{corollary}

\begin{corollary}[Recoverability]
If a $\partial$-graph $G$ is recoverable over $\mathbb{F}$, then so is any harmonic subgraph.
\end{corollary}

\begin{proof}
Let $S$ be a subgraph of $G$ and let $S' = G \setminus S$.  Then $S$ and $S'$ form a $\partial$-subgraph partition of $G$.  If $S$ is not recoverable, then there are two networks $\Sigma_1$ and $\Sigma_2$ on $S$ with different edge weights and the same boundary behavior.  Pick some network $\Sigma'$ on $S'$.  Then the networks on $G$ with edge weights given by $\Sigma_1 \cup \Sigma'$ and $\Sigma_2 \cup \Sigma'$ have the same boundary behavior but different edge weights, so $G$ is not recoverable.
\end{proof}

\subsection{Boundary Wedge-Sums}

There is one case of gluing networks together where the behavior of the smaller pieces is determined by the behavior of the whole.  We say $G$ is the \emph{boundary wedge-sum} of two harmonic subgraphs $G_1$ and $G_2$ if $G_1 \cup G_2 = G$ and $G_1 \cap G_2$ consists of a single boundary vertex $p$.  Then
\begin{lemma} \label{lem:wedgesumrecovery}
Suppose $\Gamma$ is the boundary wedge-sum of $\Gamma_1$ and $\Gamma_2$.  Assume either $\Gamma_1$ or $\Gamma_2$ is finite.  Then $\Lambda(\Gamma_1)$ and $\Lambda(\Gamma_2)$ are determined by $\Lambda(\Gamma)$.  In particular, a boundary wedge-sum of a recoverable $\partial$-graph and a finite recoverable $\partial$-graph must be recoverable.
\end{lemma}

\begin{proof}
Assume $\Gamma_1$ is finite.  For a potential on $\Gamma_1$, the net current at all vertices of a network must sum to zero, and in particular, if $u$ is harmonic, then $\sum_{p \in \partial V} \Delta u_1(p) = 0$.  Let $q$ be the vertex of $\Gamma_1 \cap \Gamma_2$.  If $u$ is a harmonic function on $\Gamma$, then the boundary data of $u_1 = u|_{\Gamma_1}$ is clearly determined by the boundary data of $u$ except possibly for $\Delta u_1(q)$.  But since the net currents sum to zero, $\Delta u_1(q)$ is also known.  But knowing the contribution from $\Gamma_1$ to $\Delta u(q)$, we also know the contribution from $\Gamma_2$, and hence the boundary data of $u_2 = u|_{\Gamma_2}$ is also known.

We thus have a maps $\Lambda(\Gamma) \to \Lambda(\Gamma_1)$ and $\Lambda(\Gamma) \to \Lambda(\Gamma_2)$ induced by restriction.  They are surjective because any harmonic function on $\Gamma_1$ or $\Gamma_2$ can be extended to $\Gamma$ by setting it to be constant on the other subnetwork.
\end{proof}

\subsection{Layer-Stripping Operations} \label{subsec:layerstrippingops}

The following definitions and results are inspired by \cite{CIM}, \S 10 and 11, and related results in \cite{CM}.

A {\bf boundary edge} is an edge $e$ with $e_+, e_- \in \partial V$.  The $\partial$-graph $G \setminus e$ obtained by {\bf deleting a boundary edge} $e$ (and leaving the sets $V$ and $V^\circ$ unchanged) is a harmonic subgraph of $G$.

A {\bf boundary spike} is an edge $e$ with endpoints $p \in \partial V$ and $q \in V^{\circ}$ such that $p$ has degree $1$.  We form the $\partial$-graph $G / e$ by {\bf contracting the boundary spike}, where $E(\Gamma / e) = E(\Gamma) \setminus \{e,\overline{e}\}$, and $V(\Gamma / e) = V(\Gamma) / \sim$, where $\sim$ is the equivalence relation given by $p \sim q$.  The vertex $\{p,q\}$ in $\Gamma / e$ is declared to be a boundary vertex.  We can (and will) identify $G / e$ with a harmonic subgraph of $G$ by mapping $\{p,q\}$ to $q$.

An {\bf isolated boundary vertex} is a boundary vertex $p$ of degree $0$, and $G \setminus p$ is the harmonic subgraph formed by deleting it.

We take the terms {\bf boundary spike contraction}, {\bf boundary edge deletion}, and {\bf isolated boundary vertex deletion} to allow multiple (even infinitely many) contractions or deletions and to include the trivial identity transformation (removing zero boundary spikes, boundary edges, or isolated boundary vertices).  For a contraction of multiple boundary spikes, we require that the spikes do not share any endpoints.  We refer to these transformations collectively as {\bf layer-stripping operations}.

An important fact is that layer-stripping operations pull back to sequences of layer-stripping operations:

\begin{lemma} \label{lem:layerstrippingfunctoriality}
If $f: G \to H$ is an unramified harmonic morphism, $H_1 \subset H_2 \subset H$ and $H_1$ is obtained from $H_2$ by a layer-stripping operation, then $f^{-1}(H_1)$ is obtained from $f^{-1}(H_2)$ by a sequence of layer-stripping operations.
\end{lemma}

\begin{proof}
Suppose that $H_2$ is obtained from $H_1$ by deleting boundary edges.  Then $f^{-1}(H_2)$ is obtained from $f^{-1}(H_1)$ by deleting boundary edges.

Suppose that $H_2$ is obtained from $H_1$ by deleting isolated boundary vertices.  The preimage of the isolated boundary vertices may contain some collapsed edges.  The collapsed edges are boundary edges in $f^{-1}(H_1)$, so we can delete them and then delete the now-isolated boundary vertices in the preimage of the isolated boundary vertices of $H_1$.

Suppose that $H_2$ is obtained from $H_1$ by contracting boundary spikes.  Note that some edges in $f^{-1}(H_1)$ may map to the boundary endpoints of the spikes in $H_1$.  In this case, they are boundary edges, so we can delete them.  Now suppose an edge $e$ maps to a boundary spike and $e_+$ corresponds to the boundary endpoint and $e_-$ corresponds to the interior endpoint.  There are two possibilities: If $e_-$ is interior, then $e$ is a boundary spike, so we can contract it.  If $e_-$ is boundary, then $e$ is a boundary edge, so we can delete it and then delete the isolated boundary vertex $e_+$.  Finally, there may be some isolated boundary vertices of $f^{-1}(H_1)$ that map to the boundary endpoint of the spikes in $H_1$, and we can delete them.  These are the only possibilities; thus, by a sequence of layer-stripping operations, we can obtain $f^{-1}(H_2)$ from $f^{-1}(H_1)$.
\end{proof}

\begin{remark*}
The result of Lemma \ref{lem:layerstrippingfunctoriality} does not quite show that sequences of layer-stripping operations pull back \emph{functorially} since a given transformation could be broken up into layer-stripping operations in multiple ways.  However, careful analysis of the last proof shows that a three-step operation of boundary edge deletion, isolated boundary vertex deletion, and boundary spike contraction will pull back to another operation of the same type.  Moreover, the decomposition into three steps is unique.  Thus, three-step layer-stripping operations pull back functorially.
\end{remark*}

\begin{remark*}
Lemma \ref{lem:layerstrippingfunctoriality} fails for harmonic morphisms in general: If $p$ is the boundary endpoint of a boundary spike $e$, then a vertex in $f^{-1}(p)$ might not have degree $1$ since there can be multiple preimages of $e$ attached to it.
\end{remark*}

We say a $\partial$-graph $G$ is {\bf layerable} if there is a filtration of subgraphs
\[
G = G_0 \supset G_1 \supset \dots
\]
indexed by $\N$ or $\{0,\dots,n\}$ such that $G_{j+1}$ is obtained from $G_j$ by layer-stripping operation and $\bigcap_{j=0}^\infty G_j = \varnothing$.  In this case, we say $\{G_j\}$ is a {\bf layerable filtration} of $G$.  The previous lemma immediately implies
\begin{lemma}
If $f: G \to H$ is a UHM and $H$ is layerable, then $G$ is layerable.
\end{lemma}

As noted in the introduction, if $\Gamma'$ is obtained from $\Gamma$ by a layer-stripping operation, then we can compute the boundary behavior of $\Gamma'$ from that of $\Gamma$ and vice versa.  Recall this is one of the steps in the layer-stripping strategy for the inverse problem described in \S \ref{subsec:strategy}.

\begin{lemma} \label{lem:reductionelectrical}
Suppose that $\Gamma'$ is obtained from $\Gamma$ by a sequence of layer-stripping operations.
\begin{itemize}
	\item Any harmonic function $u'$ on $\Gamma'$ extends to a harmonic function $u$ on $\Gamma$.
	\item The extension $u$ is uniquely determined by $u'$ and the values of $u$ on vertices which are deleted as isolated boundary vertices.
	\item Knowing the weights of the edges removed, we can compute $\Lambda(\Gamma)$ from $\Lambda(\Gamma')$ and vice versa.
\end{itemize}
\end{lemma}

\begin{proof}
To obtain $\Gamma$ from $\Gamma'$, we must add isolated boundary vertices, add boundary spikes, and add boundary edges.  When we add isolated boundary vertices, the potential on the new vertices can be chosen freely.  When we add boundary spikes, the potential on the boundary endpoints of the spikes must be chosen to make the net current at the interior vertices be zero.  When we add boundary edges, there are no new vertices and the Laplacian at the interior vertices is unchanged, so there is nothing to do.  This proves the first and second claims.  The third follows because at each step, the boundary data of the harmonic extension can be determined from the boundary data of the original function without knowing its values on the interior of $\Gamma'$, and similarly, the boundary data of $u$ determines the boundary data of $u'$.
\end{proof}

\section{Scaffolds} \label{sec:scaffolds}

\subsection{Definition}

As explained in \S \ref{subsec:formalizing}, scaffolds are a set of oriented edges designed to model the flow of information in harmonic continuation.  The definition of scaffold is phrased in terms of \emph{increasing paths}.

Recall $E$ is the set of oriented edges.  A \textbf{path} is a sequence $(e_j)$ of oriented edges indexed by $\{0,\dots,n\}$ or $\N$, such that $(e_j)_+ = (e_{j+1})_-$.  A \textbf{path from $p$ to $q$} is a path $(e_0, \dots, e_n)$ such that $(e_0)_- = p$ and $(e_n)_+ = q$.  A \textbf{cycle} is a nonempty path $(e_0, \dots, e_n)$ with $(e_0)_- = (e_n)_+$.  We say that empty sequence is a path from $p$ to $p$.  If $e \in E$, then a \textbf{path from $e$ to $q$} is path $(e_0,\dots,e_n)$ such that $e = e_0$ and $(e_n)_+ = q$.

If $f: G \to H$ is a UHM and $(e_j)$ is a path from $p$ to $q$ in $G$, then the \textbf{pushforward} $f_* (e_j)$ is defined by taking the sequence $(f(e_j))$, deleting the terms that are collapsed by $f$ into vertices, and reindexing the remaining terms in order such that the zero index is preserved.  This defines a path because $f((e_j)_+) = f((e_j)_-)$ for each deleted term.

Let $S \subset E$.  We say that a path is \textbf{increasing} with respect to $S$ if there are at most two consecutive oriented edges in $E \setminus S$ and no edges in $\overline{S}$.  It is \textbf{decreasing} with respect to $S$ if it is increasing with respect to $\overline{S}$.

\begin{lemma}
Suppose that $f: G \to H$ is a UHM and $S \subset E(H)$.
\begin{itemize}
	\item If $(e_j)$ is a path from $p$ to $q$, then $f_* (e_j)$ is a path from $f(p)$ to $f(q)$.
	\item If $(e_j)$ is a cycle, then $f_* (e_j)$ is a cycle.
	\item If $(e_j)$ is increasing with respect to $f^{-1}(S)$, then $f_* (e_j)$ is increasing with respect to $S$.
	\item If $(e_j)$ is increasing or decreasing with respect to $f^{-1}(S)$ and has infinitely many terms, then $f_* (e_j)$ has infinitely many terms.
\end{itemize}
\end{lemma}

\begin{proof}
The first two properties are immediate.  To prove the third, note that there is at most one consecutive edge not in $f^{-1}(S)$ for $(e_j)$, which implies that there is at most one consecutive edge not in $S$ for $f_* (e_j)$ since none of the collapsed edges are in $f^{-1}(S)$.  To prove the fourth property, note if $(e_j)$ has infinitely many terms, then it has infinitely many in $f^{-1}(S)$ since it cannot have more than one in a row that is not in $f^{-1}(S)$.  Thus, $f_* (e_j)$ has infinitely many terms in $S$.  
\end{proof}

We say $S \subset E$ is a \textbf{scaffold} if $S \cap \overline{S} = \varnothing$ and
\begin{enumerate}
	\item $e \mapsto e_+$ and $e \mapsto e_-$ are injective on $S$,
	\item there is no decreasing path indexed by $\N$,
	\item if there is a increasing path from $p \in V^\circ$ to $y \in V^\circ$, then either $p \in S_+$ or $y \in S_-$.
\end{enumerate}
For pronounceability, we will say $p$ is an \textbf{output} of $S$ if $p \in S_+$ and it is an \textbf{input} if $p \in S_-$.

To picture what is happening, note (1) is equivalent to saying that $p$ is the input of at most one edge in $S$ and it is the output of at most one edge in $S$, and applying (3) to the trivial path from $p$ to $p$ shows that each interior vertex is \emph{either} an input or an output or both, but there is no such requirement for the boundary vertices.

This implies that the oriented edges of $S$ form vertex-disjoint directed paths which span all the interior vertices; each path has a well-defined start point because there are no decreasing infinite paths by (2).  In fact, (2) is equivalent to saying there is a no decreasing cycle \emph{and} no decreasing path with infinitely many distinct edges.  However, the path formed by edges in $S$ may increase infinitely.  Because of (3), a path formed by edges in $S$ cannot terminate at an interior vertex at both endpoints, although it can terminate at an interior vertex at one endpoint.

We can imagine the whole structure as a rickety scaffold | the paths of $S$ are vertical ladders and the edges not in $S$ form bridges between them.  Condition (2) says that you cannot walk in a loop by climbing up ladders and walking across bridges (without crossing two bridges in a row).  Condition (3) says that you cannot get infinitely far down by going down ladders and across bridges.  Condition (4) says that if you are at the bottom of a ladder at an interior point in the graph, then climbing up ladders and walking across bridges will never bring you to an interior point at the top of a ladder.

We partition $V \sqcup E$ into three sets.  We say an edge or vertex is in the \textbf{End} $\End S$ if there a decreasing path from this vertex or edge to an interior vertex that is not an output.  Symmetrically, an edge or vertex is in the \textbf{Beginning} $\Beg S$ if there is a increasing path from this vertex or edge to an interior vertex that is not an input.  Condition (3) guarantees that the End and Beginning are disjoint.  An edge or vertex is in the \textbf{middle} $\Mid S$ if it not in the End or Beginning.

The next lemma shows that scaffolds can be pulled back functorially.

\begin{lemma} \label{lem:scaffunctor}
Let $f: G \to H$ be a UHM.  If $S \subset E(H)$ is a scaffold, then $f^{-1}(S)$ is a scaffold.  Moreover,
\begin{align*}
\Beg f^{-1}(S) &\subset f^{-1}(\Beg S), \\
\End f^{-1}(S) &\subset f^{-1}(\End S), \\
\Mid f^{-1}(S) &\supset f^{-1}(\Mid S).
\end{align*}
\end{lemma}

\begin{proof}
Clearly, $f^{-1}(S) = \overline{f^{-1}(S)} = f^{-1}(S \cap \overline{S}) = \varnothing$.  To show $e \mapsto e_+$ is injective on $f^{-1}(S)$, suppose $e_+ = e_+' = p$.  Then $f(e)_+ = f(e')_+ = f(p)$ and $f(e), f(e') \in S$, so that $f(e) = f(e')$.  But by definition of UHM, $f$ restricts to an injection $\{\tilde{e} \in E(G): e_+ = p\} \to \{\tilde{e} \in E(H): e_+ = f(p)\}$.  This implies that $e = e'$.  Properties (2) and (3) of a scaffold are immediate from the pushforward properties of paths, and so are the claims about End, Middle, and Beginning.
\end{proof}

\begin{corollary}
Let $\Scaf G$ to be the set of scaffolds on $G$ with the partial order given by inclusion as subsets of $E(G)$.  Then $G \mapsto \Scaf G$ is a contravariant functor from the category of $\partial$-graphs and UHMs to the category of posets.
\end{corollary}

We will use the following equivalent characterization of scaffolds in terms of local properties and a global partial order on $E / \bar{~}$.  Here a {\bf partial order on $E / \bar{~}$} is a partial order on the set of unoriented edges, but for notational convenience, we view it as a partial order on the set of oriented edges in which $e$ and $\overline{e}$ occupy the same position in the partial order.
\begin{lemma} \label{lem:scafequivalentdef}
Let $S \subset E(G)$ with $S \cap \overline{S} = \varnothing$.  Then $S$ is a scaffold if and only if there exists a strict partial order $\prec$ on $E / \bar{~}$ such that
\begin{enumerate}[A.]
	\item \emph{(Local Comparison Conditions)} if $e \in S$ and $e_+ = e_+'$, then $e \prec e'$, and If $e \in S$ and $e_- = e_+'$, then $e \succ e'$;
	\item \emph{(Partial Well-Order)} Every subset has a minimal element;
	\item \emph{(Input/Output Alternative)} If $e \preceq e'$, $p = e_+ \in V^\circ$, and $q = e_+' \in V^\circ$, then either $p \in S_+$ or $q \in S_-$.
\end{enumerate}
\end{lemma}

\begin{proof}
If $S$ is a scaffold, then we define the partial order by $e \prec e'$ if $e \neq e'$ and there is an increasing path from $e$ to $e'$.  Since there are no cycles by (2), this relation is irreflexive.  The local comparison conditions follow from the definition of increasing path.  The partial well-order condition follows from (2), and the input-output alternative follows from (3).

Conversely, suppose we have a partial order satisfying (A) - (C).  Then (A) implies that if $e$ and $e'$ are distinct elements of $S$, then $e_+ \neq e_+'$.  The partial order properties and (B) imply that there are no increasing cycles or decreasing infinite paths, hence (2) is satisfied.  Finally, (C) implies (3).
\end{proof}

\begin{remark*}
There can be more than one partial order associated to a given scaffold.  More precisely, we can take any partial order that contains the partial order constructed in the proof.  The partial order constructed in the proof will be called the partial order {\bf induced by} $S$.
\end{remark*}

\subsection{Scaffolds and Layerability} \label{subsec:scaflayerability}

There is a connection between scaffolds and layerability that stems from the following simple observation:

\begin{lemma} \label{lem:minimality}
Let $S$ be a scaffold on $G$ and $\prec$ the partial order induced by $S$.  Let $e \not \in \End S$.  If $e$ is minimal with respect to $\prec$, then it is a boundary spike or boundary edge.
\end{lemma}

\begin{proof}
Suppose that $e \in S \cup \overline{S}$ and assume without loss of generality that $e \in S$.  If $e$ is minimal, then the local comparison conditions imply that $e_-$ has no neighbors.  Since $e \not \in \End S$, $e_-$ cannot be an interior vertex.  Thus, $e$ is a boundary spike if $e_+ \in V^\circ$ and a boundary edge if $e_+ \in \partial V$.

Suppose that $e \not \in S \cup \overline{S}$.  If $e$ is minimal, then $e_+$ and $e_-$ cannot be in $S_+$ by the local comparison conditions.  Since $e_+$ and $e_-$ are not in $\End S$, they must be boundary vertices.  Thus, $e$ is a boundary edge.
\end{proof}

By induction, this ought to imply that if $S$ is a scaffold with $\End S = \varnothing$, then $G$ is layerable.  Conversely, given a layerable filtration $G = G_0 \supset G_1 \supset \dots$, then we can create a scaffold with $\End S = \varnothing$ by letting $S$ be the set of boundary spikes removed at some step of the filtration and defining a partial order on the edges based on which step of the filtration they are removed at.  The next technical lemma makes this idea precise; it may be omitted on a first reading.

\begin{lemma} \label{lem:layerableequiv}
For a $\partial$-graph $G$ with no isolated interior vertices, the following are equivalent:
\begin{enumerate}[a.]
	\item $G$ is layerable.
	\item There exists a scaffold $S$ with $\End S = \varnothing$.
	\item For any $e \in E(G)$, there is a scaffold $S$ with $e \not \in \End S$.
	\item For any $e \in E(G)$, there is a finite partial layerable filtration $G = G_0 \supset \dots \supset G_n$ with $e \not \in E(G_n)$.
\end{enumerate}
\end{lemma}

\begin{proof}
(a) $\implies$ (b).  Consider a layerable filtration $G = G_0 \supset G_1 \supset \dots$, and assume withoul loss of generality that each step in the filtration only involves one type of layer-stripping operation.  Let $S$ be the set of oriented edges that are removed as boundary spikes at some step of the filtration, oriented so that $e_-$ is the boundary vertex and $e_+$ is the interior vertex.  Define a partial order $e \prec e'$ if $e$ is removed at a strictly earlier step than $e'$.  Then similar reasoning as in Lemma \ref{lem:minimality} shows that the local comparison conditions are satisfied.  The partial well-order condition is satisfied because any subset has an edge removed at the minimal-indexed step of the filtration.

The input-output alternative and $\End S = \varnothing$ will be satisfied if we show that every interior vertex is in $S_+$.  But every interior vertex must be removed at some step of the filtration.  To be removed, it must have been changed into a boundary vertex, and the only way that can happen is if it was the interior endpoint of some boundary spike which was contracted.  Hence, the vertex must be in $S_+$.

(b) $\implies$ (c) is trivial.

(c) $\implies$ (d).  Consider a scaffold $S$ with the induced partial order.  I claim that for any $e \in E(G)$, there are only finitely many edges $e \preceq e_0$.  If we suppose not, then there is a minimal edge $e_0$ for which the claim does not hold.  There are only finitely many edges $e_1, \dots, e_n$ which incident to and less than $e_0$, and $\{e \preceq e_0\} = \bigcup_{j=1}^n \{e \preceq e_j\} \cup \{e_0\}$ since any increasing path which ends at $e_0$ must pass through one of the $e_j$'s.  By minimality of $e_0$, $\{e \preceq e_j\}$ is finite, which implies $\{e \preceq e_0\}$ is finite, which is a contradiction.

Now choose $e$.  Let $e_1, \dots, e_k = e$ be the edges $\preceq e$.  We can assume they are listed in some nondecreasing order.  Let $G_0 = G$.  Then $e_1$ is a minimal edge in $G_0$.  By the Lemma, this edge is a boundary spike or boundary edge.  Let $G_1$ be the graph formed by deleting or contracting this edge as appropriate.  Then $e_2$ is a minimal edge in $G_1$, hence a boundary spike or boundary edge.  So (d) follows by induction.

(d) $\implies$ (a).  We assumed in \S 1 that our graphs have countably many edges, so we can write them in a sequence $e_1,e_2,\dots$.  For each $e_n$, choose a $k_n$ and a sequence of subgraphs $G = G_{n,1} \supset \dots \supset G_{n,k_n}$ as in (d).  Then consider the following filtration:
\begin{align*}
&G = G_{1,1}, \quad G_{1,2}, \quad \dots \quad G_{1,k_1}, \\
&G_{1,k_1} \cap G_{2,1}, \quad G_{1,k_1} \cap G_{2,2}, \quad \dots \quad G_{1,k_1} \cap G_{2,k_2} \\
&G_{1,k_1} \cap G_{2,k_2} \cap G_{3,1}, \quad \dots \quad G_{1,k_1} \cap G_{2,k_2} \cap G_{3,k_3} \\
& \dots \dots
\end{align*}
The consecutive elements of this sequence are obtained by a sequence of layer-stripping operations (by Lemma \ref{lem:layerstrippingfunctoriality} applied to inclusion maps of subgraphs).  Thus, we have a layerable filtration which removes all the edges in the graph (but not necessarily all the vertices).  We can obtain a new filtration by replacing each layer-stripping operation with two layer-stripping operations | first remove the edges in the original layer-stripping, then remove any isolated boundary vertices.  The new filtration will remove all the vertices in the graph as well as all the edges since there are no isolated interior vertices.
\end{proof}

\subsection{Scaffolds and Harmonic Continuation} \label{subsec:scafcontinuation}

Recall from the introduction that to recover boundary spikes and boundary edges, we had to prove two claims:  First, there was an \emph{existence} statement that there was some harmonic function with specified boundary conditions, and second there was a \emph{uniqueness} statement that any harmonic function with these boundary conditions was forced to be zero on some region of the network.

We will show how to use scaffolds to verify both the existence and uniqueness claims.  The idea is exactly the same as in \S \ref{subsec:recoveryexample}, though there are some technical subtleties.  First, for each statement, we will only assume a scaffold is defined on a relevant subgraph, so that our statements can be used for harmonic continuation in more general situations (although we will not make full use of this generality).  For infinite $\partial$-graphs it is convenient to use Zorn's-lemma type arguments, which necessitates proof by contradiction.

In both statements, we will use harmonic functions defined on subgraphs, but the subgraphs we use must be compatible with the partial order $\prec$ induced by our scaffold.  We thus make the following definitions.

\begin{definition}
Let $\prec$ be a partial order on $E / \bar{~}$.  We say $T \subset E$ is an {\bf initial subset} if $e \prec e' \in T$ implies $e \in T$.
\end{definition}

\begin{definition}
Let $T \subset E'$.  The {\bf subgraph induced by $T$} is the subgraph $G_T \subseteq G$ with edge set given by $T$ and vertex set given by the endpoints of $T$, where a vertex is interior if and only if it is interior in $G$ and all edges incident to it are in $T$.  In other words, we use the largest possible set of interior vertices that will make $G_T$ a harmonic subgraph.
\end{definition}

\begin{definition}
Given a partial order on $E / \bar{~}$, an {\bf initial subgraph} of $G$ is a subgraph induced by an initial set of edges.
\end{definition}

\begin{lemma}[HC: Uniqueness] \label{lem:hcexistence}
Let $\Gamma'$ be a subnetwork of $\Gamma$ and $S$ a scaffold on $\Gamma'$.  Assume that $V(\Gamma') \setminus S_+ \subset \partial V(\Gamma)$.  If $u$ is a harmonic function on $\Gamma$, then the values of $u$ on $\Gamma'$ are uniquely determined by
\[
u \restriction \partial V(\Gamma) \cap V(\Gamma') \text{ and } \Delta u \restriction \partial V(\Gamma) \cap S_-.
\]
\end{lemma}

\begin{proof}
Let $u$ and $v$ be two harmonic functions on $\Gamma$ such that
\[
u \restriction \partial V(\Gamma) \cap V(\Gamma') = v \restriction \partial V(\Gamma) \cap V(\Gamma')
\]
and
\[
\Delta u \restriction \partial V(\Gamma) \cap S_- =  \Delta v \restriction \partial V(\Gamma) \cap S_-.
\]
Let $T$ be the set of oriented edges $e$ of $\Gamma'$ such that $u$ and $v$ agree on both endpoints of $e$.  Suppose for contradiction $T$ is not all of $E(\Gamma')$.  Then there is a minimal element of $E(\Gamma') \setminus T$ with respect to the scaffold $S$.
\begin{itemize}
	\item If $e \in S \cup \overline{S}$, we can assume $e \in S$.  If $e_- \in \partial V(\Gamma)$, then $u(e_-) = v(e_-)$ and $\Delta u(e_-) = \Delta v(e_-)$ by assumption.  Moreover, all the other edges incident to $e_-$ are in $T$ by minimality of $e$.  This forces $du(e) = dv(e)$ and hence $u(e_+) = v(e_+)$, which is a contradiction.  On the other hand, if $e_- \not \in \partial V(\Gamma)$, then $e_- \in S_+$ by assumption.  Thus, $e_-$ is the output of some other edge $e'$ in $S$.  By minimality of $e$, we have $e' \in T$ and hence $u(e_-) = v(e_-)$.  Moreover, $\Delta u(e_-) = 0 = \Delta v(e_-)$ since $e_-$ is interior.  Since all the other edges incident to $e_-$ are in $T$, we once again have a contradiction.
	\item Suppose $e \not \in S \cup \overline{S}$.  Each endpoint of $e$ must either be a boundary vertex or the output of some edge in $S \cap T$.  In either case, $u(e_+) = v(e_+)$ and $u(e_-) = v(e_-)$, which implies $e \in T$, which is a contradiction. \qedhere
\end{itemize}
\end{proof}

\begin{lemma}[HC: Existence] \label{lem:hcuniqueness}
Let $\Gamma'$ be a subnetwork of $\Gamma$ and $S$ a scaffold on $\Gamma'$.  Assume that $V(\Gamma') \setminus S_- \subset \partial V(\Gamma)$.  Then any harmonic function on $\Gamma \setminus \Gamma'$ extends to a harmonic function on $\Gamma$.
\end{lemma}

\begin{proof}
We can assume without loss of generality that $\Gamma$ has no isolated vertices.

Let $v_0$ be any harmonic function on $\Gamma \setminus \Gamma'$.  Let $\mathcal{Z}$ be the collection of pairs $(\Sigma, v)$ such that
\begin{itemize}
	\item $\Sigma$ is a subnetwork of $\Gamma$ which contains $\Gamma \setminus \Gamma'$.
	\item $v$ is a harmonic function on $\Sigma$ which agrees with $v_0$ on $\Gamma \setminus \Gamma'$.
	\item The subnetwork of $\Gamma$ induced by $E(\Sigma)$ is $\Sigma$ itself.
	\item $\Sigma \cap \Gamma'$ is an initial subnetwork of $\Gamma'$ with respect to $S$.
\end{itemize}
Define a partial order on $\mathcal{Z}$ by setting $(\Sigma, v) \leq (\Sigma',v')$ if $\Sigma$ is a subnetwork of $\Sigma'$ and $v'|_{\Sigma} = v$.  Note that any chain in $\mathcal{Z}$ has an upper bound given by taking the union.  Thus, by Zorn's lemma, $\mathcal{Z}$ has a maximal element $(\Sigma^*, v^*)$.

We claim that $\Sigma^*$ is all of $\Gamma$.  It suffices to show that $E(\Sigma^*)$ contains all of $E(\Gamma')$.  Suppose not.  Then there is a minimal element $e$ of $E(\Gamma') \setminus E(\Sigma^*)$ with respect to $S$.  Consider two cases:
\begin{itemize}
	\item Suppose $e \in S \cup \overline{S}$ and assume that $e \in S$.  Let $\Sigma^{**}$ be the subnetwork of $\Gamma$ induced by $E(\Sigma^*) \cup \{e,\overline{e}\}$.  Note that $\Sigma^{**} \cap \Gamma'$ is an initial subnetwork because $e$ was minimal.  By assumption, $e_+$ is either a boundary vertex of $\Gamma$ or the input of some element of $S$; if it is the input of $e' \in S$, then $e \prec e'$, so $e' \not \in E(\Sigma^*)$.  In either case, $e_+$ must be a boundary vertex of $\Sigma^{**}$.  We can extend $v^*$ to $\Sigma^{**}$ by choosing the potential on $e_+$ so as to make the net current at $e_-$ zero.\footnote{If $e_-$ is a boundary vertex, even this is unnecessary.}  Because $e_+$ is a boundary vertex of $\Sigma^{**}$, this is sufficient to guarantee that the extension $v^{**}$ is harmonic.  This contradicts the maximality of $(\Sigma^*, v^*)$.
	\item Suppose that $e \not \in S \cup \overline{S}$.  Let $\Sigma^{**}$ be the subnetwork of $\Gamma$ induced by $E(\Sigma^*) \cup \{e,\overline{e}\}$.  By the same argument as before, $e_-$ and $e_+$ are boundary vertices of $\Sigma^{**}$.  Hence, $\Sigma^{**}$ has no new interior vertices relative to $\Sigma^*$.  Thus, $v^*$ is harmonic on $\Sigma^{**}$, so once again, we have a contradiction to maximality of $(\Sigma^*, v^*)$. \qedhere
\end{itemize}
\end{proof}

\subsection{Recovery of Boundary Spikes and Boundary Edges}

In \S \ref{subsec:formalizing}, we mentioned sufficient conditions for recovery boundary spikes and boundary edges:  We say that a boundary spike is {\bf recoverable by the scaffold} $S$ if $e$ is not in $S \cup \overline{S}$ and $e$ is in the Middle of $S$.  We say that a boundary edge is {\bf recoverable by the scaffold} $S$ if $e$ is in $S \cup \overline{S}$ and $e$ is in the Middle of $S$.

\begin{lemma} \label{lem:recovery}
If a boundary spike or boundary edge $e$ of $G$ is recoverable by a scaffold, then $w(e)$ is uniquely determined by $\Lambda(\Gamma)$ for any network on $G$ over any field $\mathbb{F}$.
\end{lemma}

\begin{proof}
We shall handle the case of a boundary edge and leave the case of a boundary spike to the reader, since the example from \S \ref{subsec:recoveryexample} focused on boundary spikes.  Let $e$ be a boundary edge which is recoverable using the scaffold $S$, and suppose $e \in S$.  Let $\Gamma_1$ be the subnetwork induced by $\{e' \prec e\}$ and let $\Gamma_2$ by the subnetwork induced by $\{e' \succ e\}$.  Let $S_1 = S \cap E(\Gamma_1)$ and $S_2 = S \cap E(\Gamma_2)$.  Let
\[
P = [\partial V(\Gamma) \cap V(\Gamma_1)] \cup \{e_-\}, \quad Q = \partial V(\Gamma) \cap (S_1)_-.
\]
(In the definition of $P$, we include $e_-$ to handle the case where the only neighbor of $e_-$ is $e_+$; otherwise listing $e_-$ is redundant.)

First, we verify the uniqueness claim that any harmonic function $u$ with $u|_P = 0$ and $\Delta u|_Q = 0$ must be zero on $e_-$ and all its neighbors.  We apply Lemma \ref{lem:hcuniqueness} to $\Gamma_1$ with the scaffold $S_1$.  The hypothesis that $V(\Gamma_1) \setminus (S_1)_+ \subset \partial V(\Gamma)$ is met; indeed, since $e$ is in the middle, $\Gamma_1$ cannot intersect the End and hence any interior vertex of $\Gamma$ which is in $\Gamma_1$ must be the output of some edge in $S_1$.  The conclusion of the lemma tells us that $u$ must be zero on $\Gamma_1$.  In particular, for any $e' \neq e$ incident to $e_-$, we have $u = 0$ on the endpoints of $e'$, so $u$ is zero on the neighbors of $e_-$ and on $e_-$ itself since $e_- \in P$.

Next, we verify the existence claim that there is a harmonic function $u$ with $u|_P = 0$ and $\Delta u|_Q = 0$ and $u(e_+) = 1$ by applying Lemma \ref{lem:hcexistence} to $\Gamma_2$.  We define $u$ to be zero on all of $\Gamma \setminus \Gamma_2$ except that $u(e_+) = 1$.  Note that in $\Gamma \setminus \Gamma_2$, $e_+$ is a boundary vertex and its only neighbor is $e_-$ which is also a boundary vertex, and thus $u$ is harmonic on $\Gamma \setminus \Gamma_2$.  Moreover, since $\Gamma_2$ does not intersect the Beginning of $S$, the hypotheses of Lemma \ref{lem:hcexistence} are met, so $u$ extends to a harmonic function on $\Gamma$.

The existence and uniqueness claims demonstrate recoverability of $w(e)$ as explained in \S \ref{subsec:recoveryexample}.
\end{proof}

\subsection{Two Sufficient Conditions for Recoverability} \label{subsec:solvable}

We now have all the pieces in place to formulate sufficient conditions for recoverability.  We will give two different conditions | \emph{recoverability by scaffolds} and \emph{total layerability} | since unfortunately the more general condition can be harder to test.  As promised, we will show that recoverability by scaffolds ``pull backs'' under UHMs.

\begin{definition}
Suppose that $G$ has a layerable filtration
\[
G = G_0 \supset G_1 \supset \dots
\]
where $G_{j+1}$ is obtained from $G_j$ by a layer-stripping operation.  Suppose that each boundary edge deleted and each boundary spike contracted is recoverable using a scaffold of $G_j$.  Then we say that $G$ is {\bf recoverable by scaffolds}.
\end{definition}

\begin{theorem} \label{thm:solvablerecoverable}
A $\partial$-graph which is recoverable by scaffolds is recoverable over any field $\mathbb{F}$.
\end{theorem}

\begin{proof}
This follows from the layer-stripping strategy laid out in \S \ref{subsec:strategy}.  We recover the edge weights iteratively by recovering the boundary spikes and boundary edges at each step of the filtration which witnesses recoverability by scaffolds.  The weights of the boundary edges and boundary spikes can be recovered using Lemma \ref{lem:recovery} and the boundary behavior of the smaller subnetwork can be found using Lemma \ref{lem:reductionelectrical}.  Since the filtration will exhaust all the edges in the $\partial$-graph, all the edge weights are determined by $\Lambda(\Gamma)$.
\end{proof}

Recoverability by scaffolds is an annoying condition to check because it requires induction.  A more symmetrical and (it turns out) stronger condition is total layerability.  We say a $\partial$-graph $G$ is {\bf totally layerable} if for any oriented edge $e$, there exists a scaffold $S$ with $e \in S \cup \overline{S}$ and $e$ in the Middle of $S$, and there exists a scaffold $S$ with $e \not \in S \cup \overline{S}$ and $e$ in the Middle of $S$.

\begin{theorem}
If $G$ is totally layerable, then $G$ is layerable and $G$ is recoverable by scaffolds.  In fact, any layerable filtration can be used for the process of recovery by scaffolds.
\end{theorem}

\begin{proof}
Layerability of $G$ is nontrivial and follows from Lemma \ref{lem:layerableequiv}.  Given any layerable filtration of $G$, the boundary spikes and boundary edges removed are recoverable by scaffolds since we can find an appropriate scaffold on $G$ and restrict it to $G_j$ using functoriality (Lemma \ref{lem:scaffunctor}).
\end{proof}

\begin{theorem} \label{thm:solvablepullback}
Suppose that $f: G \to H$ is an unramified harmonic morphism. If $H$ is recoverable by scaffolds, then so is $G$.
\end{theorem}

\begin{proof}
Let $\{H_j\}$ be a layer-stripping filtration of $H$ which witnesses recoverability by scaffolds.  We proceed by cases following the same outline as in Lemma \ref{lem:layerstrippingfunctoriality}.

First, suppose $H_{j+1}$ is obtained from $H_j$ by deleting boundary edges.  The scaffolds witnessing recoverability of the boundary edges in $H_j$ will pull back to scaffolds on $H_{j+1}$ witnessing the recoverability of the edges in the preimage.

Second, suppose $H_{j+1}$ is obtained from $H_j$ by deleting isolated boundary vertices.  We must delete boundary edges in $f^{-1}(H_j)$ that map to the isolated boundary vertices.  Such boundary edges are in components with no interior vertices.  We can easily define a scaffold on each of these components.  On the rest of the $\partial$-graph, use any scaffold induced from $H_j$.  If we define a scaffold on each connected component, then that defines a scaffold on the whole $\partial$-graph.  After deleting the boundary edges from $f^{-1}(H_j)$, we simply delete isolated boundary vertices and no scaffold is required.

Third, suppose that $H_{j+1}$ is obtained from $H_j$ by contracting boundary spikes.  Then consider the following steps:
\begin{enumerate}
	\item We must delete boundary edges of $f^{-1}(H_j)$ that map to the boundary vertices of the spikes in $H_j$.  Suppose $e$ is such an edge in $f^{-1}(H_j)$ and that $e'$ is the corresponding spike in $H_j''$.  There is a scaffold $S$ on $H_j$ where $e'$ is not in $S \cup \overline{S}$ and $e \in \Mid S$.  Then $e'$ is not in $f^{-1}(S) \cup \overline{f^{-1}(S)}$ since $e'$ is collapsed by $f$ and it is in the Middle by Lemma \ref{lem:scaffunctor}, so we are done.
	\item Next, we must contract boundary spikes of $f^{-1}(H_j)$ that map to the boundary spikes of $H_j$.  In this case, we can again pull back the scaffolds witnessing recoverability of the spikes in $H_j$.
	\item We must also delete boundary edges in $f^{-1}(H_j)$ that map to the contracted spikes in $H_j$.  Suppose $e$ is a boundary edge and $f(e)$ is one of the contracted spikes of $H_j''$, where $e_+$ corresponds to the boundary endpoint of the spike and $e_-$ to the interior endpoint of the spike in $H_j$.  By assumption, there is a scaffold $S$ on $H_j$ where $f(e) \not \in S \cup \overline{S}$ and is in the Middle.  Let $S' = f^{-1}(S)$.  The boundary edge $e$ is not in $S' \cup \overline{S}'$, so we will modify $S'$.  Let $S''$ be obtained from $S'$ by adding $e$ and removing any edge $e'$ with $e_-' = e_-$.  The latter step is necessary so that $e_-$ will not be the input of two edges in $S'$; however we do not have to worry about this problem at $e_+$ since $e_+$ has degree $1$ thanks to Step 1.  Now if $e_-' = e_-$ and we remove $e'$ from $S$, that may produce new interior vertices which are not the outputs of edges in $S$.  However, since $e$ was in the Middle of $S'$, the new interior vertex which is not an output must be in the Middle or End of $S'$ and hence will not cause a problem, and $e$ will still be in the Middle of $S''$.
	\item It only remains to delete isolated boundary vertices which map to the boundary endpoints of spikes in $H_j$.  This step does not require a scaffold.
\end{enumerate}
\end{proof}

The proof of this last result is similar in spirit to the last one but easier and left as an exercise:

\begin{proposition}~
\begin{enumerate}
	\item If $f: G \to H$ is UHM which does not collapse any edges and $H$ is totally layerable, then $G$ is totally layerable.
	\item A boundary wedge-sum of two totally layerable $\partial$-graphs is totally layerable.
	\item A box product of two totally layerable $\partial$-graphs is totally layerable.
\end{enumerate}
\end{proposition}

\subsection{An Example}

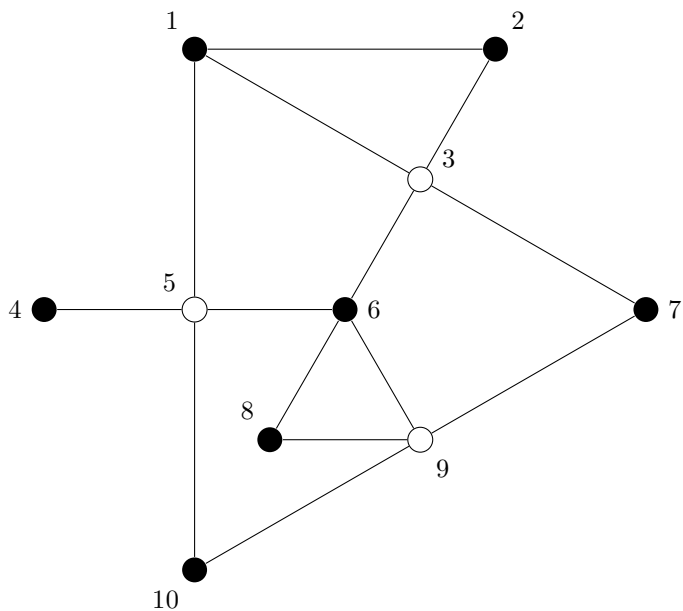
\begin{figure}

\caption{A $\partial$-graph which is solvable but not totally layerable.} \label{fig:ntlexgraph}

\begin{center}
\begin{tikzpicture}
	\node[circle,fill] (1) at (120:4) [label=120:$1$] {};
	\node[circle,fill] (2) at (60:4) [label=60:$2$] {};
	\node[circle,draw] (3) at (60:2) [label=15:$3$] {};
	\node[circle,fill] (4) at (-4,0) [label=180:$4$] {};
	\node[circle,draw] (5) at (-2,0) [label=135:$5$] {};
	\node[circle,fill] (6) at (0,0) [label=0:$6$] {};
	\node[circle,fill] (7) at (0:4) [label=0:$7$] {};
	\node[circle,fill] (8) at (240:2) [label=120:$8$] {};
	\node[circle,draw] (9) at (300:2) [label=300:$9$] {};
	\node[circle,fill] (10) at (240:4) [label=240:$10$] {};
	
	\draw (1) to (2);
	\draw (1) to (3);
	\draw (1) to (5);
	\draw (2) to (3);
	\draw (3) to (6);
	\draw (3) to (7);
	\draw (4) to (5);
	\draw (5) to (6);
	\draw (5) to (10);
	\draw (6) to (8);
	\draw (6) to (9);
	\draw (7) to (9);
	\draw (8) to (9);
	\draw (9) to (10);
\end{tikzpicture}
\end{center}

\end{figure}

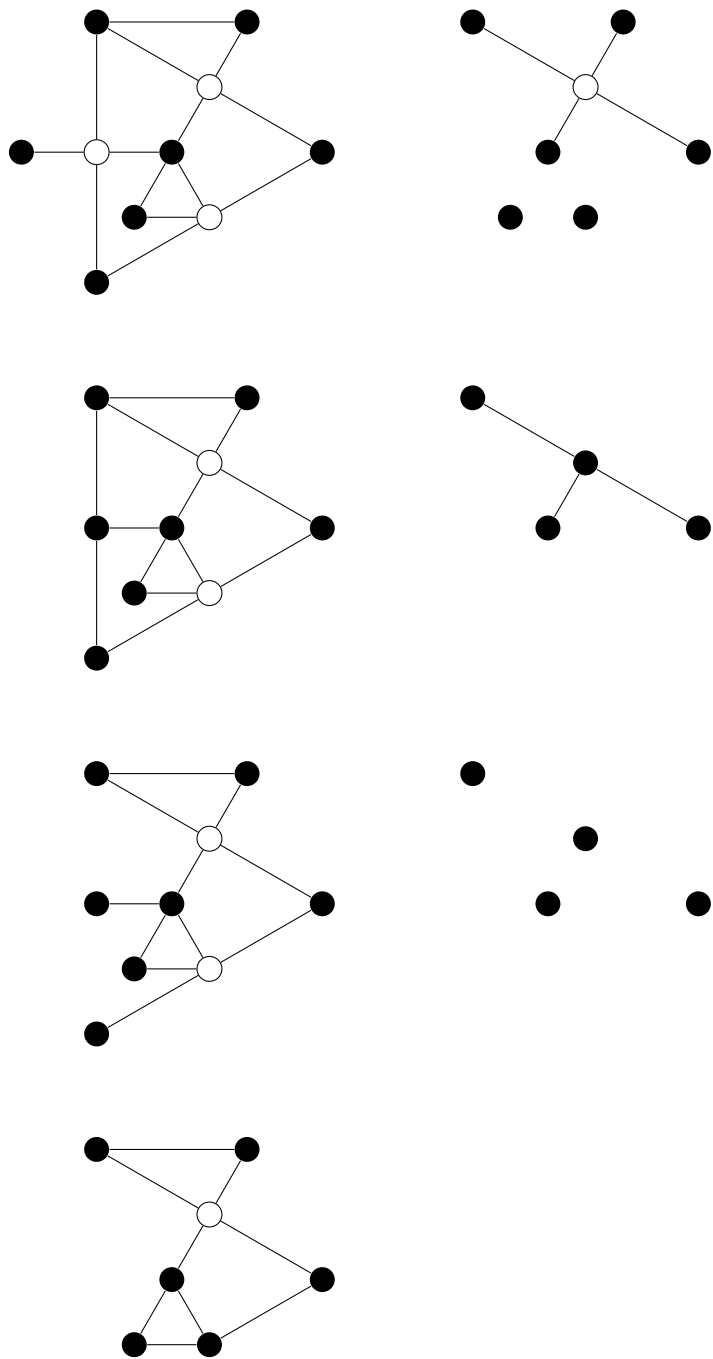
\begin{figure}

\caption{A filtration of the $\partial$-graph in Figure \ref{fig:ntlexgraph} which witnesses recoverability by scaffolds.} \label{fig:ntlexfiltration}

\begin{center}
\begin{tikzpicture}[scale = 0.5]

	\begin{scope}
		\node[circle,fill] (1) at (120:4) {};
		\node[circle,fill] (2) at (60:4) {};
		\node[circle,draw] (3) at (60:2) {};
		\node[circle,fill] (4) at (-4,0) {};
		\node[circle,draw] (5) at (-2,0) {};
		\node[circle,fill] (6) at (0,0) {};
		\node[circle,fill] (7) at (0:4) {};
		\node[circle,fill] (8) at (240:2) {};
		\node[circle,draw] (9) at (300:2) {};
		\node[circle,fill] (10) at (240:4) {};
	
		\draw (1) to (2);
		\draw (1) to (3);
		\draw (1) to (5);
		\draw (2) to (3);
		\draw (3) to (6);
		\draw (3) to (7);
		\draw (4) to (5);
		\draw (5) to (6);
		\draw (5) to (10);
		\draw (6) to (8);
		\draw (6) to (9);
		\draw (7) to (9);
		\draw (8) to (9);
		\draw (9) to (10);
	\end{scope}
		
	\begin{scope}[shift={(0,-10)}]
		\node[circle,fill] (1) at (120:4) {};
		\node[circle,fill] (2) at (60:4) {};
		\node[circle,draw] (3) at (60:2) {};
		\node[circle,fill] (5) at (-2,0) {};
		\node[circle,fill] (6) at (0,0) {};
		\node[circle,fill] (7) at (0:4) {};
		\node[circle,fill] (8) at (240:2) {};
		\node[circle,draw] (9) at (300:2) {};
		\node[circle,fill] (10) at (240:4) {};
	
		\draw (1) to (2);
		\draw (1) to (3);
		\draw (1) to (5);
		\draw (2) to (3);
		\draw (3) to (6);
		\draw (3) to (7);
		\draw (5) to (6);
		\draw (5) to (10);
		\draw (6) to (8);
		\draw (6) to (9);
		\draw (7) to (9);
		\draw (8) to (9);
		\draw (9) to (10);
	\end{scope}
	
	\begin{scope}[shift={(0,-20)}]
		\node[circle,fill] (1) at (120:4) {};
		\node[circle,fill] (2) at (60:4) {};
		\node[circle,draw] (3) at (60:2) {};
		\node[circle,fill] (5) at (-2,0) {};
		\node[circle,fill] (6) at (0,0) {};
		\node[circle,fill] (7) at (0:4) {};
		\node[circle,fill] (8) at (240:2) {};
		\node[circle,draw] (9) at (300:2) {};
		\node[circle,fill] (10) at (240:4) {};
	
		\draw (1) to (2);
		\draw (1) to (3);
		\draw (2) to (3);
		\draw (3) to (6);
		\draw (3) to (7);
		\draw (5) to (6);
		\draw (6) to (8);
		\draw (6) to (9);
		\draw (7) to (9);
		\draw (8) to (9);
		\draw (9) to (10);
	\end{scope}
	
	\begin{scope}[shift={(0,-30)}]
		\node[circle,fill] (1) at (120:4) {};
		\node[circle,fill] (2) at (60:4) {};
		\node[circle,draw] (3) at (60:2) {};
		\node[circle,fill] (6) at (0,0) {};
		\node[circle,fill] (7) at (0:4) {};
		\node[circle,fill] (8) at (240:2) {};
		\node[circle,fill] (9) at (300:2) {};
	
		\draw (1) to (2);
		\draw (1) to (3);
		\draw (2) to (3);
		\draw (3) to (6);
		\draw (3) to (7);
		\draw (6) to (8);
		\draw (6) to (9);
		\draw (7) to (9);
		\draw (8) to (9);
	\end{scope}

	\begin{scope}[shift={(10,0)}]
		\node[circle,fill] (1) at (120:4) {};
		\node[circle,fill] (2) at (60:4) {};
		\node[circle,draw] (3) at (60:2) {};
		\node[circle,fill] (6) at (0,0) {};
		\node[circle,fill] (7) at (0:4) {};
		\node[circle,fill] (8) at (240:2) {};
		\node[circle,fill] (9) at (300:2) {};
	
		\draw (1) to (3);
		\draw (2) to (3);
		\draw (3) to (6);
		\draw (3) to (7);
	\end{scope}

	\begin{scope}[shift={(10,-10)}]
		\node[circle,fill] (1) at (120:4) {};
		\node[circle,fill] (3) at (60:2) {};
		\node[circle,fill] (6) at (0,0) {};
		\node[circle,fill] (7) at (0:4) {};
	
		\draw (1) to (3);
		\draw (3) to (6);
		\draw (3) to (7);
	\end{scope}
	
	\begin{scope}[shift={(10,-20)}]
		\node[circle,fill] (1) at (120:4) {};
		\node[circle,fill] (3) at (60:2) {};
		\node[circle,fill] (6) at (0,0) {};
		\node[circle,fill] (7) at (0:4) {};
	\end{scope}	
	
\end{tikzpicture}
\end{center}

\end{figure}

\begin{figure}

\caption{A scaffold on the graph from Figure \ref{fig:ntlexgraph}.  The numbering indicates one possible total order of the edges which is compatible with the partial order induced by the scaffold.} \label{fig:ntlexscaffold}

\begin{center}
\begin{tikzpicture}
	\node[circle,fill] (1) at (120:4) {};
	\node[circle,fill] (2) at (60:4) {};
	\node[circle,draw] (3) at (60:2) {};
	\node[circle,fill] (4) at (-4,0) {};
	\node[circle,draw] (5) at (-2,0) {};
	\node[circle,fill] (6) at (0,0) {};
	\node[circle,fill] (7) at (0:4) {};
	\node[circle,fill] (8) at (240:2) {};
	\node[circle,draw] (9) at (300:2) {};
	\node[circle,fill] (10) at (240:4) {};
	
	\draw[blue] (6) to node[auto] {$1$} (8);
	\draw[orange] (8) to node[auto] {$2$} (9) [arrow inside={end=stealth,opt={orange, scale=2}}{0.5}];
	\draw[blue] (9) to node[auto] {$3$} (10);
	\draw[orange] (10) to node[auto] {$4$} (5) [arrow inside={end=stealth,opt={orange, scale=2}}{0.5}];
	\draw[blue] (4) to node[auto] {$5$} (5);
	\draw[blue] (5) to node[auto] {$6$} (6);
	\draw[blue] (6) to node[auto] {$7$} (9);
	\draw[orange] (5) to node[auto] {$8$} (1) [arrow inside={end=stealth,opt={orange, scale=2}}{0.5}];
	\draw[orange] (6) to node[auto] {$9$} (3) [arrow inside={end=stealth,opt={orange, scale=2}}{0.5}];
	\draw[orange] (9) to node[auto] {$10$} (7) [arrow inside={end=stealth,opt={orange, scale=2}}{0.5}];
	\draw[blue] (1) to node[auto] {$11$} (3);
	\draw[blue] (3) to node[auto] {$12$} (7);
	\draw[orange] (3) to node[auto] {$13$} (2) [arrow inside={end=stealth,opt={orange, scale=2}}{0.5}];
	\draw[blue] (1) to node[auto] {$14$} (2);

\end{tikzpicture}
\end{center}

\end{figure}
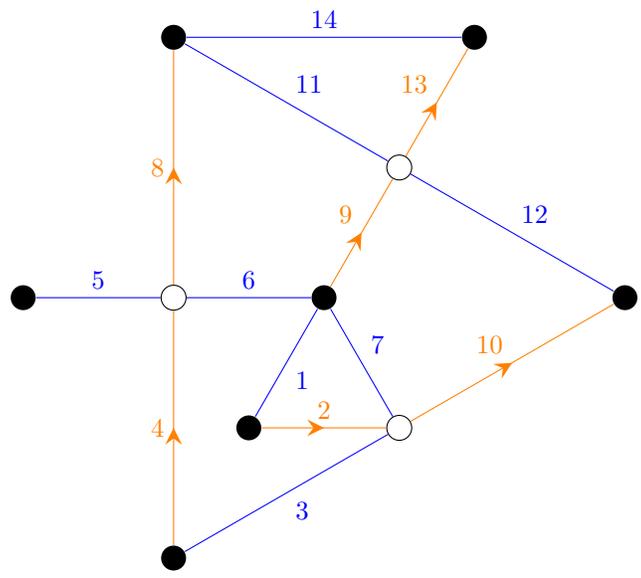

Total layerability and recoverability by scaffolds are not equivalent in general.  Figure \ref{fig:ntlexgraph} shows a $\partial$-graph which is recoverable by scaffolds, but not totally layerable.  A filtration witnessing recoverability is shown in Figure \ref{fig:ntlexfiltration}.  A scaffold used for recovering the first boundary spike (and several of the later steps is shown in Figure \ref{fig:ntlexscaffold}.  The construction of the scaffolds for the remaining steps is left as an exercise.

However, this $\partial$-graph is not totally layerable.  Let's index the vertices as in Figure \ref{fig:ntlexgraph} and denote by $(i,j)$ the oriented edge from vertex $i$ to vertex $j$.  I claim that there does not exist a scaffold in which $(1,2)$ is in $S \cup \overline{S}$ and $\Mid S$.

Suppose for contradiction such a scaffold $S$ exists.\footnote{For best results when reading this proof, the reader should keep referring to Figure \ref{fig:ntlexgraph} and mark in pencil the ladders and planks in each of the scenarios considered.  Time and space constraints prevent me from including figures for each case.}  As a result of Lemma \ref{lem:scafdelete} in \S \ref{subsec:IOscaffolds}, we can assume without loss of generality that each boundary vertex is incident to at most one edge in $S$.  For finite $\partial$-graphs, $S$ is a scaffold if and only if $\overline{S}$ is a scaffold, so we can also assume $(1,2) \in S$.  Now we have
\[
(1,3) \prec (1,2) \prec (2,3).
\]
Since $(1,2)$ is assumed to be in the middle of $S$, we know $3$ is both the input and the output of some edge in $S$.  Hence, $(3,6)$ and $(3,7)$ are both in $S \cup \overline{S}$, and one must be oriented going into $3$ and one going out.  Since each boundary vertex (in particular, vertex $6$ or $7$) is incident to at most one edge in $S$, we conclude that $(9,7)$ and $(9,6)$ are not in $S \cup \overline{S}$.

We now treat two cases:
\begin{itemize}
	\item Suppose $(6,3) \in S$ and $(3,7) \in S$ (in that orientation).  Then
	\[
	(9,7) \succ (3,7) \succ (2,3) \succ (1,2),
	\]
	and so $(9,7) \not \in \Beg \mathcal{S}$.  Hence, $9$ must be the input of some edge in $S$.  We already know $(9,7)$ and $(9,6)$ are not in $S$, either $(9,8)$ or $(9,10)$ must be in $S$.  It cannot be $(9,8)$ because in that case
	\[
	(9,8) \prec (8,6) \prec (6,3) \prec (3,7) \prec (7,9) \prec (9,8),
	\]
	a contradiction.  So suppose $(9,10)$ is a ladder with $9$ as its foot.  In that case, $5$ is incident to only one edge in $S$, namely $(4,5)$, since all the other boundary vertices adjacent to $5$ already some edge in $S$ incident to them.  Since $(5,1) \prec (1,2) \in \Mid S$, vertex $5$ must be the output of some edge in $S$.  But since
	\[
	(5,10) \succ (9,10) \succ (9,7) \succ (3,7) \succ (2,3) \succ (1,2),
	\]
	vertex $5$ must be the input of some edge in $S$ as well.  This contradicts the fact that $5$ can only be incident to one edge in $S$.
	
	\item Suppose $(7,3) \in S$ and $(3,6) \in S$ (in that orientation).  By similar reasoning as before, since $(9,7) \prec (1,2)$, we have $(9,7) \not \in \End S$, hence $9$ must be the output of some edge in $S$.  We know $(8,9)$ cannot be in $S$, since then we would have
	\[
	(8,9) \prec (9,7) \prec (7,3) \prec (3,6) \prec (6,8) \prec (8,9).
	\]
	On the other hand, if $(10,9)$ is in $S$, then similar to before, $5$ can only be incident to one edge in $S$.  Since
	\[
	(5,10) \prec (10,9) \prec (9,7) \prec (7,3) \prec (3,1) \prec (1,2) \in \Mid S,
	\]
	we know $5$ must be the output of some edge in $S$.  But since
	\[
	(5,6) \succ (6,3) \succ (2,3) \succ (1,2) \in \Mid S,
	\]
	we know $5$ must be the input of some edge in $S$.  Thus, we have another contradiction.
\end{itemize}

\section{IO-Graphs and Factorization} \label{sec:IO}

\subsection{The Category of Linear Relations}

As explained in \S \ref{subsec:mixeddata} and \S \ref{subsec:gluing}, we consider certain linear relations of boundary data on electrical networks.  In preparation, we state the definition and basic properties of the category of linear relations.  The proofs are left as exercises.

\begin{definition}
In the {\bf category of linear relations} $\F$\textbf{-LinRel}, the objects are finite-dimensional vector spaces over $\F$.  A morphism $T: V \rightsquigarrow W$ is a linear relation between $V$ and $W$, that is, a linear subspace of $V \times W$.  Here we use ``$\rightsquigarrow$'' to emphasize that $T$ is not necessarily a function.  If $S: U \rightsquigarrow V$ and $T: V \rightsquigarrow W$, then we define $T \circ S: U \rightsquigarrow W$ by
\[
T \circ S = \{(x,z): \exists y \in V \text{ such that } (x,y) \in T \text{ and } (y,z) \in S\}.
\]
The identity morphism $\id_V$ is the diagonal subspace of $V \times V$.
\end{definition}

\begin{definition}
If $T: V \rightsquigarrow W$, we define $\overline{T}: W \rightsquigarrow V$ by $\{(y,x): (x,y) \in T\}$.  Generalizing the notions of kernel and image from linear algebra, we define
\[
\ker T = \{x \in V: (x,0) \in T\}, \qquad \im T = \{y \in W: (x,y) \in T \text{ for some } x\}.
\]
\end{definition}

\begin{definition}
Note that $T$ defines a linear isomorphism $\im \overline{T} / \ker T \to \im T / \ker \overline{T}$.  We define
\[
\rank T = \dim \im \overline{T} - \dim \ker T = \dim \im T - \dim \ker \overline{T} = \rank \overline{T}.
\]
\end{definition}

\begin{lemma} \label{lem:rankinequality}
$\rank(T \circ S) \leq \min(\rank T, \rank S)$.
\end{lemma}

\begin{lemma} \label{lem:relmorphisms}
Let $T: V \rightsquigarrow W$.  Then
\begin{itemize}
	\item $T$ is an isomorphism if and only if $T$ defines a linear bijection $V \to W$.
	\item $T$ is a monomorphism if and only if $\im \overline{T} = V$ and $\ker T = 0$.
	\item $T$ is an epimorphism if and only if $\im T = W$ and $\ker \overline{T} = 0$.
\end{itemize}
\end{lemma}

\begin{lemma} \label{lem:kernelcomp}
If $S$ is an epimorphism, then
\[
\dim \ker T \circ S = \dim \ker T + \dim \ker S.
\]
Symmetrically, if $T$ is a monomorphism, then
\[
\dim \ker \overline{T \circ S} = \dim \ker \overline{T} + \dim \ker \overline{S}.
\]
\end{lemma}

\begin{lemma} \label{lem:relmorphisms2}
If $S: U \rightsquigarrow V$ is an epimorphism and $T: V \rightsquigarrow W$ is a monomorphism, then
\begin{itemize}
	\item $\rank(T \circ S) = \dim V$.
	\item $\dim \ker T \circ S = \dim \ker S$.
	\item $\dim \ker \overline{T \circ S} = \dim \ker \overline{T}$.
\end{itemize}
\end{lemma}

\subsection{The Category of IO-Graphs}

The following constructions are adapted from \cite{BF}.  The motivation is explained in \S \ref{subsec:gluing}.

\begin{definition}
We define the {\bf category of IO-graphs} as follows:  The objects are finite sets.  A morphism $P \to Q$ is an equivalence class of triples $(G,i,j)$, where $G$ is a finite graph and $i: P \to V(G)$ and $j: Q \to V(G)$ are injective maps called {\bf labelling functions}.  Here we say that $(G,i,j)$ and $(G',i',j')$ are equivalent if there is a graph isomorphism $f: G \to G'$ such that the following commutes:
\begin{center}
	\begin{tikzpicture}[scale = 1.3]
		\node (P) at (1,0) {$P$};
		\node (Q) at (-1,0) {$Q$};
		\node (V1) at (0,1) {$V(G)$};
		\node (V2) at (0,-1) {$V(G')$};
		
		\draw[->] (P) to node[auto,swap] {$i$} (V1);
		\draw[->] (Q) to node[auto] {$j$} (V1);
		\draw[->] (P) to node[auto] {$i'$} (V2);
		\draw[->] (Q) to node[auto,swap] {$j'$} (V2);
		\draw[->] (V1) to node[auto] {$f$} (V2);
	\end{tikzpicture}
\end{center}
We denote the morphism by $[G,i,j]$.

We define the composition of two morphisms $[G,i,j]: P \to Q$ and $[G',i',j']: Q \to R$ as follows:  Let $G^*$ be ``the'' disjoint union of $G$ and $G'$ modulo the identifications $j(q) \sim i'(q)$ for $q \in Q$.  We let $[G',i',j'] \circ [G,i,j] = [G^*,i,j']$, where we view $i$ and $j'$ as functions from $P$ and $R$ into $V(G^*)$ by precomposing with $V(G) \to V(G^*)$ and $V(G') \to V(G^*)$.
\end{definition}

\begin{definition}
For a morphism $[G,i,j]$, we call the vertices $i(P)$ {\bf inputs} and the vertices of $j(Q)$ {\bf outputs}.
\end{definition}

It is an standard exercise to verify that composition is well-defined and associative.  Moreover, the identity transformation $P \to P$ is given by a graph with vertex set $P$ and no edges, and $i: P \to P$ and $j: P \to P$ are the identity function.

\begin{definition}
The {\bf category of IO-networks} $\F$\textbf{-IO-net} is defined in the same way, with graphs replaced by networks without specified boundary vertices.  An isomorphism of networks is assumed to preserve the edge weights.  We denote the network by $\Gamma$ and the morphism by $[\Gamma,i,j]$, hoping that the distinctions will be made clear by context.
\end{definition}

\begin{definition}
The {\bf IO-boundary behavior functor} $X: \F \text{\bf -IO-net} \to \F \text{\bf -LinRel}$ is defined as follows.  For a finite set $P$, we define $X(P) = (\F^P)^2$.  Now suppose $[\Gamma,i,j]: P \to Q$ is a morphism of $\F$\text\bf{-IO-net}.  Let $i_*$ and $j_*$ be the inclusions $\F^P \to \F^V$ and $\F^Q \to \F^V$ and let $i^*$ and $j^*$ be the projections $\F^V \to \F^P$ and $\F^V \to \F^Q$.  We define $X([\Gamma,i,j]): (\F^P)^2 \rightsquigarrow (\F^Q)^2$ as the set of all
\[
(x,y) = ((x_1,x_2), (y_1,y_2)) \in (\F^P)^2 \times (\F^Q)^2
\]
such that there exists $u \in \F^V$ with
\[
i^* u = x_1, \quad j^* u = x_2, \quad \Delta u = j_* y_2 - i_* x_2.
\]
\end{definition}

This definition says that $u$ agrees with $x_1$ on $P$ and $y_1$ on $Q$.  Moreover, $\Delta u(p)$ is zero for any vertex $p$ which is not in $i(P)$ or $j(Q)$, so $u$ is harmonic on the network with $\partial V = i(P) \cup j(Q)$.  The boundary current is $\Delta u(i(p)) = x_2(p)$ for $-i(p) \in i(P) \setminus j(Q)$, it is $\Delta u(j(q)) = y_2(q)$ for $j(q) \in j(Q) \setminus i(P)$, and it is $\Delta u(r) = y_2(q) - x_2(p)$ whenever $r = i(p) = j(q)$.  Thus, it matches the description in \S \ref{subsec:gluing}.  The verification that $X$ is a functor is left as an exercise (see \S \ref{subsec:gluing} and Proposition \ref{prop:subnetworkgluing}).

\subsection{Elementary IO-network Morphisms} \label{subsec:elementarymorphisms}

We define the following types of {\bf elementary IO-graph morphisms}, and call them uncreatively {\bf type 1}, \dots, {\bf type 4}:
\begin{enumerate}
	\item A graph where each component is either an isolated vertex or two vertices connected by a single edge.  Each of the isolated vertices is both an input and an output.  Each edge has one endpoint as an input and one as an output.
	\item A graph where all the vertices are both inputs and outputs with some edges between them.
	\item A graph with no edges in which all the vertices are inputs, but not all are outputs.  The non-outputs are called {\bf input stubs}.
	\item A graph with no edges in which all the vertices are outputs, but not all are inputs.  The non-inputs are called {\bf output stubs}.
\end{enumerate}
The {\bf elementary IO-network morphisms} are defined the same way, except with weights attached to the edges.

Let us compute $X$ for each of the elementary IO-network morphisms.  If $[\Gamma,i,j]$ is type 1, then $X([G,i,j])$ is an isomorphism.  Explicitly, let us index the inputs and the outputs by $[n] = \{0,\dots,n\}$, such that the input and output in the same component have the same index.  Assume that the components with edges correspond to indices $1, \dots, k$ with weights $w_1, \dots, w_k$.  Then $(x,y) \in X([G,i,j])$ if and only if
\[
\begin{pmatrix} y_1 \\ y_2 \end{pmatrix} = \begin{pmatrix} 1 & \sum_{j=1}^k w_j^{-1} E_{j,j} \\ 0 & 1 \end{pmatrix} \begin{pmatrix} x_1 \\ x_2 \end{pmatrix},
\]
where the blocks are $n \times n$ and $E_{i,j}$ is the matrix with a $1$ in the $(i,j)$ entry and zeros elsewhere.

\begin{remark*}
Note that for the case of $1$ edge, this is the transformation as in \S \ref{subsec:strategy} for changing the boundary data when adding a boundary spike.  We will explain this more fully in \S \ref{subsec:IOlayerstripping}.
\end{remark*}

If $[\Gamma,i,j]$ is type 2, then $X([G,i,j])$ is an isomorphism.  Explicitly, let us index the inputs and outputs by $[n]$ with the input and output indices equal to each other.  Suppose that we have edges between vertices $p_j$ and $q_j$ with weight $w_j$ for $j = 1, \dots, k$.  Then $X([G,i,j])$ is given by
\[
\begin{pmatrix} y_1 \\ y_2 \end{pmatrix} = \begin{pmatrix} 1 & 0 \\ \sum_{j=1}^k w_j(E_{p_j,p_j} + E_{q_j,q_j} - E_{p_j,q_j} - E_{q_j,p_j}) & 1 \end{pmatrix} \begin{pmatrix} x_1 \\ x_2 \end{pmatrix}
\]

If $[\Gamma,i,j]$ is type 3, then $X([\Gamma,i,j])$ is an epimorphism.  Explicitly, let us index the inputs by $[n]$ and the outputs by $[k] \subset [n]$.  Then
\[
X([\Gamma,i,j]) = \{(x,y): x_1|_{[k]} = y_1, x_1|_{[n] \setminus [k]} = 0\}.
\]
The dimension of the kernel is $n - k$.

If $[\Gamma,i,j]$ is type 4, then $X([\Gamma,i,j])$ is a monomorphism and the same formula holds switching the inputs and outputs.

\subsection{Elementary Factorizations and the Rank-Connection Principle} \label{subsec:elementary}

An {\bf elementary factorization} of $[\Gamma,i,j]: P \to Q$ is a factorization
\[
[\Gamma,i,j] = [\Gamma_n,i_n,j_n] \circ \dots \circ [\Gamma_1,i_1,j_1],
\]
where each factor is an elementary IO-network morphism and the type 3 morphisms come \emph{before} the type 4 morphisms in the order of composition.  If $[\Gamma_k, i_k, j_k]: P_{k-1} \to P_k$, then we define the {\bf width} of the factorization to be $\min_k |P_k|$.

\begin{remark*}
In an elementary factorization, the type 3 networks represent obstacles to existence of harmonic extensions and the type 4 networks represent obstacles to uniqueness.
\end{remark*}

With all the setup we have done it is now easy to show that the rank-connection principle holds for $[\Gamma,i,j]: P \to Q$ whenever it has an elementary factorization.  We first handle the algebraic quantities:

\begin{lemma} \label{lem:rankwidth}
Suppose $[\Gamma,i,j]: P \to Q$ admits an elementary factorization.  Then
\begin{itemize}
	\item $\rank X([\Gamma,i,j])$ is twice the width of the factorization.
	\item $\dim \ker X([\Gamma,i,j])$ is the number of input stubs.
	\item $\dim \im \overline{X([\Gamma,i,j])}$ is twice the width plus the number of input stubs.
\end{itemize}
The same holds with $P$ and $Q$ reversed.  In particular, all these quantities are independent of the choice of edge weights and of the field.
\end{lemma}

\begin{proof}
Suppose
\[
[\Gamma,i,j] = [\Gamma_n,i_n,j_n] \circ \dots \circ [\Gamma_1,i_1,j_1],
\]
where $[\Gamma_k,i_k,j_k]: P_{k-1} \to P_k$.  Pick $P_\ell$ with $|P_\ell|$ minimal, so that $|P_\ell|$ is the width of the factorization.  Since the type 3 networks decrease the size of $P_k$ and the type 4 networks increase it, we know that the type 3 networks come before $P_\ell$ and the type 4 networks come after $P_\ell$.  Let $\Sigma_1$ be the morphism $P \to P_\ell$ and $\Sigma_2$ the morphism $P_\ell \to Q$ formed by composing the morphisms in the factorization.  Then from our description of the boundary behavior for elementary networks, $X(\Sigma_1)$ is a composition of epimorphisms, hence an epimorphism and $X(\Sigma_2)$ is a composition of monomorphisms, hence a monomorphism.  Thus, Lemma \ref{lem:relmorphisms2} implies that $\rank X([\Gamma,i,j])$ is the dimension of $(\F^{P_\ell})^2$, which is twice the width of the factorization.

Now the epimorphisms $[\Gamma_1,i_1,j_1], \dots, [\Gamma_k,i_k,j_k]$ have a zero-dimensional kernel in the case of type 1 and 2, and in the case of a type 3 network, the dimension is the number of input stubs.  This implies by Lemma \ref{lem:kernelcomp} that the dimension of $\ker X(\Sigma_1)$ is the total number of input stubs.  Similarly, the dimension of $\ker \overline{X(\Sigma_2)}$ is the total number of output stubs.  This proves the first and second claims, and the third claim follows by ``rank-nullity'' arithmetic.
\end{proof}

Now we must deal with the geometric quantity of the maximum size connection, whose precise definition is as follows:
\begin{definition}
Let $G$ be a $\partial$-graph and let $P, Q \subset \partial V(G)$; then a {\bf connection from $P$ to $Q$} is a collection of disjoint paths in $G$ such that each path has its starting point in $P$ and no other points in $P$ and has its ending point in $Q$ and no other points in $Q$.  Note that this makes sense even if $P$ and $Q$ intersect, but in this case the only paths allowable for vertices in $P \cap Q$ are the trivial paths connecting a vertex to itself.  We define $m(P,Q)$ to be the maximum size connection from $P$ to $Q$.  If $[G,i,j]: P \to Q$ is an IO-graph morphism, then we define $m([G,i,j]) = m(i(P), j(Q))$ in the $\partial$-graph with $\partial V = i(P) \cup j(Q)$.
\end{definition}

\begin{lemma} \label{lem:connectionwidth}
Suppose that $[G,i,j]: P \to Q$ admits an elementary factorization.  Then the maximum size connection between $P$ and $Q$ is the width of the factorization.  In particular, the width of the factorization only depends on $[G,i,j]$.
\end{lemma}

\begin{proof}
Note that if $[G,i,j]$ is a type 1, 2, or 3 elementary morphism, then there are disjoint paths from any subset of $j(Q)$ to some subset of $i(P)$.  If $[G,i,j]: P \to Q$ is composed of several elementary morphisms of types 1, 2, and 3 we can join the paths in each elementary morphism together to conclude that there are disjoint paths from all of $j(Q)$ to some subset of $i(P)$.  Similarly, if $[G,i,j]$ is composed of type 1, 2, and 4 elementary morphisms, then there are disjoint paths from all of $i(P)$ to some subset of $j(Q)$.

If $[G,i,j]$ admits an elementary factorization, then it can be written as $[G_2,i_2,j_2] \circ [G_1,i_1,j_1]$, where $[G_1,i_1,j_1]: P \to R$ is composed of types 1, 2, 3 and $[G_2, i_2, j_2]: R \to Q$ is composed of types 1, 2, 4.  There are paths through $G_1$ connecting all the vertices of $j_1(R)$ to some of the vertices of $i_1(P)$ and paths through $G_2$ connecting all the vertices of $i_2(R)$ to some of the vertices of $j_2(Q)$.  Joining these paths together provides a connection of size $|R|$ from $i(P)$ ot $j(Q)$.  Thus, the maximum size connection is at least as large as the width $|R|$.

On the other hand, any path from $i(P)$ to $j(Q)$ must pass through $i_2(R)$.  Thus, there can be at most $|R|$ disjoint paths.
\end{proof}

Combining these two lemmas, together with the fact that $|P|$ is is the width of the factorization plus the number of input stubs, yields
\begin{theorem}[Rank-Connection Principle 1] \label{thm:RC1}
If $[\Gamma,i,j]: P \to Q$ admits an elementary factorization, then
\begin{itemize}
	\item $\rank X([\Gamma,i,j]) = 2 m([\Gamma,i,j])$;
	\item $\dim \ker X([\Gamma,i,j]) = |P| - m([\Gamma,i,j])$;
	\item $\dim \im \overline{X([\Gamma,i,j])} = |P| + m([\Gamma,i,j])$.
\end{itemize}
The same holds with $P$ and $Q$ reversed.
\end{theorem}

\subsection{Elementary IO-Graph Morphisms and Layer-Stripping Operations} \label{subsec:IOlayerstripping}

We can think of IO-networks as \emph{transformations} to apply to ordinary networks in the following way.  If $\Gamma_0$ is a finite network, then we can view $\Gamma_0$ as a morphism $[\Gamma_0,i_0,j_0]: \varnothing \to [n]$, where $n = |\partial V|$ and $j_0$ is any labelling of $\partial V$ by $[n]$.  In this case, $X([\Gamma_0,i_0,j_0]) = \{0\} \times \Lambda(\Gamma_0)$.

Now if $[\Gamma,i,j]: [n] \to [m]$ is an IO-network morphism, then $[\Gamma,i,j] \circ [\Gamma_0,i_0,j_0]$ is a network with $m$ boundary vertices.  This gives a geometric meaning to the postcomposition map
\[
[\Gamma,i,j] \circ -: \Hom_{\F\mathbf{-IO-net}}(\varnothing,[n]) \to \Hom_{\F\mathbf{-IO-net}}(\varnothing,[m]).
\]
In the case of elementary morphisms, postcomposition yields the following:
\begin{itemize}
	\item If $[\Gamma,i,j]$ is type 1, then postcomposing it adds some boundary spikes to the network.
	\item If $[\Gamma,i,j]$ is type 2, then postcomposing it adds some boundary edges to the network.
	\item If $[\Gamma,i,j]$ is type 3, then postcomposing it changes some boundary vertices to interior.
	\item If $[\Gamma,i,j]$ is type 4, then postcomposing it adds some isolated boundary vertices.
\end{itemize}
Thus, elementary morphisms are a geometric realization of the inverses of layer-stripping operations.  By application of the functor $X$, we see that each of these transformations modifies the boundary behavior $\Lambda$ of the original network by postcomposing with the linear relation corresponding to the elementary morphism.

In particular, if $\Gamma'$ is obtained from $\Gamma$ by attaching a boundary spike or boundary edge, then we have
\[
\Lambda(\Gamma') = \Xi \cdot \Lambda(\Gamma),
\]
where $\Xi$ is the invertible matrix corresponding to the elementary morphism described in \S \ref{subsec:elementarymorphisms}.

Now if $\Gamma$ is a finite layerable network, then we can express the morphism $[\Gamma,i,j]: \varnothing \to [n]$ as
\[
[\Gamma,i,j] = [\Gamma_n,i_n,j_n] \circ \dots \circ [\Gamma_0,i_0,j_0],
\]
where $[\Gamma_0,i_0,j_0]: \varnothing \to [n]$ is a network with $V = \partial V = [n]$ and no edges, and the other morphisms are type 1 or type 2.  This implies that
\[
\Lambda(\Gamma) = \Xi_n \circ \dots \circ \Xi_1(\F^n \times \{0\})
\]
since $\Lambda(\Gamma_0) = \F^n \times \{0\} \subset \F^n \times \F^n$.

\begin{remark*}
This is an efficient way to compute a basis for $\Lambda(\Gamma)$ since each $\Xi_j$ amounts to $1$ row operation for each boundary spike added or $4$ row operations for each boundary edge added (as a consequence of the formula for the $\Xi$ matrices).
\end{remark*}

In the same spirit as \cite{LP}, these considerations lead us to define the {\bf electrical linear group} $EL_n(\F)$ as the group of matrices generated by
\[
\{X([\Gamma,i,j]) \text{ for } [\Gamma,i,j]: [n] \to [n] \text{ is type 1 or type 2}\}.
\]
It is generated by the matrices corresponding to morphisms with only one edge, which we name as follows:
\[
\Xi_j(a) = \begin{pmatrix} 1 & aE_{j,j} \\ 0 & 1 \end{pmatrix}, \quad \Xi_{i,j}(a) = \begin{pmatrix} 1 & 0 \\ a(E_{i,i} - E_{i,j} - E_{j,i} + E_{j,j}) & 1 \end{pmatrix}.
\]
Note that $a \mapsto \Xi_j(a)$ and $a \mapsto \Xi_{i,j}(a)$ are group homomorphisms from the addtive group $\F$ to the multiplicative group $GL_{2n}(\F)$.  In particular, the inverse of each generator is another generator of the same type.  Thus, in terms of boundary data, contracting a spike of weight $w$ is equivalent to adding a spike of weight $-w$, and the same goes for boundary edges.

We define the {\bf electrical Grassmannian} $EG_n(\F)$ as the set of subspaces $\Lambda \subset \F^{2n}$ that are the boundary behaviors of some electrical network with $\partial V = [n]$.  Then by construction, $EL_n$ acts on $EG_n$ by applying the transformation $\Xi \in EL_n$ to the subspace $\Lambda \in EG_n$, which corresponds to a sequence of operations of adding boundary spikes or boundary edges.

\begin{remark*}
In \S \ref{sec:symplectic}, we will give a explicit characterization of $EL_n$ and $EG_n$ in terms of symplectic transformations and Lagrangian subspaces of $\F^{2n}$ using ideas from \cite{LP} and \cite{BF}.  The matrices $\Xi_j$ and $\Xi_{i,j}$ were written down in \cite{WJ}, but without any explanation of their significance in terms of boundary behavior.
\end{remark*}

\subsection{Elementary Factorizations and Scaffolds} \label{subsec:IOscaffolds}

As explained in \S \ref{subsec:gluing}, elementary factorizations and scaffolds are two ways of geometrically modeling the same process of harmonic continuation.  Here we give the conversion between the two frameworks, so that we can use the two tools interchangeably when convenient.  For an example, refer to Figure \ref{fig:scaffactorization}.

\begin{proposition} \label{prop:IOscaffolds}
Let $[G,i,j]: P \to Q$ and consider $G$ as a $\partial$-graph with $\partial V = i(P) \cup j(Q)$.  The following are equivalent:
\begin{enumerate}
	\item $[G,i,j]$ admits an elementary factorization.
	\item There is a scaffold on $S$ on $G$ such that $S_- \cap \partial V = i(P) \setminus j(Q)$ and $S_+ \cap \partial V = j(Q) \setminus i(P)$.
\end{enumerate}
One can always arrange that the input stubs are $V \setminus j(Q) \setminus S_-$, the output stubs are $V \setminus i(P) \setminus S_+$, and $S$ is precisely the set of edges in the type 1 networks oriented from input to output.
\end{proposition}

\begin{proof}
First, suppose $[G,i,j]$ has an elementary factorization into $[G_1,i_1,j_1]$, \dots , $[G_n,i_n,j_n]$.  Define the scaffold $S$ as the set of oriented edges that are in the type 1 networks, oriented from the input side to the output side.  To check this is a scaffold, we use the characterization in terms of a partial order and local comparison conditions from Lemma \ref{lem:scafequivalentdef}.  The partial order is defined by $e' \prec e$ if the elementary morphism that contains $e'$ comes before the elementary morphism that contains $e$ in the order of composition.  If $e \in S$ and $e'$ is incident to its starting endpoint $e_-$, then $e'$ must be in some elementary morphism before $e$, and similarly, if $e'$ is incident to $e_+$, then $e'$ must come after $e$, so the local comparison condition (A) in Lemma \ref{lem:scafequivalentdef} is satisfied.  Condition (B) is trivial since the $\partial$-graph is finite.  To prove condition (C) by contrapositive, suppose $p \in V^\circ \setminus S_+$ and $q \in V^\circ \setminus S_-$.  Then $p$ is forced to be an input stub in some $[G_k,i_k,j_k]$ and $q$ is an output stub in some $[G_\ell,i_\ell,j_\ell]$.  Then $k < \ell$.  Any edge incident to $p$ must come before step $k$ and any edge incident to $q$ must come after step $\ell$, so (C) holds.

Conversely, suppose that $G$ admits a scaffold $S$ satisfying (2).  Let $\prec$ be the induced partial order on $E / \bar{~}$. The elementary factorization is constructed, roughly speaking, by starting with a minimal edge and ``peeling off'' elementary morphisms corresponding to the edges in the Beginning and Middle, and next starting with the maximal edge and ``peeling off'' elementary morphisms corresponding to the edges in the End.  More precisely, the Beginning or Middle is nonempty, we first check if there is an isolated boundary vertex in $i(P) \setminus j(Q)$, and if there is, then we factor out a type 3 morphism from $[G,i,j]$.  If there are no such isolated vertices, then a minimal edge must be a boundary spike or boundary edge with boundary endpoints in $i(P)$ by similar reasoning as in Lemma \ref{lem:minimality}, and in this case, we can factor out a type 1 or type 2 morphism.  We repeat this inductively until the Beginning and Middle are empty, and then if the End is nonempty, we ``peel off'' layers from the End using maximal edges instead of minimal ones.  The reader may verify that this produces an elementary factorization as we claim.
\end{proof}

\begin{remark*}
When choosing an elementary factorization from a scaffold, one can always arrange that $e$ is in the Middle if and only if it comes after the type 3 networks and before the type 4 networks.  On the other hand, when constructing a scaffold from an elementary factorization, the Middle of the scaffold might be strictly larger than the ``middle'' of the factorization.
\end{remark*}

There are many scaffolds that do not fit the description in the last proposition.  The scaffolds on the last proposition always have $S_+ \cap S_- \cap \partial V = \varnothing$, but this is not true of all scaffolds.  If $f: G \to H$ is a UHM and $S$ is a scaffold on $H$ with $S_+ \cap S_- \cap \partial V(H) = 0$, then $f^{-1}(S)$ does not necessarily satisfy the same conditions, since a boundary vertex of $G$ might map to an interior vertex of $H$ which is in $S_+ \cap S_-$.  This is one reason why elementary factorizations cannot be pulled back functorially under UHMs in any nice way.

However, given any scaffold $S$, it is possible to construct one for which $S_+ \cap S_- \cap \partial V$ is empty by removing some edges from $S$.  In the next lemma, we will carry out this process while leaving any given edge in the Middle of the scaffold untouched.  Essentially the same proof works even when there are no edges in the Middle, but we leave this case to the reader.  Once $S_- \cap S_+ \cap \partial V = \varnothing$, then we can construct an elementary factorization as in the last Proposition, for properly defined $i, P, j, Q$.

\begin{lemma} \label{lem:scafdelete}
Suppose $S$ is a scaffold on $G$ with $e$ in the Middle.  Then we can obtain a scaffold $S'$ by deleting edges $\neq e$ such that $e$ is still in the Middle of $S'$ and $S_+' \cap S_-' \cap \partial V = \varnothing$.
\end{lemma}

\begin{proof}
Let $R = S_+ \cap S_- \cap \partial V$.  Each $p \in R$ is the output of one edge in $S$ and the input of one edge in $S$.  We will delete one of these two edges from $S$.  But in doing this we will make $V^\circ \setminus S_+$ or $V^\circ \setminus S_-$ larger, so we must make sure that none of the ``input stubs'' come after the ``output stubs'' and $e$ is still in the Middle.  Let $\prec$ be the partial order induced by $S$; then $\prec$ can be completed to a total order such that the edges in the Beginning come before those in the Middle which come before those in the End.  If $r \in R$ and $r = (e_1)_+ = (e_2)_-$, then $e_1 \prec e_2$.  Hence, either $e_1 \prec e$ or $e_2 \succ e$.  If $e_1 \prec e$, we delete $e_1$ from $S$, and otherwise we delete $e_2$ from $S$.  The new vertices in $V^\circ \setminus S_-'$ will be before $e$ and the new vertices in $V^\circ \setminus S_+'$ will be after $e$, so we are done.
\end{proof}

The preceding lemma and remarks yield the following proposition, which allows us to test recoverability by scaffolds using elementary factorizations instead of scaffolds:
\begin{proposition}
Let $G$ be a finite $\partial$-graph and $e \in E(G)$.  The following are equivalent:
\begin{itemize}
	\item There exists a scaffold $S$ with $e$ in the middle of $S$ and in $S \cup \overline{S}$ (resp.\ not in $S \cup \overline{S}$).
	\item There exists an elementary factorization of some morphism $[G,i,j]$ where $e$ is in a type 1 (resp.\ type 2) network that comes after the type 3 networks and before the type 4 networks.
\end{itemize}
\end{proposition}

\section{Layering $\partial$-Graphs on Surfaces} \label{sec:surfaces}

We have not yet given a general way of constructing scaffolds and elementary factorizations from scratch.  We have only constructed scaffolds from other scaffolds or converted between elementary factorizations and scaffolds.  This section will describe how to construct scaffolds and elementary factorizations for graphs on surfaces using medial graphs.  We outline an approach that potentially applies to general surfaces and execute it for critical circular planar $\partial$-graphs and $\partial$-graphs in the half-plane.

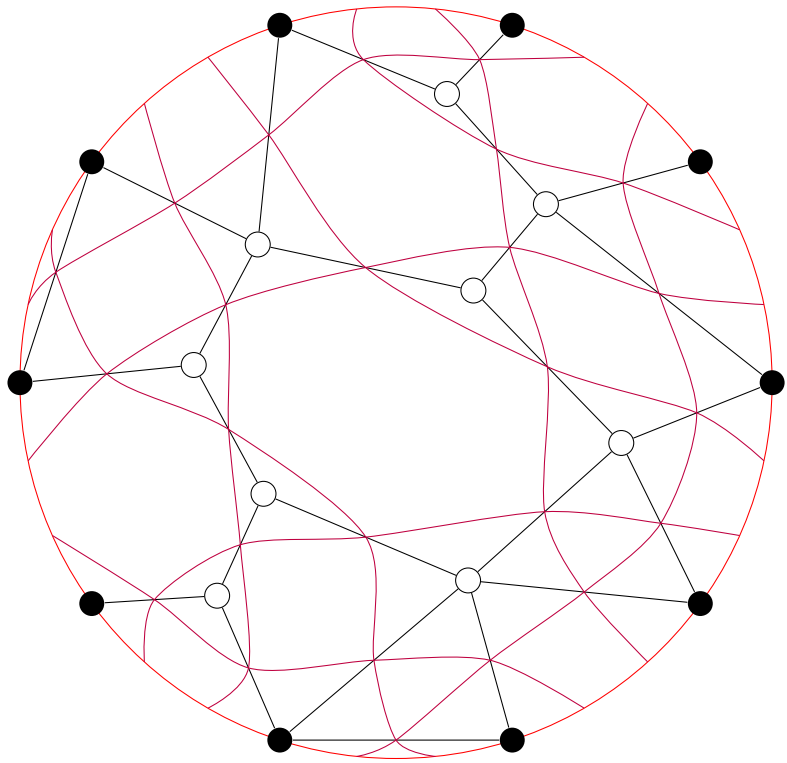
\begin{figure}
\caption{A lensless strand arrangement for a $\partial$-graph embedded on the disk.  The medial strands are purple.  As an exercise, color in all the cells which have vertices of $G$.} \label{fig:diskstrands}

\begin{center}
\begin{tikzpicture}[yscale=-1]
	%boundary circle
	\draw[red] (0,0) circle (5);
	
	%vertices
	\node[circle,fill] (A) at (0:5) {};
	\node[circle,fill] (B) at (36:5) {};
	\node[circle,fill] (C) at (72:5) {};
	\node[circle,fill] (D) at (108:5) {};
	\node[circle,fill] (E) at (144:5) {};
	\node[circle,fill] (F) at (180:5) {};
	\node[circle,fill] (G) at (216:5) {};
	\node[circle,fill] (H) at (252:5) {};
	\node[circle,fill] (I) at (288:5) {};
	\node[circle,fill] (J) at (324:5) {};
	\node[circle,draw] (K) at (15:3.1) {};
	\node[circle,draw] (L) at (70:2.8) {};
	\node[circle,draw] (M) at (130:3.7) {};
	\node[circle,draw] (N) at (140:2.3) {};
	\node[circle,draw] (O) at (185:2.7) {};
	\node[circle,draw] (P) at (225:2.6) {};
	\node[circle,draw] (Q) at (280:3.9) {};
	\node[circle,draw] (R) at (310:3.1) {};
	\node[circle,draw] (S) at (310:1.6) {};
	
	%edges
	\draw (A) to (K); \draw (A) to (R); \draw (B) to (K); \draw (B) to (L); \draw (C) to (D);
	\draw (C) to (L); \draw (D) to (L); \draw (D) to (M); \draw (E) to (M); \draw (F) to (G);
	\draw (F) to (O); \draw (G) to (P); \draw (H) to (P); \draw (H) to (Q); \draw (I) to (Q);
	\draw (J) to (R); \draw (K) to (L); \draw (K) to (S); \draw (L) to (N); \draw (M) to (N);
	\draw (N) to (O); \draw (O) to (P); \draw (P) to (S); \draw (Q) to (R); \draw (R) to (S);
	
	%medial strands
	\begin{scope}[purple]
		\draw plot[smooth,tension = 0.5] coordinates {(12:5) (barycentric cs:A=1,K=1) (barycentric cs:K=1,S=1) (barycentric cs:S=1,P=1) (barycentric cs:H=1,P=1) (240:5)};
		\draw plot[smooth,tension = 0.5] coordinates {(24:5) (barycentric cs:B=1,K=1) (barycentric cs:K=1,L=1) (barycentric cs:L=1,N=1) (barycentric cs:M=1,N=1) (barycentric cs:E=1,M=1) (132:5)};
		\draw plot[smooth,tension = 0.5] coordinates {(48:5) (barycentric cs:B=1,L=1) (barycentric cs:K=1,L=1) (barycentric cs:K=1,S=1) (barycentric cs:R=1,S=1) (barycentric cs:Q=1,R=1) (barycentric cs:I=1,Q=1) (276:5)};
		\draw plot[smooth,tension = 0.5] coordinates {(60:5) (barycentric cs:C=1,L=1) (barycentric cs:D=1,L=1) (barycentric cs:D=1,M=1) (barycentric cs:E=1,M=1) (156:5)};
		\draw plot[smooth,tension = 0.5] coordinates {(84:5) (barycentric cs:C=1,D=1) (barycentric cs:D=1,L=1) (barycentric cs:N=1,L=1) (barycentric cs:N=1,O=1) (barycentric cs:F=1,O=1) (barycentric cs:F=1,G=1) (204:5)};
		\draw plot[smooth,tension = 0.5] coordinates {(96:5) (barycentric cs:C=1,D=1) (barycentric cs:C=1,L=1) (barycentric cs:B=1,L=1) (barycentric cs:B=1,K=1) (barycentric cs:A=1,K=1) (barycentric cs:A=1,R=1) (barycentric cs:J=1,R=1) (312:5)};
		\draw plot[smooth,tension = 0.5] coordinates {(120:5) (barycentric cs:D=1,M=1) (barycentric cs:M=1,N=1) (barycentric cs:N=1,O=1) (barycentric cs:O=1,P=1) (barycentric cs:P=1,G=1) (228:5)};
		\draw plot[smooth,tension = 0.5] coordinates {(168:5) (barycentric cs:F=1,O=1) (barycentric cs:O=1,P=1) (barycentric cs:P=1,S=1) (barycentric cs:S=1,R=1) (barycentric cs:R=1,A=1) (348:5)};
		\draw plot[smooth,tension = 0.5] coordinates {(192:5) (barycentric cs:F=1,G=1) (barycentric cs:G=1,P=1) (barycentric cs:P=1,H=1) (barycentric cs:H=1,Q=1) (barycentric cs:Q=1,I=1) (300:5)};
		\draw plot[smooth,tension = 0.5] coordinates {(264:5) (barycentric cs:H=1,Q=1) (barycentric cs:Q=1,R=1) (barycentric cs:R=1,J=1)  (336:5)};
	\end{scope}
	
\end{tikzpicture}
\end{center}
\end{figure}

Any connected graph embedded in the disk has a \emph{medial graph}.  Medial graphs were invented by Steinitz and their construction is decribed in \cite{CIM}, \S 6 (for instance).  A medial graph on the disk is shown in Figure \ref{fig:diskstrands}.  The medial graph is an invaluable tool for studying circular planar networks and was central to the results of \cite{CIM}, \cite{dVGV}, \cite{WJ}, and \cite{IZ}.  This is not surprising since medial graphs are related to layer-stripping:  Examining Figure \ref{fig:diskstrands}, we can see that boundary edges and boundary spikes correspond to small boundary triangles in the medial graph, and contracting a spike or deleting a boundary edge corresponds to uncrossing the medial strands that meet at the empty boundary triangle (see \cite{CM}).

\subsection{Medial Strand Arrangements}

The construction of the medial graph in \cite{CM} works for connected graphs that are embedded nicely on a surface.  However, applying layer-stripping operations (or just passing to a subgraph) might easily produce a graph which is disconnected, has isolated bounary or interior vertices, or is not embedded as nicely.  To apply our theory of scaffolds and elementary factorizations without hiccups, we need medial graphs to make sense even in these degenerate cases.

Unfortunately, this means that ``the'' medial graph will no longer be well-defined.  But this is not really a problem.  We will be happy as long there is \emph{some} medial graph that we can manipulate.  Rather than \emph{constructing} the medial graph as in \cite{CM}, we will say what it means for a medial graph to be \emph{compatible} with a given embedded graph as in \cite{WJ}.  We will also make the definition work for infinite $\partial$-graphs in order to understand the infinite supercritical half-planar graphs of \cite{IZ}.

The medial graph depicted in Figure \ref{fig:diskstrands} can be viewed as an embedded graph where the interior vertices correspond to the edges of $G$ and each have degree $4$.  However, it will be more consistent with our later manipulations to view the medial graph as a \emph{collection of curves} where only two curves (or segments of curves) intersect at any point.  In general, the curves can be loops and are allowed to be infinite if $S$ is not compact.  We will formalize our requirements in the definition of ``strand arrangement'' below.

The compatibility between the medial strands and the graph is roughly described as follows:  As in Figure \ref{fig:diskstrands}, half of the medial cells contain vertices of $G$ and half of them do not.  If we color the cells with vertices black and the cells without vertices white, then two cells that share an edge have opposite colors.  The black cells are in bijective correspondence with the vertices of $G$, and the medial vertices are in bijective correspondence with the edges of $G$.

For simplicity and to rule out pathologies, we will assume that all our graph embeddings and medial strands are smooth.  We recall the following terminology and facts from basic manifold theory:
\begin{itemize}
	\item The \emph{smooth $1$-manifolds with boundary} are $[0,1]$, $S^1$, $[0,+\infty)$, and $\R$ (up to diffeomorphism).  The boundaries of $[0,1]$ and $[0,+\infty)$ are the sets of endpoints and the boundaries of $S^1$ and $\R$ are empty.
	\item If $J$ is a $1$-manifold with boundary and $S$ is a smooth $2$-manifold with boundary then $f: J \to M$ is an \emph{immersion} if and only if $f' \neq 0$.
	\item Suppose $J_1$ and $J_2$ are $1$-manifolds with boundary and $f_1: J_1 \to S$ and $f_2: J_2 \to S$ are smooth maps and that $f_1(t_1) = f_2(t_2) = x$.  Then the intersection is \emph{transversal} if and only if $f_1'(t_1)$ and $f_2'(t_2)$ are linearly independent.
	\item A continuous map $f: X \to Y$ between topological spaces is \emph{proper} if the preimage of a compact set is compact.
\end{itemize}

\begin{definition}
For any graph $G$, there is a corresponding topological space, the quotient space obtained from $E \times [0,1]$ by identifying $(e,t)$ with $(\overline{e},1-t)$ and identifying $(e,0)$ and $(e',0)$ if $e_+ = (e')_+$.  We will call this topological space $G$ as well since no confusion will result.
\end{definition}

\begin{definition}
A {\bf smooth embedding} of a $\partial$-graph on a surface with boundary $S$ is a proper continuous injective map $f: G \to S$ such that
\begin{itemize}
	\item $f(G) \cap \partial S = \partial V(G)$,
	\item $f |_e$ is a smooth immersion of $[0,1] \to S$,
	\item For each vertex $p$ and edges with $e_- = e_-' = p$, $(f|_{e})'(0)$ and $(f|_{e'})'(0)$ are not positive scalar multiples of each other.  That is, the edges all exit $p$ in different directions.
\end{itemize}
\end{definition}

\begin{remark*}
Note that we do not require the components of $S^\circ \setminus G$ are homeomorphic to the disk.  Thus, the embeddings may be rather degenerate.
\end{remark*}

\begin{definition}
A {\bf strand arrangement} on a surface with boundary $S$ is a collection of curves $\{s_\alpha\}$ in $S$ such that:
\begin{itemize}
	\item Each $s_\alpha$ admits a smooth parametrization $f_\alpha$ by a $1$-manifold with boundary $J_\alpha$, which is an immersion.
	\item For each $\alpha$, we have $f_\alpha^{-1}(\partial S) = \partial J_\alpha$.
	\item The map $F: \bigsqcup_\alpha J_\alpha \to S$ induced by $f_\alpha$ is proper.
	\item Only two segments of curves can intersect at any point and they cannot intersect on $\partial S$, that is, $\# F^{-1}(x) \leq 2$ and is $\# F^{-1}(x) \leq 1$ if $x \in \partial S$.
	\item The intersections between two curves and the intersections of a curve with itself are transversal.
\end{itemize}
The curves $s_\alpha$ are called {\bf strands}, the intersections points are called {\bf vertices}, and the components of $S^\circ \setminus \bigcup_\alpha s_\alpha$ are called {\bf cells}.  Note that the parametrizations $f_\alpha$ are injective except at the vertices and the cells are open in $S$.
\end{definition}

\begin{definition}
Note that the boundary of a cell $\mathcal{A}$ consists of segments of strands.  Two cells $\mathcal{A}$ and $\mathcal{B}$ are {\bf adjacent} if their boundaries share some nontrivial curve segment.  Given a strand arrangement $\{s_\alpha\}$, a {\bf two-coloring} of the cells is an assignment of ``white'' or ``black'' to each cell such that adjacent cells have opposite colors.
\end{definition}

\begin{definition}
{\bf A medial strand arrangment} $\mathcal{M}$ for a $\partial$-graph $G$ embedded in $S$ is a strand arrangement on $S$ with a two-coloring of the cells such that
\begin{itemize}
	\item Each black cell is homeomorphic to the open disk and has compact closure.\footnote{However, the closure might not be homeomorphic to the closed disk.}
	\item There is a bijection $p \mapsto \mathcal{A}_p$ from $V(G)$ to the black cells of $\mathcal{M}$ such that $p \in \mathcal{A}_p$.
	\item We have $p \in \partial V(G)$ if and only if $\overline{\mathcal{A}_p}$ intersects $\partial S$.
	\item There is a bijection $e \mapsto x_e$ from the unoriented edges of $G$ to the vertices of $\mathcal{M}$ such that $e \cap \mathcal{M} = \{x_e\}$.
	\item The edges of $G$ intersect the strands transversally.
\end{itemize}
\end{definition}

\begin{remark*}
It follows from the definition that an edge $e$ from a vertex $p$ to $q \neq p$ must exit $\mathcal{A}_p$ directly into $\mathcal{A}_q$ through $x_e$.  At $x_e$, $e$ must cross two strands (or possibly two parts of the same strand) as shown in Figure \ref{fig:localmedial}.  If $e$ is a self-loop, then $e$ goes from $\mathcal{A}_p$ into itself and this is the case where $\overline{\mathcal{A}_p}$ is not homeomorphic to the closed disk.
\end{remark*}

\begin{definition}
We say the medial strand arrangement is {\bf nondegenerate} if $\partial \mathcal{A}_p \cap \partial S$ is a single arc for each $p \in \partial V$ (rather than multiple distinct components).
\end{definition}

\begin{figure}

\caption{Compatibility between a graph and a medial strand arrangement near an edge $e$ from $p$ to $q$.} \label{fig:localmedial}

\begin{center}
\begin{tikzpicture}

\fill[black!40] (2,2) -- (2,-2) -- (0,0) -- (2,2);
\fill[black!40] (-2,2) -- (-2,-2) -- (0,0) -- (-2,2);
\draw[thick] (-2,0) -- (2,0);
\draw[purple] (-2,-2) -- (2,2);
\draw[purple] (-2,2) -- (2,-2);
\fill (0,0) circle (0.1);
\node[circle,draw,fill=white] at (-2,0) {$p$};
\node[circle,draw,fill=white] at (2,0) {$q$};

\node at (-1.5,0.5) {$\mathcal{A}_p$};
\node at (1.5,0.5) {$\mathcal{A}_q$};
\node at (0,-0.5) {$x_e$};

\end{tikzpicture}
\end{center}

\end{figure}
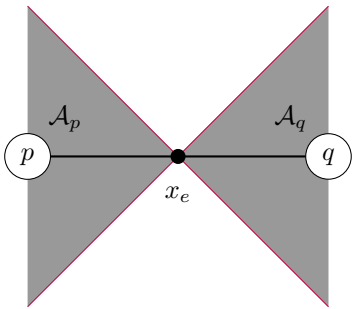

\subsection{Producing Scaffolds from the Medial Strands}

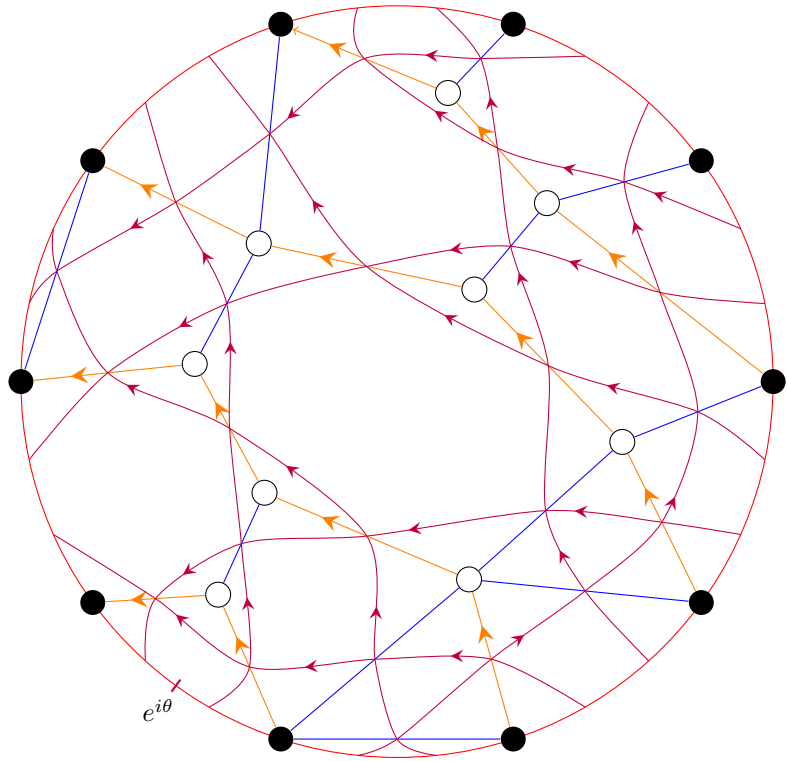
\begin{figure}

\caption{Scaffold produced by orienting medial strands in the disk.  Medial strands are purple.  The edges in $S$ are orange.} \label{fig:diskscaffold}

\begin{center}
\begin{tikzpicture}[yscale=-1]
	%boundary circle
	\draw[red] (0,0) circle (5);
	
	%vertices
	\node[circle,fill] (A) at (0:5) {};
	\node[circle,fill] (B) at (36:5) {};
	\node[circle,fill] (C) at (72:5) {};
	\node[circle,fill] (D) at (108:5) {};
	\node[circle,fill] (E) at (144:5) {};
	\node[circle,fill] (F) at (180:5) {};
	\node[circle,fill] (G) at (216:5) {};
	\node[circle,fill] (H) at (252:5) {};
	\node[circle,fill] (I) at (288:5) {};
	\node[circle,fill] (J) at (324:5) {};
	\node[circle,draw] (K) at (15:3.1) {};
	\node[circle,draw] (L) at (70:2.8) {};
	\node[circle,draw] (M) at (130:3.7) {};
	\node[circle,draw] (N) at (140:2.3) {};
	\node[circle,draw] (O) at (185:2.7) {};
	\node[circle,draw] (P) at (225:2.6) {};
	\node[circle,draw] (Q) at (280:3.9) {};
	\node[circle,draw] (R) at (310:3.1) {};
	\node[circle,draw] (S) at (310:1.6) {};
	
	%edges
	\draw[blue] (A) to (K);
	\draw[orange] (A) to (R) [arrow inside={end=stealth,opt={orange, scale=2}}{0.75}];
	\draw[orange] (B) to (K) [arrow inside={end=stealth,opt={orange, scale=2}}{0.75}];
	\draw[blue] (B) to (L);
	\draw[blue] (C) to (D);
	\draw[orange] (C) to (L) [arrow inside={end=stealth,opt={orange, scale=2}}{0.75}];
	\draw[blue] (D) to (L);
	\draw[orange] (D) to (M) [arrow inside={end=stealth,opt={orange, scale=2}}{0.75}];
	\draw[orange] (M) to (E) [arrow inside={end=stealth,opt={orange, scale=2}}{0.75}];
	\draw[blue] (F) to (G);
	\draw[orange] (O) to (F) [arrow inside={end=stealth,opt={orange, scale=2}}{0.75}];
	\draw[orange] (P) to (G) [arrow inside={end=stealth,opt={orange, scale=2}}{0.75}];
	\draw[blue] (H) to (P);
	\draw[->,orange] (Q) to (H) [arrow inside={end=stealth,opt={orange, scale=2}}{0.75}];
	\draw[blue] (I) to (Q);
	\draw[blue] (J) to (R);
	\draw[blue] (K) to (L);
	\draw[orange] (K) to (S) [arrow inside={end=stealth,opt={orange, scale=2}}{0.75}];
	\draw[orange] (L) to (N) [arrow inside={end=stealth,opt={orange, scale=2}}{0.75}];
	\draw[blue] (M) to (N);
	\draw[orange] (N) to (O) [arrow inside={end=stealth,opt={orange, scale=2}}{0.75}];
	\draw[blue] (O) to (P);
	\draw[orange] (S) to (P) [arrow inside={end=stealth,opt={orange, scale=2}}{0.75}];
	\draw[orange] (R) to (Q) [arrow inside={end=stealth,opt={orange, scale=2}}{0.75}];
	\draw[blue] (R) to (S);
	
	%cut points
	\node at (126:5.4) {$e^{i\theta}$};
	\draw[purple,thick] (126:5.1) -- (126:4.9);
	
	%medial strands
	\begin{scope}[purple]
		\draw plot[smooth,tension = 0.5] coordinates {(12:5) (barycentric cs:A=1,K=1) (barycentric cs:K=1,S=1) (barycentric cs:S=1,P=1) (barycentric cs:H=1,P=1) (240:5)} [arrow inside={end=stealth,opt={purple,scale=1.5}}{0.25,0.5,0.75}];
		\draw plot[smooth,tension = 0.5] coordinates {(24:5) (barycentric cs:B=1,K=1) (barycentric cs:K=1,L=1) (barycentric cs:L=1,N=1) (barycentric cs:M=1,N=1) (barycentric cs:E=1,M=1) (132:5)} [arrow inside={end=stealth,opt={purple,scale=1.5}}{0.25,0.5,0.85}];
		\draw plot[smooth,tension = 0.5] coordinates {(48:5) (barycentric cs:B=1,L=1) (barycentric cs:K=1,L=1) (barycentric cs:K=1,S=1) (barycentric cs:R=1,S=1) (barycentric cs:Q=1,R=1) (barycentric cs:I=1,Q=1) (276:5)} [arrow inside={end=stealth,opt={purple,scale=1.5}}{0.2,0.6,0.85}];
		\draw plot[smooth,tension = 0.5] coordinates {(60:5) (barycentric cs:C=1,L=1) (barycentric cs:D=1,L=1) (barycentric cs:D=1,M=1) (barycentric cs:E=1,M=1) (156:5)} [arrow inside={end=stealth,opt={purple,scale=1.5}}{0.25,0.5,0.75}];
		\draw plot[smooth,tension = 0.5] coordinates {(84:5) (barycentric cs:C=1,D=1) (barycentric cs:D=1,L=1) (barycentric cs:N=1,L=1) (barycentric cs:N=1,O=1) (barycentric cs:F=1,O=1) (barycentric cs:F=1,G=1) (204:5)} [arrow inside={end=stealth,opt={purple,scale=1.5}}{0.25,0.5,0.75}];
		\draw plot[smooth,tension = 0.5] coordinates {(96:5) (barycentric cs:C=1,D=1) (barycentric cs:C=1,L=1) (barycentric cs:B=1,L=1) (barycentric cs:B=1,K=1) (barycentric cs:A=1,K=1) (barycentric cs:A=1,R=1) (barycentric cs:J=1,R=1) (312:5)} [arrow inside={end=stealth,opt={purple,scale=1.5}}{0.25,0.5,0.85}];
		\draw plot[smooth,tension = 0.5] coordinates {(120:5) (barycentric cs:D=1,M=1) (barycentric cs:M=1,N=1) (barycentric cs:N=1,O=1) (barycentric cs:O=1,P=1) (barycentric cs:P=1,G=1) (228:5)} [arrow inside={end=stealth,opt={purple,scale=1.5}}{0.2,0.6,0.75}];
		\draw plot[smooth,tension = 0.5] coordinates {(348:5) (barycentric cs:R=1,A=1) (barycentric cs:S=1,R=1) (barycentric cs:P=1,S=1) (barycentric cs:O=1,P=1) (barycentric cs:F=1,O=1) (168:5)} [arrow inside={end=stealth,opt={purple,scale=1.5}}{0.25,0.4,0.75}];
		\draw plot[smooth,tension = 0.5] coordinates {(300:5) (barycentric cs:Q=1,I=1) (barycentric cs:H=1,Q=1) (barycentric cs:P=1,H=1) (barycentric cs:G=1,P=1) (barycentric cs:F=1,G=1) (192:5)} [arrow inside={end=stealth,opt={purple,scale=1.5}}{0.25,0.5,0.8}];
		\draw plot[smooth,tension = 0.5] coordinates {(336:5) (barycentric cs:R=1,J=1) (barycentric cs:Q=1,R=1) (barycentric cs:H=1,Q=1) (264:5)} [arrow inside={end=stealth,opt={purple,scale=1.5}}{0.2,0.4,0.7}];
	\end{scope}
	
\end{tikzpicture}
\end{center}
\end{figure}

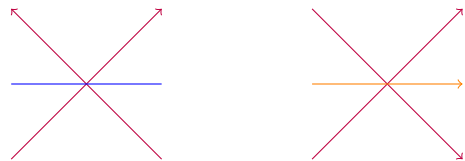
\begin{figure}

\caption{Assignment of oriented edges of a scaffold based on orientations of the medial strands.} \label{fig:assignment}

\begin{center}
	\begin{tikzpicture}
		\draw[->,purple] (-1,-1) -- (1,1);
		\draw[->,purple] (1,-1) -- (-1,1);
		\draw[blue] (-1,0) -- (1,0);
		
		\begin{scope}[shift = {(4,0)}]
			\draw[->,purple] (-1,-1) -- (1,1);
			\draw[->,purple] (-1,1) -- (1,-1);
			\draw[->,orange] (-1,0) -- (1,0);	
		\end{scope}
	\end{tikzpicture}
\end{center}

\end{figure}

We now describe how to produce scaffolds using medial strands (see Figure \ref{fig:diskscaffold}).  We will assign an orientation for each strand and then choose the edges in $S$ as in Figure \ref{fig:assignment}.  Since we assumed that the edges intersect the strands transversally, the assignments of edges in $S$ are determined by looking at any circular ordering of the tangent vectors to the curves at $x_e$.  The partial order on $E / \bar{~}$ induced by the scaffold will correspond to the partial order on the medial vertices defined by $x \prec y$ if there is a positively oriented path from $x$ to $y$ along the medial strands.

Not all orientations of the medial strands will produce a scaffold, of course.  We will describe sufficient conditions and prove their validity using Lemma \ref{lem:scafequivalentdef} which characterizes scaffolds in terms of local comparison conditions and a partial order.

\begin{definition}
A {\bf piecewise orientation of $\partial S$} is a division of $\partial S$ into arcs and an assigned orientation for each arc, such that any compact set intersects only finitely many arcs.  An {\bf orientation} of a strand arrangement $\{s_\alpha\}$ is an assigned orientation for each strand $s_\alpha$ specifying whether the parametrization $f_\alpha$ is orientation-preserving or orientation-reversing, together with a piecewise orientation of $\partial S$.
\end{definition}

\begin{definition}
For an orientation $\mathcal{O}$, let $T_\mathcal{O}$ be the set of oriented edges which intersect the medial strands like the oriented orange edge in Figure \ref{fig:assignment}.
\end{definition}

\begin{definition}
For an orientation $\mathcal{O}$, we can define a relation $\prec = \prec_{\mathcal{O}}$ on the medial vertices by $x \prec y$ if there is a positively oriented path from $x$ to $y$ along segments of medial strands and arcs of $\partial S$.  Note that $\prec$ is automatically transitive and it defines a partial order if and only if there are no positively oriented loops in $\mathcal{M}$.  In this case, we say $\mathcal{M}$ is {\bf acyclic}.
\end{definition}

\begin{definition}
Assume $G$ has no self-looping edges.  Suppose $\mathcal{O}$ is an orientation of a medial strand arrangement for $G$ and $\mathcal{A}_p$ is a black cell.  Then $\partial \mathcal{A}_p$ can be bijectively parametrized by $S^1$ and hence given two different orientations.  We say an orientation $\mathcal{O}$ has the {\bf Desired Behavior at } $\mathcal{A}_p$ if $\partial \mathcal{A}_p$ can be oriented and partitioned into two arcs $I$ and $J$, such that $I$ is positively oriented with respect to $\mathcal{O}$ and $J$ is negatively oriented with respect to $\mathcal{O}$.  The Desired Behavior at an interior vertex is shown in Figure \ref{fig:desiredbehavior}.  Note that the definition is independent of which orientation of $\partial \mathcal{A}_p$ is chosen, and thus it makes sense even if the surface is non-orientable.
\end{definition}

\begin{figure}

\caption{Desired Behavior of oriented medial strands on the boundary of medial black cell containing an interior vertex.  Let's orient $\partial \mathcal{A}_p$ counterclockwise.  If the arc $I$ consists of the two sides on the right of the cell and the arc $J$ consists of the three sides on the left, then $I$ is positively oriented and $J$ is negatively oriented with respect to $\mathcal{O}$.} \label{fig:desiredbehavior}

\begin{center}
\begin{tikzpicture}

\fill[black!20] (0:1.2) -- (72:1.2) -- (144:1.2) -- (216:1.2) -- (288:1.2) -- (0:1.2);

\draw[blue] (0,0) -- (0:2);
\draw[orange] (0,0) to (72:2) [arrow inside={end=stealth,opt={orange, scale=2}}{0.35}];
\draw[blue] (0,0) -- (144:2);
\draw[blue] (0,0) -- (216:2);
\draw[orange] (-72:2) to (0,0) [arrow inside={end=stealth,opt={orange, scale=2}}{0.75}];

\draw[fill=white] (0,0) circle (0.15);

\draw[purple] (-24:2) to (96:2) [arrow inside={end=stealth,opt={purple, scale=1.5}}{0.5}];
\draw[purple] (168:2) to (48:2) [arrow inside={end=stealth,opt={purple, scale=1.5}}{0.5}];
\draw[purple] (240:2) to (120:2) [arrow inside={end=stealth,opt={purple, scale=1.5}}{0.5}];
\draw[purple] (312:2) to (192:2) [arrow inside={end=stealth,opt={purple, scale=1.5}}{0.5}];
\draw[purple] (-96:2) to (24:2) [arrow inside={end=stealth,opt={purple, scale=1.5}}{0.5}];

\end{tikzpicture}
\end{center}
\end{figure}
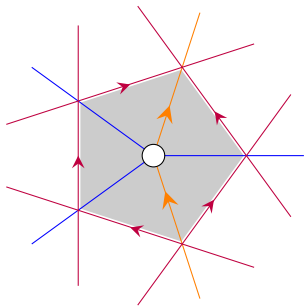

\begin{lemma} \label{lem:orientationscaffold}
Let $G$ be a graph smoothly embedded on $S$, let $\mathcal{M}$ be a medial strand arrangement, and let $\mathcal{O}$ be an orientation of $\mathcal{M}$.  Suppose that
\begin{enumerate}[a.]
	\item $\mathcal{M}$ is acyclic.
	\item There is no negatively oriented path in $\mathcal{M}$ with a proper continuous parametrization by $[0,+\infty)$.
	\item $\mathcal{O}$ has the Desired Behavior at each medial black cell.
\end{enumerate}
Then the set $T_{\mathcal{O}}$ is a scaffold; moreover, every edge is in $\Mid T_{\mathcal{O}}$.
\end{lemma}

\begin{proof}
We will use Lemma \ref{lem:scafequivalentdef}.  Note that (a) implies $\prec = \prec_{\mathcal{O}}$ defines a partial order on the medial vertices, hence a partial order on $E / \bar{~}$.  Next, (b) implies that there is no infinite decreasing chain $x_0 \succ x_1 \succ \dots$, and hence every subset of $E / \bar{~}$ has a minimal element, which verifies (B) of Lemma \ref{lem:scafequivalentdef}.

Examining Figure \ref{fig:desiredbehavior} and the definition of Desired Behavior, we can see that each interior vertex $p$ has exactly one edge $e$ in $T_\mathcal{O}$ entering it and one $e'$ exiting it | namely, the edges corresponding to the two medial vertices which divide $\partial \mathcal{A}_p$ into the two arcs $I$ and $J$.  Following the oriented arcs $I$ and $J$ shows that all the other edges incident to $p$ are between $e$ and $e'$ with respect to $\prec$, and hence the local comparison conditions (A) of Lemma \ref{lem:scafequivalentdef} are satisfied.  A similar argument verifies that (A) holds for $p \in \partial V$.  Finally, (C) is trivial since every interior vertex is both an input and an output of some edge in $T_\mathcal{O}$.

Therefore, $T_{\mathcal{O}}$ is a scaffold and every vertex and edge is in the Middle of $T_{\mathcal{O}}$.
\end{proof}

\begin{remark*}
There is no reason that we could not divide the strands into segments and give a different orientation to each segment, so long as the segment divisions do not fall on medial vertices.  In this case, the division into two arcs for the Desired Behavior might not fall on a medial vertex, and this will mean that some vertices are not the inputs or outputs of edges in $T_{\mathcal{O}}$.  One has to do more work to determine when this defines a scaffold and what the Middle is.  But this approach is potentially more flexible and adaptable to general surfaces, though we will not need it for the disk.
\end{remark*}

\subsection{Scaffolds for Critical Circular Planar $\partial$-Graphs} \label{subsec:circularscaf}

Lemma \ref{lem:orientationscaffold} provides a strategy for constructing scaffolds, which we will now implement for critical circular planar $\partial$-graphs.  We will show that such $\partial$-graphs are totally layerable, hence recoverable, and thus reprove one of the main results of \cite{CIM} and \cite{dVGV}.

\begin{definition}
A {\bf circular planar $\partial$-graph} is a $\partial$-graph embedded in the surface $\D = \{|z| \leq 1\} \subset \C$ with boundary $\partial \D = \{|z| = 1\}$.
\end{definition}

\begin{definition}
We say a strand arrangement $\mathcal{M}$ is {\bf lensless} if none of the strands intersects itself or forms a loop, and no two strands intersect each other more than once.  Note that for $\D$ or any compact surface with boundary, this implies that each strand is parametrized by $[0,1]$ and has two endpoints on the boundary.
\end{definition}

\begin{definition}
If $G$ is circular planar and admits a lensless strand arrangement, then we say $G$ is {\bf critical}.
\end{definition}

\begin{definition}
Let $\mathcal{M}$ be a lensless strand arrangement in $\D$, and suppose $e^{i\theta} \in \partial \D$ is not the endpoint of any strand.  We define $\mathcal{O}_\theta$ as follows:  If $s$ is a strand with endpoints $e^{ia}$ and $e^{ib}$ such that $\theta < a < b < \theta + 2\pi$, then the positive orientation of $s$ moves from $e^{ia}$ to $e^{ib}$.  We choose an interval $I$ of $\partial \D$ that does not contain any endpoints of strands and orient $\partial \D \setminus I$ counterclockwise and $I$ clockwise.  (See Figure \ref{fig:diskstrands}.)
\end{definition}

\begin{lemma} \label{lem:diskacyclic}
If $\mathcal{M}$ is a lensless strand arrangement in $\D$, then $\mathcal{O}_\theta$ is an acyclic orientation.
\end{lemma}

\begin{proof}
The proof is by induction on the number of strands.  It clearly holds for one strand.  Suppose it holds for $n - 1$ strands and consider $n$ strands $s_1, \dots, s_n$ with endpoints $e^{i \alpha_j}$ and $e^{i \beta_j}$ with $\theta < \alpha_j < \beta_j < \theta + 2\pi$.  Without loss of generality, $\alpha_n = \min(\alpha_j)$.

From the Jordan curve theorem, we know that $\D \setminus s_j$ has two components, one on the left of $s_j$ and one on the right of $s_j$.  Since the strand arrangement is lensless, $s_j$ can only cross $s_k$ in one direction and the direction can be detected from the positions of the start and end points of $s_j$ and $s_k$ on $\partial \D$.  For any $j \neq n$, we have $\theta < \alpha_n < \alpha_j$, and this implies that $s_j$ either does not cross $s_n$ or $s_j$ crosses $s_n$ from right to left.  Thus, there is no strand that crosses $s_n$ from left to right.

From the induction hypothesis, $s_1, \dots, s_{n-1}$ and the boundary circle do not form any oriented loops.  Thus, if a loop exists it must contain some segment of $s_n$ and clearly it cannot be entirely contained in $s_n$.  Thus, the loop must exit $s_n$ at some point.  After that, it must move into the left component of $\D \setminus s_n$ because no strand crosses $s_n$ from left to right.  But then at some point the loop must return to (or cross) $s_n$ from the left component of $\D \setminus s_n$ at some point $x$.

Since no strand crosses $s_n$ from left to right, the only possibility is that the part of the path before $x$ was part of the oriented boundary.  The boundary crosses $s_n$ from left to right only at the start point $e^{i \alpha_n}$ of $s_n$.  There are no endpoints of strands between $e^{i\theta}$ and $e^{i \alpha_n}$, and we chose an interval $I$ around $e^{i\theta}$ which is oriented clockwise.  Thus, there is no way the loop could have entered the counterclockwise oriented segment between $I$ and $e^{i \alpha_n}$.  This causes a contradiction, so there is no loop.
\end{proof}

\begin{lemma} \label{lem:cellboundary}
Let $\mathcal{A}$ be a cell of a lensless strand arrangement on $\D$.  Let $s_1, \dots, s_n$ be the strands that intersect $\partial \mathcal{A}$, listed in CCW order around $\partial \mathcal{A}$ and oriented in the same direction as the CCW orientation of $\partial \mathcal{A}$ (with $\mathcal{A}$ on the left of each $s_j$).  Let $x_j$ and $y_j$ be respectively the start and end of $s_j$.  Then $x_1, \dots, x_n$ occur in CCW order around $\partial \D$, and so do $y_1, \dots, y_n$.
\end{lemma}

\begin{remark*}
We do not assume in the hypothesis that $s_1, \dots, s_n$ are distinct, although that turns out to be true.
\end{remark*}

\begin{proof}
Note that for a \emph{lensless} strand arrangement on the disk, each medial cell is bounded by a Jordan curve formed by segments of the strands (as can be proved using the Jordan curve theorem and induction on the number of strands).  Hence, the boundary of the cell has a well-defined counterclockwise orientation.

Suppose $\mathcal{A}$ is an interior cell.  Let $z$ be the vertex of $\partial \mathcal{A}$ where $s_1$ and $s_2$ intersect.  Let $C$ be the counterclockwise arc of $\partial D$ from $x_1$ to $x_2$.  Let $h_1$ and $h_2$ be the arcs of $s_1$ and $s_2$ from $x_1$ and $x_2$ to $z$, so that $C$, $h_1$, and $h_2$ bound a triangle $T$.

Suppose for contradiction that there is some other $x_j \in C$.  Let $w$ be the first point where $s_j$ hits $\partial T$.  If $w \in h_2$, then $s_j$ crosses $s_2$ there from left to right.  It cannot intersect $s_2$ again since $\mathcal{M}$ is lensless, but that implies it cannot intersect $\partial \mathcal{A}$ because $\mathcal{A}$ is on the left side of $s_2$.  So suppose $w \in h_1$.  Then at $w$, $s_j$ crosses from the left to the right side of $s_2$, and this occurs before the point $z$ along $s_2$, which implies $z \in \partial \mathcal{A}$ is on the right side of $s_j$.  This also is impossible because $\mathcal{A}$ is supposed to be on the left side of $s_j$.

This contradiction proves that there is no $x_j$ between $x_1$ and $x_2$, and the same argument applies to $x_k$ and $x_{k+1}$ for all $k$, hence $x_1, \dots, x_n$ occur in counterclockwise order.  By a symmetrical argument, $y_1, \dots, y_n$ occur in counterclockwise order.  In the case of a boundary cell, similar reasoning applies except that arcs of $\partial \D$ may intervene between the strand segments; details left to the reader.
\end{proof}

\begin{lemma} \label{lem:diskDB}
Let $\mathcal{M}$ be a lensless strand arrangement on $\D$.  Then $\mathcal{O}_\theta$ has the Desired Behavior at each medial cell $\mathcal{A}$.
\end{lemma}

\begin{proof}
Consider an interior medial cell $\mathcal{A}$.  Let $s_j$ and $x_j$ and $y_j$ as in Lemma \ref{lem:cellboundary}.  Suppose that $x_j = e^{ia_j}$ and $y_j = e^{ib_j}$.  We can assume without loss of generality that $\theta < a_1 < a_2 < \dots < a_n < \theta + 2\pi$, that $b_1 < \dots < b_n < b_1 + 2\pi$, and that $a_j < b_j < a_j + 2\pi$.  Then whenever $b_j < \theta + 2\pi$, the orientation of $s_j$ given by $\mathcal{O}_\theta$ matches the CCW orientation of $\partial \mathcal{A}$, and whenever $\theta + 2 \pi < b_j$, the orientations are opposite.  Let $k$ be the last index with $b_k < \theta + 2\pi$.  Then $\partial \mathcal{A}$ can be divided into two arcs
\[
\partial \mathcal{A} \cap (s_1 \cup \dots \cup s_k), \qquad \partial \mathcal{A} \cap (s_{k+1} \cup \dots \cup s_n),
\]
such that $\mathcal{O}_\theta$ orients the first arc CCW around $\mathcal{A}$ and the second CW.  This shows that the strands that bound $\mathcal{A}$ have the Desired Behavior.

The case of a boundary cell is similar and follows from casework (which is easier if the medial graph is nondegenerate, but works in the general case).
\end{proof}

It now follows from Lemmas \ref{lem:diskacyclic} and \ref{lem:diskDB} together with Lemma \ref{lem:orientationscaffold} that

\begin{proposition}
Suppose $G$ is a $\partial$-graph on $\D$ with a lensless medial strand arrangement $\mathcal{M}$.  Then $\mathcal{O}_\theta$ defines a scaffold where all edges are in the Middle.
\end{proposition}

\begin{theorem}[cf.\ \cite{CIM} Theorem 2, \cite{dVGV}, and \cite{WJ} Theorem 6.7] \label{thm:CCPTL}
Any critical circular planar $\partial$-graph is totally layerable, hence recoverable by scaffolds, and recoverable over any field $\F$.
\end{theorem}

\begin{proof}
Let $e$ be any edge and let $x$ be the corresponding medial vertex, and $s_1$ and $s_2$ the strands that meet there.  Note $s_1$ and $s_2$ divide $\D$ into four components, and $e$ is contained in two opposite components.  If $e^{i\theta}$ is on the boundary of one of the components that contains $e$, then $e$ is an edge which is not in the scaffold $T_{\mathcal{O}_\theta}$, and if $e^{i\theta}$ is on the boundary of one of the other components, then $e$ is in the scaffold.  In either case, $e \in \Mid T_{\mathcal{O}_\theta}$ since all edges are in the Middle.
\end{proof}

Thanks to the general setup of \S \ref{sec:scaffolds}, we also know that

\begin{corollary}
If $f: G \to H$ is a UHM and $H$ is circular planar, then $G$ is recoverable by scaffolds, hence recoverable for any field $\F$.
\end{corollary}

\subsection{Embedded Subgraph Partitions and Elementary Factorizations}

Embeddings and medial strand arrangements provide a way of constructing $\partial$-subgraph partitions of a graph.  Indeed, we can use a collection of curves to cut the surface $S$ into smaller surfaces $S_\alpha$ with piecewise smooth boundary.  Each $S_\alpha$ will correspond to a subgraph $G_\alpha$ of $G$ whose vertices are given by pieces of medial cells in $S_\alpha$ and whose edges are given by the medial vertices in $S_\alpha$.

More precisely, suppose $G$ is a $\partial$-graph embedded on $S$ with medial strand arrangement $\mathcal{M}$.  Let $\mathcal{C}$ be another strand arrangement such that $\mathcal{C} \cup \mathcal{M}$ also forms a strand arrangement.  Let $\{S_\alpha\}$ be the components of $S \setminus \mathcal{C}$.  Assume that for each medial cell $\mathcal{A}$, $\mathcal{A} \cap S_\alpha$ is homeomorphic to a disk.  Then we can define a subgraph $G_\alpha$ of $G$ as follows:
\begin{itemize}
	\item The vertices of $G_\alpha$ are the vertices of $G$ whose medial cells intersect $S_\alpha$.
	\item The edges of $G_\alpha$ are the edges of $G$ whose medial vertices are contained in $S_\alpha^\circ$.
	\item A vertex of $G_\alpha$ is interior if and only if its medial cell is contained in $S_\alpha^\circ$.
\end{itemize}
Then the $G_\alpha$'s form a $\partial$-subgraph partition of $G$ (exercise), and we say that it is an {\bf embedded subgraph partition}.

\begin{remark*}
$G_\alpha$ can be embedded in $S_\alpha$ with medial strand arrangement $\mathcal{M} \cap S_\alpha$.  However, this requires a perturbation of the original embedding of $G$, and it might not be possible to achieve this for all $G_\alpha$'s simultaneously.  Thus, it is better to build our geometric intuition on what happens to the medial cells rather than the original vertices of $G$.
\end{remark*}

An embedding also provides a way to assign input and output vertices to make $G$ into an IO-graph morphism.  Take a partition of $\partial S$ into two sets $D_1$ and $D_2$ (for instance, two arcs of the boundary of a disk).  Then declare $p \in \partial V(G)$ to be input if the closure of its medial cell intersects $D_1$ and output if the closure of its medial cell intersects $D_2$.

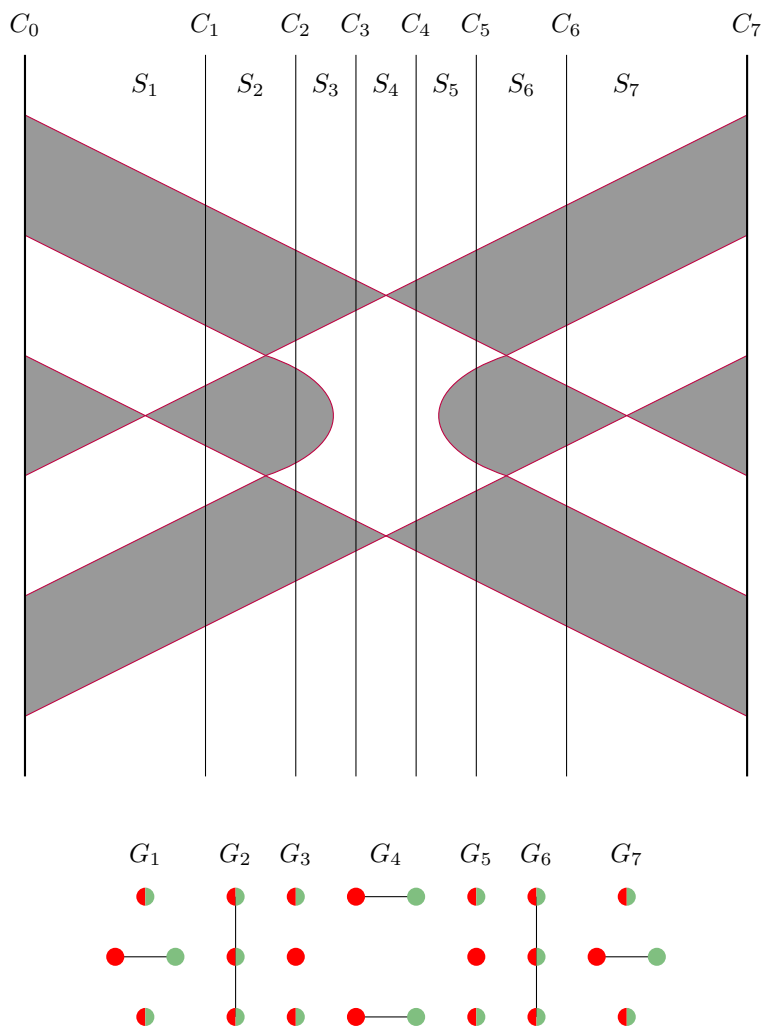
\begin{figure}

\caption{An elementary factorization for a graph embedded in a strip.  This can be viewed as a factorization for a circular planar $\partial$-graph if we identify the top ends of the vertical lines to a point and the bottom ends to another point.} \label{fig:embeddedfactorization}

\begin{center}
\begin{tikzpicture}[scale=0.8]
	
	\draw[draw=purple,fill=black!40] (-6,-1) -- (-6,1) -- (-4,0) -- (-6,-1);
	\draw[draw=purple,fill=black!40] (-4,0) -- (-2,-1) .. controls (-0.5,-0.5) and (-0.5,0.5) .. (-2,1) -- (-4,0);
	\draw[draw=purple,fill=black!40] (-2,-1) -- (0,-2) -- (-6,-5) -- (-6,-3) -- (-2,-1);
	\draw[draw=purple,fill=black!40] (-2,1) -- (0,2) -- (-6,5) -- (-6,3) -- (-2,1);
	\draw[draw=purple,fill=black!40] (6,-1) -- (6,1) -- (4,0) -- (6,-1);
	\draw[draw=purple,fill=black!40] (4,0) -- (2,-1) .. controls (0.5,-0.5) and (0.5,0.5) .. (2,1) -- (4,0);
	\draw[draw=purple,fill=black!40] (2,-1) -- (0,-2) -- (6,-5) -- (6,-3) -- (2,-1);
	\draw[draw=purple,fill=black!40] (2,1) -- (0,2) -- (6,5) -- (6,3) -- (2,1);
	
	\draw[thick] (-6,-6) -- (-6,6);
	\draw[thick] (6,-6) -- (6,6);
	
	\node at (-6,6.5) {$C_0$};
	\node at (-4,5.5) {$S_1$};
	\draw (-3,-6) -- (-3,6);
	\node at (-3,6.5) {$C_1$};
	\node at (-2.25,5.5) {$S_2$};
	\draw (-1.5,-6) -- (-1.5,6);
	\node at (-1.5,6.5) {$C_2$};
	\node at (-1,5.5) {$S_3$};
	\draw (-0.5,-6) -- (-0.5,6);
	\node at (-0.5,6.5) {$C_3$};
	\node at (0,5.5) {$S_4$};
	\draw (0.5,-6) -- (0.5,6);
	\node at (0.5,6.5) {$C_4$};
	\node at (1,5.5) {$S_5$};
	\draw (1.5,-6) -- (1.5,6);
	\node at (1.5,6.5) {$C_5$};
	\node at (2.25,5.5) {$S_6$};
	\draw (3,-6) -- (3,6);
	\node at (3,6.5) {$C_6$};
	\node at (4,5.5) {$S_7$};
	\node at (6,6.5) {$C_7$};

	\begin{scope}[shift = {(0,-9)}]
	
		\node at (-4,1.7) {$G_1$};
		\draw (-4.5,0) -- (-3.5,0);
		\rgvertex{-4}{1}
		\rgvertex{-4}{-1}
		\rvertex{-4.5}{0}
		\gvertex{-3.5}{0}
	
		\node at (-2.5,1.7) {$G_2$};
		\draw (-2.5,1) -- (-2.5,-1);
		\rgvertex{-2.5}{1}
		\rgvertex{-2.5}{0}
		\rgvertex{-2.5}{-1}
	
		\node at (-1.5,1.7) {$G_3$};
		\rgvertex{-1.5}{1}
		\rvertex{-1.5}{0}
		\rgvertex{-1.5}{-1}
	
		\node at (0,1.7) {$G_4$};
		\draw (-0.5,1) -- (0.5,1);
		\draw (-0.5,-1) -- (0.5,-1);
		\rvertex{-0.5}{1}
		\rvertex{-0.5}{-1}
		\gvertex{0.5}{1}
		\gvertex{0.5}{-1}
		
		\node at (1.5,1.7) {$G_5$};
		\rgvertex{1.5}{1}
		\rvertex{1.5}{0}
		\rgvertex{1.5}{-1}	

		\node at (2.5,1.7) {$G_6$};
		\draw (2.5,1) -- (2.5,-1);
		\rgvertex{2.5}{1}
		\rgvertex{2.5}{0}
		\rgvertex{2.5}{-1}
		
		\node at (4,1.7) {$G_7$};
		\draw (4.5,0) -- (3.5,0);
		\rgvertex{4}{1}
		\rgvertex{4}{-1}
		\gvertex{4.5}{0}
		\rvertex{3.5}{0}

	\end{scope}

\end{tikzpicture}
\end{center}

\end{figure}

We can produce an factorization of $G$ into IO-graph morphisms $[G_k, i_k,j_k]$ by using a strand arrangement to ``cut $S$ into thin regions'' $S_1$, \dots, $S_n$ corresponding to $G_1$, \dots, $G_n$ as shown in Figure \ref{fig:embeddedfactorization}.  To state what is happening precisely, suppose that
\begin{itemize}
	\item $C_0$, $C_1$, \dots, $C_n$ are smooth curves.
	\item $C_0 \cup C_n = \partial S$.
	\item $C_1$, \dots, $C_{n-1}$ form a strand arrangement $\mathcal{C}$.
	\item Suppose that $C_{k-1} \cup C_k$ is the piecewise smooth boundary of a surface $S_k$, and that $S = \bigcup_k C_k \cup \bigcup_k S_k$ and $\partial S_k \cap \partial S_{k+1} = C_k$.
	\item $\mathcal{C}$ induces a an embedded subgraph partition of $G$ into $G_1$, \dots, $G_n$ (as described above), where $G_k$ corresponds to $S_k$.
\end{itemize}
Let $P_k$ be the set of vertices $p$ such that $\overline{\mathcal{A}_p} \cap C_k \neq \varnothing$.  Then $G_k$ defines an IO-graph morphism $P_{k-1} \to P_k$, and $[G_n] \circ \dots \circ [G_0]$ is a factorization of $[G]: P_0 \to P_n$.  We call this construction a {\bf embedded factorization}.

\subsection{Elementary Factorizations in the Disk}

Any two ``cut-points'' $e^{i\theta}$ and $e^{i \phi}$ divide $\partial \D$ into two arcs; let $D_1$ be the CCW arc from $e^{i \theta}$ to $e^{i \phi}$ and let $D_2$ be the other arc.  Let $P$ and $Q$ be the sets of vertices of $G$ whose medial cells touch $D_1$ and $D_2$ respectively.  Then $P$ and $Q$ are called a \emph{circular pair}.  $P \cap Q$ contains at most two vertices.  The strands fall into three types:
\begin{itemize}
	\item A strand with both endpoints on $D_1$ is called {\bf $D_1$-reentrant}.
	\item A strand with both endpoints on $D_2$ is called {\bf $D_2$-reentrant}.
	\item A strand with one endpoint on $D_1$ and one on $D_2$ is called {\bf transverse}.
\end{itemize}

The following theorem combines the ``cut-point lemma'' of \cite{CM} (see also \cite{IZ}) with the machinery of elementary factorizations.  For an example, refer to Figure \ref{fig:embeddedfactorization}.

\begin{theorem} \label{thm:circularpairfactorization}
Let $G$ be a $\partial$-graph on $\D$ with a lensless nondegeneral medial strand arrangement $\mathcal{M}$.  Suppose $P$ and $Q$ are a circular pair corresponding to boundary arcs $D_1$ and $D_2$.  Then the IO-graph morphism $[G,i,j]: P \to Q$ represented by $G$ admits an embedded elementary factorization.  Hence, the rank-connection principle holds for $P$ and $Q$ for any network $\Gamma$ on $G$.  Moreover,
\begin{itemize}
	\item $\rank X([\Gamma,i,j]) = 2m(P,Q) = \# \text{\emph{(transverse strands)}} + |P \cap Q|$.
	\item $\dim \ker X([\Gamma,i,j]) = \#(\text{\emph{input stubs}}) = \# (D_1\text{\emph{-reentrant strands}})$.
	\item $\dim \ker \overline{X([\Gamma,i,j])} = \#(\text{\emph{output stubs}}) = \#(D_2\text{\emph{-reentrant strands}})$.
\end{itemize}
\end{theorem}

\begin{proof}
Our first goal is to find one of the following:
\begin{enumerate}[a.]
	\item A $D_1$-reentrant medial strand $s$ with no medial vertices on it.  In this case, there is a black cell on one side of $s$.  Because the closure of a medial cell only intersects $D_1$ in one arc, not two, the black medial cell must be one component of $\D \setminus s$ and must represent an isolated boundary vertex of $G$ on $D_1$.
	\item A triangular medial cell formed by two medial strand segments and an arc of $D_1$.  The two strand segments meet at some medial vertex $a$.  If the cell is black, then $a$ represents a boundary spike of $G$ and the black cell is the boundary vertex of the spike and is in $P$ and not $Q$.  If the cell is white, then $a$ represents a boundary edge of $G$ between two vertices in $P$.
\end{enumerate}

Let $e^{i\theta}$ and $e^{i\phi}$ be the two cut points dividing $\partial \D$ into $D_1$ and $D_2$, and consider the orientation $\mathcal{O}_\theta$. The transverse strands are all oriented to start at $D_1$ and end at $D_2$.  We claim that the medial vertices on the $D_1$-reentrant strands come before those on the $D_2$-reentrant strands in the partial order induced by $\mathcal{O}_\theta$, when they are comparable.  Let $\Omega_1$ be the union of the regions bounded by $D_1$-reentrant strands and (arcs of) $D_1$, and let $\Omega_2$ be the union of the regions bounded by the $D_2$-reentrant strands and $D_2$.  Note that $\Omega_1$ and $\Omega_2$ are disjoint by a simple Jordan curve theorem argument.

Moreover, there is no positively oriented path from $\Omega_2$ to $\Omega_1$ since any such path would have to exit $\Omega_2$ at some point $x$.  When it exits, it is moving along some strand $s$ which cannot be $D_2$-reentrant and hence has its starting point on $D_1$.  But $s$ must be crossing a reentrant strand $s'$, from the inside to the outside of the region bounded by $s'$ and $D_2$.  Since $s$ cannot cross $s'$, this implies that the start point of $s$ is inside the closure of the region bounded by $s'$ and $D_2$, and hence $s$ starts on $D_2$, which is a contradiction.

Let $W$ be the set of medial vertices $x$ in $\D \setminus \overline{\Omega_2}$.  The previous argument showed that $W$ is an initial subset of the medial vertices with respect to $\prec$.

Assume (a) does not occur and that $W$ is nonempty, and we will prove (b) occurs.  Let $x_1$ be a minimal element of $W$.  Then two medial strands $s_1$ and $t_1$ meet at $x_1$, and $s$ and $t$ have no medial vertices between $C_1$ and $x_1$.  Let $T_1$ be the triangle formed by $C_1$ and the segments of $s_1$ and $t_2$ from $C_1$ to $x_1$.  Now $T_1$ may be medial cell satisfying (b).  However, if the medial strand arrangement is disconnected, $T_1$ may contain some entire medial strands, which are necessarily $D_1$-reentrant.  In this case, let $\mathcal{M}_1$ be the union of the medial strands contained in $T_1$.  Let $x_2$ be a minimal medial vertex in $\mathcal{M}_1$.  Then $x_2$ is the vertex of a medial triangle $T_2$ by the same reasoning as before.  $T_2$ either satisfies (b) or contains some $\mathcal{M}_2$.  This process must terminate after finitely many steps since $\mathcal{M}_{j+1}$ contains strictly fewer strands than $\mathcal{M}_j$.  Hence, there is a triangle satisfying (b).

Therefore, either (a) or (b) occurs or else there are no $D_1$-reentrant strands and $W$ is empty.  If (a) or (b) occurs, we can write $[G,i,j] = [G',i',j'] \circ [G_1,i_1,j_1]$, where $G_1$ is an elementary IO-graph of type 1, 2, or 3 and the factorization can be represented by cutting $\D$ into two components with a curve $C_1$ from $e^{i\theta}$ to $e^{i \phi}$.

Let $U_1$ be the component of $\D \setminus C_1$ containing $G'$.  Then $U_1$ is homeomorphic to $\D$ (by standard results from topology) and the orientation satisfies all the same properties as before.  (Cutting the disk into two regions with $C_1$ may produce medial cells which intersect $\partial U_1$ in two arcs, but it cannot produce any which intersect $C_1$ in two arcs or $D_2$ in two arcs.)  If there are medial vertices in $W$ contained in $U_1$, we can repeat this process with $U_1$ instead of $\D$ and $C_1$ instead of $D_1$.  After finitely many iterations, we produce curves $C_1$, \dots, $C_k$ which induce an embedded factorization of $G$ into elementary type 1, 2, and 3 morphisms represented by $G_1$, \dots, $G_k$, and some other morphism represented $G^*$, such that $G^*$ is embedded in the region bounded by $C_k$ and $D_2$ and this region contains no vertices of $W$.

Next, we repeat this process starting at $D_2$ instead of $D_1$, using a $D_2$-reentrant strand with no medial vertices or a maximal medial vertex in our partial order, and hence finding a boundary edge, boundary spike, or isolated boundary vertex on the output side $D_2$.  We ``peel off layers'' from the output side rather than the input side.  When this process ends, the two layer-stripping processes from the input side and the output side must meet in the middle and produce a complete factorization.  Indeed, since both processes terminated, there cannot be any more medial vertices or reentrant strands in the region that is left, and hence all the strands are transverse and do not intersect, and the IO-graph morphism in this region is the identity.  Thus, the factorization is complete.

It follows from Theorem \ref{thm:RC1} that the rank-connection principle holds.  Thus, it only remains to establish the relationship between the strands and the number of input stubs, the number of output stubs, and the maximum size connection.

Each time we factored out a type 3 morphism from the input side, we removed a reentrant strand on the input side.  However, we when factored out a type 1 or type 2 morphism, this did not change the number of reentrant strands.  Thus, the $D_1$-reentrant strands correspond to input stubs.  Similarly, the $D_2$-reentrant strands correspond to the output stubs.  Factoring out any of the elementary networks did not change the number of transverse strands or the maximum size connection.  Thus, to prove the claim about the transverse strands, we can reduce to the case where all the strands are transverse and do not intersect each other, and here the claim follows from easy casework.
\end{proof}

\begin{remark*}
The scaffold corresponding to the factorization in Theorem \ref{thm:circularpairfactorization} can be represented by an orientation of the medial strands, if we allow two segments of the same strand to have opposite orientations.
\end{remark*}

\subsection{Supercritical Half-Planar $\partial$-Graphs} \label{subsec:halfplanar}

In \cite{IZ}, a $\partial$-graph $G$ embedded in the upper half-plane $\h \subset \C$ is called {\bf supercritical} if it has compatible lensless medial strand arrangement such that each medial strand begins and ends on $\R$ rather than going off to $\infty$.  \cite{IZ} adapts the techniques of \cite{WJ} to prove recoverabiliy of supercritical half-planar $\partial$-graphs.  We shall prove

\begin{theorem} \label{thm:supercritical}
Any supercritical half-planar $\partial$-graph $G$ is totally layerable, that is, for each edge $e_0$ there is a scaffold $S$ such that $e_0$ is in $S \cup \overline{S}$ and in $\Mid S$ and a scaffold such that $e_0$ is not in $S \cup \overline{S}$ and $e_0 \in \Mid S$.  Moreover, the scaffolds can be chosen so that $\{e: e \succ e_0\}$ is finite.
\end{theorem}

Finiteness of $\{e: e \succ e_0\}$ implies that the harmonic functions constructed for solving the inverse problem in \S \ref{subsec:scafcontinuation} are \emph{finitely supported}.  This is useful because it allows flexibility in defining the $\Lambda(\Gamma)$ in the infinite case.  For positive real edge weights, one might want to consider the boundary data of finitely supported, bounded, or finite-power harmonic functions rather than all harmonic functions (see \cite{IZ}).  However, harmonic continuation might a priori produce unbounded or infinite-power harmonic functions.  On the other hand, finitely supported functions automatically satisfy whatever growth conditions one wants to impose at infinity.

Without the finiteness condition on $\{e: e \succ e_0\}$, one can prove that a supercritical half-planar $\partial$-graph is totally layerable using an orientation of the medial strands, similar to the method for the disk, but slightly more complicated, because we must make sure every subset has a minimal element.  However, the finiteness condition makes the proof more tricky.  The basic plan is as follows:

Let $t_0 < t_1$ be two points on the real line that are not the endpoints of medial strands.  The goal is to construct a scaffold such that the harmonic continuation process where the ``inputs'' are on $(-\infty,t_0] \cup [t_1,\infty)$ and ``output'' are on $[t_0,t_1]$.  This cannot be accomplished by simply orienting each medial strand.  Instead, we will divide $\h$ into three regions, produce a scaffold on each region, and then patch the scaffolds together.  The most annoying part of the proof is finding the correct way of cutting up $\h$.  This proof is technical and may be omitted on a first reading.

\textbf{Division of $\h$ into Three Regions:}  Each strand divides $\h$ into two components--one is bounded, and we will call it the ``inside,'' and the other is unbounded, and we will call it the ``outside.''  Each strand has an endpoint which is further left on the real axis and one which is further right, and hence there is a left-to-right orientation of each strand.  In the left-to-right orientation of the strand, the inside is on the right of the strand and the outside is on the left.

Let $U$ be the union of all the following regions:
\begin{itemize}
	\item The inside of a $[t_0,t_1]$-reentrant strand.
	\item Any triangle bounded by a segment of a strand with one endpoint on $(-\infty,t_0]$ and one endpoint on $[t_0,t_1]$, a strand with one endpoint on $[t_0,t_1]$ and one endpoint on $[t_1,\infty)$, and a segment of $[t_0,t_1]$. 
\end{itemize}

\begin{claim*}
$U$ is the region to the right of some oriented Jordan arc $C_0$ formed by strand segments and segments of $[t_0,t_1]$ such that
\begin{itemize}
	\item The path starts at $t_0$ and ends at $t_1$.
	\item Each strand used in the path has at least one endpoint on $[t_0,t_1]$.
	\item For each strand segment in the path, the orientation of the path matches the left-to-right orientation of the strand.
	\item For each segment of $[t_0,t_1]$ in the path, the orientation in the path matches the increasing orientation of $[t_0,t_1]$.
\end{itemize}
\end{claim*}

\begin{proof}
Let $\mathcal{O}$ be the orientation of $\mathcal{M}$ formed by orienting each strand from left to right and real line from negative to positive.  Then $\mathcal{O}$ is acyclic.  Indeed, any cycle would be formed by only finitely many strands $s_1, \dots, s_n$.  If $F$ is a conformal map of $\h$ onto $\D$ and $e^{i\theta} = F(\infty)$, then the orienation $\mathcal{O}$ of $s_1, \dots, s_n$ corresponds to $\mathcal{O}_\theta$ on the disk.  But we already showed this is acyclic.  Thus, $\mathcal{O}$ defines a partial order on the medial vertices.  This can be extended to a partial order on the medial vertices \emph{and} endpoints of strands such that if two endpoints $x$, $y$ are on the real line with $x < y$ in $\R$, then $x \prec y$.

We say that a region satisfies ($*$) if it is the region to the right of some path satisfying the conditions of the Claim.  Note that $U$ is defined as the union of finitely many regions which satisfy ($*$).  Thus, it suffices to show that if $U_1$ and $U_2$ satisfy ($*$), then so does $U_1 \cup U_2$.  Let $g_1$ and $g_2$ be the corresponding paths and extend them to infinite paths by adjoining an interval of the form $(-\infty,a]$ to the beginning and $[b,+\infty)$ to the end.  The intersection points / intervals of $g_1$ and $g_2$ must occur in increasing order along $g_1$ and in increasing order along $g_2$ since they are both positively oriented paths with respect to $\mathcal{O}$.  Hence, the intersections occur in the same order for $g_1$ and $g_2$.  Thus, we can form a path $g_3$ as follows:  Start at $-\infty$.  As long as $g_1$ and $g_2$ agree, we follow along their common path, and when $g_1$ and $g_2$ split up, we choose the path farther to the left.  Then $U_1 \cup U_2$ is the region to the right of $g_3$, hence satisfies ($*$).
\end{proof}

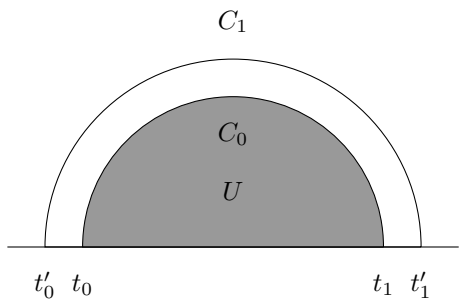
\begin{figure}

\caption{The region $U$ and curves $C_0$ and $C_0'$.}

\begin{center}
\begin{tikzpicture}

\draw (-3,0) -- (3,0);
\draw[fill=black!40] (2,0) arc (0:180:2) -- (2,0);
\draw (2.5,0) arc (0:180:2.5) -- (2.5,0);
\node at (0,1.5) {$C_0$};
\node at (0,3) {$C_1$};
\node at (0,0.75) {$U$};

\node at (-2.5,-0.5) {$t_0'$};
\node at (-2,-0.5) {$t_0$};
\node at (2,-0.5) {$t_1$};
\node at (2.5,-0.5) {$t_1'$};

\end{tikzpicture}
\end{center}

\end{figure}

We next produce another curve $C_0'$ that ``hugs the outside of $C_0$'' but does not contain any medial vertices:

\begin{claim*}
There exists an oriented Jordan arc $C_0'$ such that
\begin{itemize}
	\item $C_0'$ does not contain any medial vertices.
	\item If $s$ is a medial strand with one endpoint on $[t_0,t_1]$ and one endpoint on $(-\infty,t_0] \cup [t_1,+\infty)$, then $C_0'$ intersects $s$ exactly once.
	\item The region to the left of $C_0'$ contains $\overline{U}$ and does not contain any medial vertices not in $\overline{U}$.
	\item The start point $t_0'$ of $C_0'$ is to the left of $t_0$ with no endpoints of strands in between them.  The end point $t_1'$ of $C_0'$ is to the right of $t_1$ with no endpoints of strands in between.
\end{itemize}
\end{claim*}

\begin{proof}
Let $\mathcal{A}_1, \dots, \mathcal{A}_n$ be the medial cells outside $U$ whose closures intersect $C_0$, listed in order along $C_0$.  Construct $C_1$ inductively starting on $\R \cap \partial \mathcal{A}_1$, then going into $\mathcal{A}_2$, and so forth.

The hardest condition to verify is the second one: Suppose $s$ is a medial strand with one endpoint on $(-\infty,t_0]$ and one endpoint on $[t_0,t_1]$.  If $s$ crosses $C_0'$, then it must enter $\overline{U}$ immediately afterward.  At the point where it enters $U$, it must either cross a $[t_0,t_1]$-reentrant strand or enter a triangle formed by strands $s_1$ and $s_2$, where $s_1$ has endpoints on $(-\infty,t_0]$ and $[t_0,t_1]$, and $s_2$ has endpoints on $[t_0,t_1]$ and $[t_1,+\infty)$.  Move along $s$ starting at the endpoint on $(-\infty,t_0]$.  If $s$ crosses a $[t_0,t_1]$-reentrant strand, then it cannot cross it again, and hence is trapped inside $\overline{U}$ and cannot cross $C_0'$ again.  If it enters a triangle formed by $s_1$ and $s_2$, then it must have crossed $s_2$ at some point since it started outside $s_2$.  Then the triangle formed by $s_1$ and $s$ is inside $U$, so the rest of $s$ must also be inside $\overline{U}$.  A symmetrical argument works if $s$ has one endpoint on $[t_0,t_1]$ and one on $[t_1,\infty)$.
\end{proof}

\begin{claim*}
There is a point $z$ on $C_0'$ such that
\begin{itemize}
	\item Any strand starting on $(-\infty,t_0]$ and ending on $[t_0,t_1]$ must intersect $C_0'$ before $z$ (``before'' along $C_0'$).
	\item Any strand starting on $[t_1,+\infty)$ and ending on $[t_0,t_1]$ must intersect $C_0'$ after $z$.
\end{itemize}
\end{claim*}

\begin{proof}
Let $\mathcal{A}_1, \dots, \mathcal{A}_n$ be as above.  Since $C_0'$ ends on the outside of all strands with endpoints on $(-\infty,t_0]$ and $[t_0,t_1]$, there must be a first $\mathcal{A}_j$ that is on the the outside of all such strands.  Let $z$ be a point of $C_0'$ inside $\mathcal{A}_j$, and let $s_1$ be the last strand with endpoints on $(-\infty,t_0]$ before $\mathcal{A}_j$.

Suppose for contradiction $s$ is a strand with endpoints on $[t_1,\infty)$ and $[t_0,t_1]$ that intersects $C_0'$ before $z$.  Since $s$ only intersects $C_0'$ once, the only way it can do this is by crossing $s_1$ outside of $C_0'$, which contradicts the definition of $U$.
\end{proof}

\begin{figure}

\caption{Division of $\h$ into three regions.  The orange arrows show the general direction of harmonic continuation.  Here $C_0' = C_1 \cup C_2$ and $R_3$ is slightly larger than $\overline{U}$.} \label{fig:division}

\begin{center}
\begin{tikzpicture}

\draw[<->] (-5,0) -- (5,0);
\draw (2.5,0) arc (0:180:2.5) (-2.5,0);
\draw[->] (0,2.5) -- (0,6);

\node at (0.4,2.9) {$z$};
\node at (-2.5,-0.5) {$t_0'$};
\node at (2.5,-0.5) {$t_1'$};

\node at (0.5,5) {$C_1$};
\node at (135:3) {$C_2$};
\node at (45:3) {$C_3$};

\node at (-4,4) {$R_1$};
\node at (4,4) {$R_2$};
\node at (0,1) {$R_3$};

\draw[->,orange] (-4.5,0) arc (190:70:1.5);
\draw[->,orange] (-4.5,3) arc (120:50:3);

\draw[->,orange] (4.5,0) arc (-10:110:1.5);
\draw[->,orange] (4.5,3) arc (60:130:3);

\draw[->,orange] (135:2) -- (145:0.9);
\draw[->,orange] (45:2) -- (35:0.9);

\draw[->,orange] (-4,5) -- (-1,5);
\draw[->,orange] (4,5) -- (1,5);

\end{tikzpicture}
\end{center}

\end{figure}
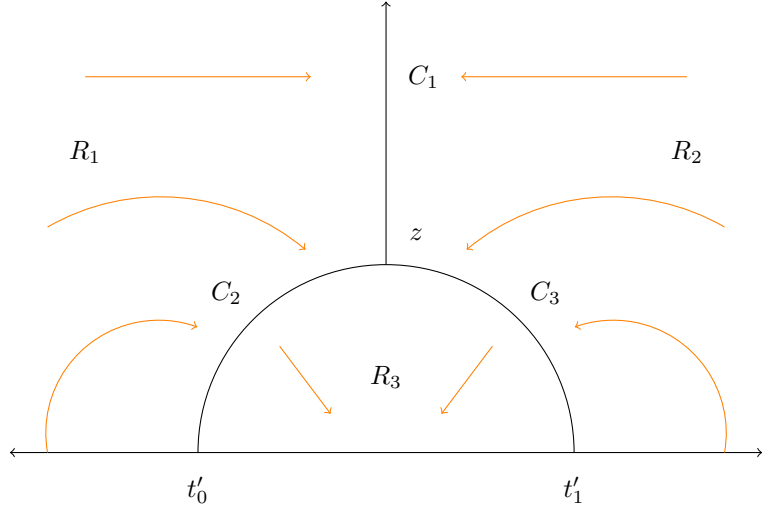

\begin{claim*}
There exists an oriented curve $C_1$ injectively parametrized by $[0,+\infty)$ such that
\begin{itemize}
	\item $C_1$ starts at $z$ and goes to complex $\infty$.
	\item $C_1$ does not contain any medial vertices.
	\item $C_1$ intersects each strand at most once.
	\item $C_1$ only crosses strands from inside to outside.
	\item $C_1$ never intersects $C_0'$ again.
\end{itemize}
\end{claim*}

\begin{proof}
For a given medial cell $\mathcal{A}$ bounded by strands $s_1, \dots, s_n$, there are two possibilities:
\begin{enumerate}
	\item $\mathcal{A}$ is on the inside of some $s_j$.
	\item $\mathcal{A}$ is on the outside of each $s_j$.  In this case, by a connectedness argument, $\mathcal{A}$ is exactly the intersection of the outsides of the $s_j$'s, and hence is unbounded.
\end{enumerate}
We construct $C_1$ inductively cell by cell, starting at $z$.  As long as we are in a cell where (1) holds, we can continue into another cell by crossing a strand from inside to outside.  If we ever reach a cell where (2) holds, we can stay inside the cell and go to $\infty$.  Because we only ever cross strands from inside to outside, we never cross the same strand twice or enter the same medial cell twice.

We never entered $\overline{U}$ because to do that, we would have to cross from the outside to the inside of some strand (by previous Claims about $U$).  Thus, we can arrange that we never cross $C_0'$ (since $C_0'$ was defined to ``skirt the outside of'' $\overline{U}$).

Now we prove the path goes to $\infty$.  This is trivial if (2) ever occurs.

If (1) occurs infinitely many times, then I claim the path is eventually outside any given strand $s$.  The path crosses infinitely many strands from inside to outside.  However, the inside of $s$ only intersects finitely many strands, so the path cannot stay inside $s$ forever, and once it goes outside of $s$ it cannot come back inside.

Suppose $K \subset \overline{\h}$ is compact, and we will show that the path is eventually outside of $K$.  Only finitely many medial cells intersect $K$.  Let $K'$ be the union of the medial cells that intersect $K$ and satisfy (1).  Since we assumed (2) never occurs, the path never enters any unbounded cells, so it suffices to show the path is eventually outside $K'$.  But any cell of $K'$ is on the inside of some strand, and we just proved that the path $C_1$ is eventually outside every strand.
\end{proof}

\begin{claim*}
Let $C_2$ be the arc of $C_0'$ before $z$ and let $C_3$ be the arc of $C_0'$ after $z$.  Then $C_1$, $C_2$, and $C_3$ divide $\h$ into three simply connected regions homeomorphic to the disk:
\begin{itemize}
	\item $R_1$ is the region outside $C_0'$ and to the left of $C_1$.  It is bounded by $(-\infty,t_0']$, $C_1$, and $C_2$.
	\item $R_2$ is the region outside $C_0'$ and to the right of $C_2$.  It is bounded by $[t_1',+\infty)$, $C_1$, and $C_3$.
	\item $R_3$ is the region inside $C_0'$.  It is bounded by $[t_0',t_1']$, $C_2$ and $C_3$.
\end{itemize}
\end{claim*}

\begin{proof}
Use the Jordan curve theorem and conformal equivalence of the half-plane and disk.
\end{proof}

\begin{claim*}
Let $G_1, G_2, G_3$ be the $\partial$-subgraph partition of $G$ induced by the division of $\h$ into $R_1$, $R_2$, and $R_3$, and let $\mathcal{M}_1, \mathcal{M}_2, \mathcal{M}_3$ be the corresponding medial strand arrangements.  Then
\begin{itemize}
	\item Any strand of $\mathcal{M}_1$ either has both endpoints on $(-\infty,t_0']$ or one endpoint on $(-\infty,t_0']$ and one on $C_2 \cup C_1$.
	\item Any strand of $\mathcal{M}_2$ either has both endpoints on $[t_1',+\infty)$ or one endpoint on $[t_1',+\infty)$ and one on $C_3 \cup C_1$.
\end{itemize}
\end{claim*}

\begin{proof}
Consider a medial strand $s$ from the original medial strand arrangement $\mathcal{M}$.
\begin{itemize}
	\item If $s$ is $[t_0,t_1]$-reentrant since then it would is entirely contained in $\overline{U} \subset R_3$, so there is nothing to prove.
	\item Suppose $s$ has one endpoint on $(-\infty,t_0] \cup [t_1,+\infty)$ and one on $[t_0,t_1]$.  Then it crosses $C_0'$ exactly once.  Since $C_1$ only crosses strands from inside to outside and it starts outside $s$, we know $s$ never crosses $C_1$, so we are done.
	\item Suppose $s$ has one endpoint on $(-\infty,t_0]$ and one on $[t_1,+\infty)$, and that it never crosses $C_0'$.  Then we are done since $C_1$ intersects each strand at most once.
	\item Suppose $s$ has one endpoint on $(-\infty,t_0]$ and one on $[t_1,+\infty)$, and that it crosses $C_0'$ at some time.  Orient $s$ to start on $(-\infty,t_0]$ and end on $[t_1,+\infty)$.  Note $s$ cannot intersect a $[t_0,t_1]$-reentrant strand.  Thus, once $s$ enters $\overline{U}$, it must be inside one of the triangles in the definition of $U$, hence it has gone to the inside of a strand $s'$ with one endpoint on $[t_0,t_1]$ and one endpoint on $[t_1,+\infty)$.  Since $C_1$ is outside of $s'$, $s$ can never intersect $C_1$ after this point.  But by a symmetrical argument, $s$ can never intersect $C_1$ \emph{before exiting} $\overline{U}$.  Thus, it can never intersect $C_1$ at all.
	
	Furthermore, if $s$ crosses $C_2$ and hence enters $\overline{U}$, it is inside $s'$ and hence remains outside $R_1$ and never crosses $C_2$ again.  Thus, $s$ can must cross $C_2$ exactly once and $C_3$ exactly once by symmetry.
\end{itemize}
\end{proof}

\textbf{Construction of Scaffold:}  The scaffold will be defined so that the direction of harmonic continuation is roughly as follows:
\begin{itemize}
	\item In $G_1$, it will go from $(-\infty,t_0']$ to $C_2 \cup C_1$.
	\item In $G_2$, it will go from $[t_1',+\infty)$ to $C_3 \cup C_1$.
	\item In $G_3$, it will go from $C_2 \cup C_3$ to $[t_0',t_1']$.
\end{itemize}
We will define scaffolds on $G_1$, $G_2$, and $G_3$, then paste them together.

\begin{claim*}
Let $\mathcal{O}_1$ be the orientation of $\mathcal{M}_1$ defined as follows:
\begin{itemize}
	\item A $(-\infty,t_0']$-reentrant strand is oriented from right to left.
	\item A strand with one endpoint on $(-\infty,t_0']$ and one on $[t_0',t_1']$ is oriented to start on $(-\infty,t_0']$.
	\item The boundary is oriented counterclockwise except for a small interval near $t_0'$.
\end{itemize}
Then $\mathcal{O}_1$ defines a scaffold $S_1$ on $G_1$.
\end{claim*}

\begin{proof}
Let $F: \h \to \D$ be a conformal map and let $e^{i\theta} = F(t_0')$.  The orientation $\mathcal{O}_1$ matches $\mathcal{O}_\theta$ on the disk, and hence is acyclic.  The same argument shows that the medial cells have the Desired Behavior.

To show that every subset has a minimal element, it suffices to show that any descreasing path of medial strand segments must terminate.  Let $C$ be any such path, and let $Z$ be the set of strands used in the path.  Let $s_0$ be the strand with the endpoint closest to $t_0'$ on the real line.  Then no strand can cross from the right (outside) of $s$ to the left (inside) of $s$.  Hence, once the decreasing path reaches $s$, it remains trapped in the closure of the region inside $s$, which contains only finitely many medial vertices.  Hence, the path must terminate.
\end{proof}

\begin{claim*}
Symmetrically, Let $\mathcal{O}_2$ be the orientation of $\mathcal{M}_2$ defined as follows:
\begin{itemize}
	\item A $[t_1',+\infty)$-reentrant strand is oriented from right to left.
	\item A strand with one endpoint on $[t_1',+\infty)$ and one on $[t_0',t_1']$ is oriented to start on $[t_1',+\infty)$.
	\item The boundary is oriented counterclockwise except for a small interval near $t_1'$.
\end{itemize}
Then $\mathcal{O}_2$ defines a scaffold $S_2$ on $G_2$.
\end{claim*}

To construct a scaffold on $G_3$, note there is a homeomorphism $F: R_3 \to \D$ (by corollaries of the Jordan curve theorem), and the homeomorphism extends to the closures.  In particular, $G_3$ is circular planar with no lenses in the medial strands.  Let $S_3$ be the scaffold obtained by pulling back $\mathcal{O}_\theta$ through $F$, where $\theta$ is chosen with $e^{i\theta} = F(t_1')$.  This is chosen so that all the strands with one endpoint on $C_2 \cup C_3$ and one on $[t_0',t_1']$ are oriented from $C_2 \cup C_3$ to $[t_0',t_1']$.

Now $S = S_1 \cup S_2 \cup S_3$ is not necessarily a scaffold on $G$.  But observe that
\begin{itemize}
	\item Any interior vertex of $G_1$ or $G_2$ or $G_3$ is both an input and an output of the edges in $S$.
	\item If $p$ is a vertex of $G_1$ whose cell touches $C_1 \cup C_2$, then $p$ cannot be an input of an edge in $S_1$ since there are no oriented strands exiting $C_1 \cup C_2$. A symmetrical claim holds for $G_2$.
	\item If $p$ a vertex of $G_1$ which is interior in $G$ and its medial cell touches $C_1 \cup C_2$, then its medial cell does not touch $\R$, and this vertex must be an output of an edge in $S_1$.  A symmetrical claim holds for $G_2$.
	\item Similarly, any vertex $p$ in $G_3$ which is interior in $G$ and touches $C_2 \cup C_3$ must be an input of $S_3$.  Any vertex which touches $C_2 \cup C_3$ cannot be an output of edges in $S_3$ since there are no strands in $\mathcal{M}_3$ entering $C_2 \cup C_3$.
\end{itemize}
We define $S'$ as $S$ minus the edges in $e \in E(S_2)$ such that the medial cell of $e_+$ in $\mathcal{M}_2$ touches $C_1 \cup C_2$.  Now every interior vertex of $G$ is the output of some edge in $S$, but any vertex is the output of at most one edge in $S$ and the input of at most one edge in $S$.

\begin{claim*}
$S$ is a scaffold.
\end{claim*}

\begin{proof}
We can use the partial order $\prec$ defined by using the partial orders associated to $S_1$, $S_2$, and $S_3$, and declaring that edges in $G_1$ are less than edges in $G_2$ which are less than edges in $G_3$.  The local comparison and partial well-order conditions of Lemma \ref{lem:scafequivalentdef} are verified by casework.  The input-output alternative is trivial since every interior vertex is the output of some edge in $S$.
\end{proof}

\begin{claim*}
$E(G_3) \subset \Mid S$.
\end{claim*}

\begin{proof}
Any interior vertex of $G$ which is in $G_3$ is the input of some edge in $S$.  Thus, any edge incident to a vertex in $V^\circ(G) \setminus S_-$ is in $G_1$ or $G_2$ and hence is $\prec$ the edges in $G_3$ by definition of our partial order.
\end{proof}

\textbf{Proof of Theorem \ref{thm:supercritical}:}  Choose an edge $e_0$.  By choosing $t_0$ and $t_1$ correctly and constructing a scaffold $S$ as above, we can arrange that the medial vertex of $e_0$ is on a $[t_0,t_1]$-reentrant strand, hence in $U$ and hence in $G_3$.  We can arrange $e_0$ is either in $S \cup \overline{S}$ or not in $S \cup \overline{S}$ as desired.  Thus, the Theorem follows from the previous claim.

\section{The Rank-Connection Principle} \label{sec:characterization}

One can show using Lemma 4.1 of \cite{CIM} or the determinantal formulas of \cite{RF} that the rank-connection principle holds for any finite network for \emph{generic} edge weights in any algebraically closed field.  Here we want to understand what happens in the \emph{general} rather than the generic case.  We will show that $\rank X([\Gamma,i,j]) \leq m([\Gamma,i,j])$ always, and give a geometric characterization of morphisms $[G,i,j]$ such that the rank-connection principle holds for \emph{all} edge weights over any field, using a slight generalization of elementary factorizations.

\subsection{Completely Reducible $\partial$-Graphs} \label{subsec:completelyreducible}

Our task is even more subtle than it might first appear, since it turns out that a harmonic function is not necessarily uniquely determined by its boundary data.  Even if the $\partial$-graph is connected, there can be degenerate edge weights for which some nonzero harmonic functions have zero potential and zero net current on the boundary.  As stated, the rank-connection principle pertains to the boundary behavior rather than the space of harmonic functions; we want to know how the input data on $P$ and output data on $Q$ are related, even though we might have no control over the values of $u$ on the interior between them!

For one example of nonzero harmonic functions with zero boundary data, consider the ``triangle-in-triangle'' network with boundary vertices $\{1,\dots,6\}$ and interior vertices $\{7,8,9\}$ and edges with coefficients $w(e)$ shown in the figure.  (This $\partial$-graph was considered in \cite{LP}.)  The matrix of $\Delta$ is
\[
\begin{pmatrix}
0 & 0 & 0 & 0 & 0 & 0 & -1 & 0 & 1 \\
0 & 0 & 0 & 0 & 0 & 0 & 1 & -1 & 0 \\
0 & 0 & 0 & 0 & 0 & 0 & 0 & 1 & -1 \\
0 & 0 & 0 & 0 & 0 & 0 & -1 & 0 & 1 \\
0 & 0 & 0 & 0 & 0 & 0 & 1 & -1 & 0 \\
0 & 0 & 0 & 0 & 0 & 0 & 0 & 1 & -1 \\
-1 & 1 & 0 & -1 & 1 & 0 & 0 & 0 & 0 \\
0 & -1 & 1 & 0 & -1 & 1 & 0 & 0 & 0 \\
1 & 0 & -1 & 1 & 0 & -1 & 0 & 0 & 0
\end{pmatrix}.
\]
Let $e_p$ be the vector with $1$ on vertex $p$ and zero elsewhere.  Then $e_7 + e_8 + e_9$ is a harmonic potential which is zero on the boundary and also has net current zero at each boundary vertex.

\begin{figure}
\caption{Singular edge weights on the triangle-in-triangle network.  Boundary vertices are colored in.  Vertices are labelled with their index.  Edges are labelled with their conductance.}

	\begin{center}
		\begin{tikzpicture}
			\node[circle,fill] (1) at (90:5) [label = above: $1$] {};
			\node[circle,fill] (2) at (210:5) [label = left: $2$] {};
			\node[circle,fill] (3) at (330:5) [label = right: $3$] {};
			\node[circle,fill] (4) at (90:2) [label = below: $4$] {};
			\node[circle,fill] (5) at (210:2) [label = left: $5$] {};
			\node[circle,fill] (6) at (330:2) [label = right: $6$] {};
			\node[circle,draw] (7) at (150:4) [label = left: $7$] {};
			\node[circle,draw] (8) at (270:4) [label = below: $8$] {};
			\node[circle,draw] (9) at (30:4) [label = right: $9$] {};
			
			\draw (7) to node[auto,swap] {$1$} (2) to node[auto,swap] {$-1$} (8) to node[auto,swap] {$1$} (3) to node[auto,swap] {$-1$} (9) to node[auto,swap] {$1$} (1) to node[auto,swap] {$-1$} (7) to node[auto] {$1$} (5) to node[auto] {$-1$} (8) to node[auto] {$1$} (6) to node[auto] {$-1$} (9) to node[auto] {$1$} (4) to node[auto] {$-1$} (7);
		\end{tikzpicture}
	\end{center}
\end{figure}
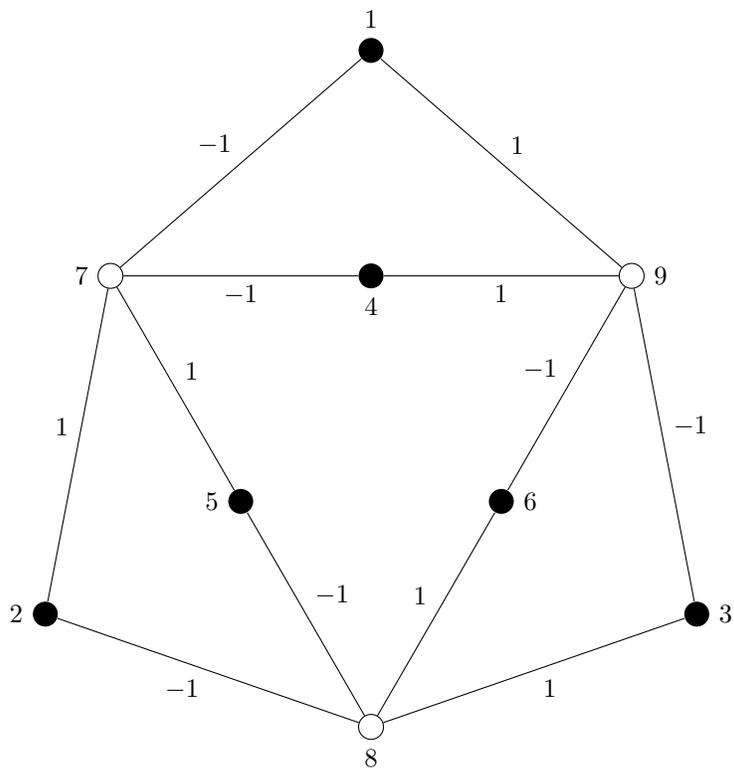

As a warm-up to handling the rank-connection principle, we will first give a geometric characterization of $\partial$-graphs for which a harmonic function is always uniquely determined by its boundary data.  This proposition was proved by Avi Levy and the author in \cite{torsion}.  The necessity of using boundary wedge-sums in the characterization was pointed out by Will Dana, Austin Stromme, and Collin Literell, students at the University of Washington REU 2015.

\begin{definition}
A finite $\partial$-graph is {\bf completely reducible} if it can be reduced to the empty graph by layer-stripping operations and splitting apart boundary wedge-sums.  In other words, completely reducible $\partial$-graphs are the smallest class of finite $\partial$-graphs which is closed under adding isolated boundary vertices, adding boundary edges, adding boundary spikes, and joining two graphs together into a boundary wedge-sum.  
\end{definition}

\begin{definition}
We define $\mathcal{U}_0(\Gamma) = \{u \in \mathcal{U}(\Gamma): u|_{\partial V} = 0, \Delta u_{\partial V} = 0\}$.
\end{definition}

\begin{proposition} \label{prop:completelyreducible}
A finite $\partial$-graph $G$ is completely reducible if and only $\mathcal{U}_0(\Gamma) = 0$ for all nonzero edge weights over any field.  In fact, for any given infinite field $\F$, $G$ is completely reducible if and only if $\mathcal{U}_0(\Gamma) = 0$ for all nonzero edge weights in $\F$.
\end{proposition}

\begin{proof}
To prove the forward direction, it suffices to show that the class of $\partial$-graphs for which $\mathcal{U}_0(\Gamma)$ is always zero is closed under the operations of adding spikes, boundary edges, and isolated boundary vertices, and under boundary wedge-sums.  This follows from similar reasoning as in Lemma \ref{lem:wedgesumrecovery} and Lemma \ref{lem:reductionelectrical}.

Now consider the converse implication.  Let $\F$ be an infinite field and suppose $G$ is not completely reducible.  By applying layer-stripping operations and breaking apart boundary wedge-sums and connected components we can find a harmonic subgraph $H$ which has no boundary spikes, boundary edges, or isolated boundary vertices, and cannot be written as a boundary wedge-sum or disjoint union.  It suffices to find edge weights on $H$ which will make $\mathcal{U}_0$ nonzero, since any harmonic function with zero potential and current on a subgraph can be extended to be zero on the larger network.

We can assume without loss of generality that there are no self-looping edges since these edges have no effect on $\Delta$ or on complete reducibility (they can always be removed as boundary edges once the endpoint is a boundary vertex).

Our strategy will be to choose a potential function $u$ first with $u|_{\partial V} = 0$, and \emph{then} choose edge weights such that $\Delta u \equiv 0$.  Let $S$ be the set of edges in $H$ that are contained in a cycle.  Define $u$ to be zero on any component of $H \setminus S$ that contains boundary vertices of $H$, and assign $u$ a distinct, nonzero value on each of the other components.  It follows from the definition of $S$ that any edge in $S$ must have endpoints in distinct components of $H \setminus S$.

To guarantee that $u$ does not vanish on any interior vertex, it suffices to show that any edge $e$ with endpoints $e_- = p \in \partial V$ and $e_+ = q \in V^\circ$ must be contained in a cycle.  By hypothesis, $e$ is not a self-loop or boundary spike. Thus there is some edge $e' \neq e$ with $e_-' = e_-$.  If $e_+' = e_+$, then $e$ and $e'$ form a two-cycle.  Otherwise, let $r = e_+' \neq q$.  Since $G$ is not a boundary wedge-sum, deleting $p$ leaves $G$ connected. Thus, there is a path from $q$ to $r$ avoiding $p$.  Hence, there is a cycle containing $e$ and $e'$. Consequently, $u$ is nonzero on all the interior vertices.

Now we choose the edge weights.  Choose oriented cycles $C_1, \dots, C_k$ such that $S = \bigcup_{j=1}^k C_j \cup \overline{C}_j$.  For each $j$, define
  \[
    w_j(e) = w_j(\overline{e}) = \begin{cases}
      1/du(e), \text{ for } e \in C_j \\
      0, \text{ for } e \not \in C_j \cup \overline{C}_j.
    \end{cases}
  \]
Then $w_j(e) du(e)$ is $1$ on $C_j$ and $-1$ on $\overline{C}_j$ and vanishes elsewhere.  Thus, $\Delta_{w_j} u = 0$.  For each $e \in S$, there is a weight function $w_j$ with $w_j(e) \neq 0$.  Since $\F$ is infinite and the graph is finite, we may choose $\alpha_j \in F$ such that $\sum_{j=1}^k \alpha_j w_j(e) \neq 0$ for all $e \in S$ simultaneously.

Set $w = 1_S+\sum_{j=1}^k \alpha_j w_j$.  Then $w(e) \neq 0$ for each $e$.  Since $\Delta_{w_j} u = 0$ for each $j$ and $du(e) = 0$ when $e \not \in S$, we have $\Delta_w u = 0$ by linearity.  Thus, $\mathcal{U}_0(H,w) \neq 0$ as desired.
\end{proof}

\subsection{A Max-Flow Min-Cut Principle}

One ingredient in our rank-connection theorem is the following result, which is also of interest in its own right.  The number $m(P,Q)$ can be thought of as the ``maximum flow'' from $P$ to $Q$, although our setup is different than the standard max flow problem in that the flow \emph{through each vertex} is limited rather than just the flow through each edge.  The correct analogue of ``minimum cut'' can be phrased in terms of factorizations in the IO-graph category.

\begin{proposition}[Max-Flow Min-Cut] \label{prop:maxflowmincut}
Let $[G,i,j]: P \to Q$.  Then $m([G,i,j])$ is the minimum value of $|R|$ such that $[G,i,j]$ factors as the composition of two morphisms $P \to R$ and $R \to Q$.
\end{proposition}

\begin{proof}
To simplify notation, we assume that $P$, $Q$, and $R$ are literally subsets of $V(G)$ rather than merely labels, and write $[G]: P \to Q$ rather than $[G,i,j]: P \to Q$.  We will also write graphs without specifying boundary vertices, with the understanding that if we write $[G]: P \to Q$, then we are treating $P \cup Q$ as the set of boundary vertices.

Let $n$ be the minimum value of $|R|$ and let $m$ be the maximum size connection.  It is clear that $m \leq n$ since for any factorization into morphisms $P \to R$ and $R \to Q$, any path from $P$ to $Q$ must pass through $R$.

We prove the reverse equality by induction on the number of edges.  If there are zero edges, then the maximum connection is taken by using the trivial paths from $P \cap Q$ to $P \cap Q$.  On the other hand, the IO-graph morphism can be factored into maps $P \to P \cap Q \to Q$.

Suppose $G$ has at least one edge.  Choose an edge $e$ with endpoints $p$ and $q$ and let $H$ be obtained from $G$ by deleting $e$ without changing the vertex set.  By the induction hypothesis $[H]: P \to Q$ has a factorization into two subgraphs
\[
[H_1]: P \to R, \quad [H_2]: R \to Q
\]
such that $|R| = m_H(P,Q)$.

Suppose $p$ and $q$ are both in $V(H_1)$ or both in $V(H_2)$.  By symmetry, we can assume without loss of generality $p, q \in V(H_1)$.  Then we have a factorization of $G$ into
\[
[H_1 \cup e]: P \to R, \quad [H_2]: R \to Q.
\]
Any maximum size connection from $P$ to $Q$ in $H$ is also a connection in $G$.  Since there is a connection of size $|R|$ and a factorization through $R$, we have $m \geq n$ as desired.

On the other hand, if $p$ and $q$ are not both in $V(H_1)$ or $V(H_2)$, then by symmetry, we can assume $p \in V(H_1) \setminus V(H_2)$ and $q \in V(H_2) \setminus V(H_1)$.   There is a connection in $H$ of size $|R|$, and this connection must restrict to a connection from $P$ to $R$ in $H_1$ which uses all the vertices of $R$.  Thus, $m_{H_1}(P, R \cup \{p\}) \geq |R|$, so it is either $|R|$ or $|R| + 1 = |R \cup \{p\}|$.  Assume that
\[
m_{H_1}(P, R \cup \{p\}) = |R| + 1 = m_{H_2}(R \cup \{q\}, Q).
\]
Then there is a connection in $H_1$ from $P$ to $R \cup \{p\}$ that uses all the vertices of $R \cup \{p\}$ and a connection in $H_2$ from $R \cup \{q\}$ to $Q$ that uses all the vertices of $R \cup \{q\}$.  Joining these connections together with the edge $e$ from $p$ to $q$ provides a connection of size $|R| + 1$ in $G$ from $P$ to $Q$.  On the other hand, we can factorize $[G]: P \to Q$ as
\[
[H_1 \cup e]: P \to R \cup \{q\}, \quad [H_2]: R \cup \{q\} \to Q.
\]
Therefore, $m \geq |R| + 1 \geq n$.

The only case that remains is when $p \in V(H_1) \setminus V(H_2)$ and $q \in V(H_2) \setminus V(H_1)$ and either $m_{H_1}(P, R \cup \{p\})$ or $m_{H_2}(R \cup \{q\}, Q)$ equals $|R|$ rather than $|R| + 1$.  Assume that $m_{H_1}(P, R \cup \{p\}) = |R|$ since the other case is symmetrical.  Then since $H_1$ has fewer edges, the induction hypothesis yields a factorization of $[H_1]: P \to R \cup \{p\}$ into
\[
[H_1']: P \to P' \qquad [H_1'']: P' \to R \cup \{p\}
\]
such that $|P'| = |R|$.  Then we may factorize $[G]: P \to Q$ into
\[
[H_1']: P \to P' \qquad [H_2 \cup e] \circ [H_1'']: P' \to R \cup \{p\} \to Q.
\]
This implies that $m \geq |R| = |P'| \geq n$, so we are done.
\end{proof}

The max-flow min-cut principle allows us to prove the following version of the rank-connection principle, which works for all finite networks, but only yields an inequality:
\begin{proposition}[Rank-Connection Principle 2]
Let $[\Gamma,i,j]: P \to Q$.   Then
\[
\rank X([\Gamma,i,j]) \leq 2 m([\Gamma,i,j]).
\]
\end{proposition}

\begin{proof}
By Proposition \ref{prop:maxflowmincut}, we can write $[\Gamma,i,j] = [\Gamma_2,i_2,j_2] \circ [\Gamma_1,i_1,j_1]: P \to R \to Q$.  Then
\[
\rank X([\Gamma_1,i_1,j_1]) \leq \dim (\F^R)^2 = 2|R| = 2 m([\Gamma,i,j]).
\]
Applying Lemma \ref{lem:rankinequality} to $X([\Gamma_2,i_2,j_2]) \circ X([\Gamma_1,i_1,j_1])$ completes the proof.
\end{proof}

\subsection{Semi-Elementary Factorizations} \label{subsec:semielementary}

We will now give a geometric characterization of when the rank-connection principle holds for all edge weights.  Inspired by Proposition \ref{prop:completelyreducible}, we want to extend the framework of elementary factorization to incorporate boundary wedge sums.  We define two types of ``semi-elementary'' networks which roughly correspond to attaching components to the input side or to the output side by boundary wedge-sums or disjoint unions.

\begin{definition}
A morphism $[G,i,j]: P \to Q$ is called {\bf type 3'} if every output vertex is also an input and any two outputs are in separate components of the graph.  {\bf Type 4'} is defined the same way with inputs and outputs switched.  Note that type 3' includes type 3 and type 4' includes type 4.
\end{definition}

\begin{definition}
A {\bf semi-elementary factorization} of $[G,i,j]: P \to Q$ is a factorization into morphisms of types 1, 2, 3', 4' such that the type 3' morphisms come before the type 4'.  The {\bf width} of a factorization into morphisms $P_k \to P_{k+1}$ is $\min_k |P_k|$.
\end{definition}

\begin{theorem}[Rank-Connection Principle 3] \label{thm:RC3}
Let $[G,i,j]: P \to Q$.  The following are equivalent:
\begin{enumerate}
	\item $[G,i,j]$ admits a semi-elementary factorization.
	\item For any network $\Gamma$ on $G$, we have $\rank X([\Gamma,i,j]) = 2 m([\Gamma,i,j])$.
\end{enumerate}
\end{theorem}

\begin{proof}
If $[G,i,j]$ admits a semi-elementary factorization, then a similar argument to Lemma \ref{lem:connectionwidth} shows that the maximum size connection is the same as the width of the factorization.

Next, note that if $[\Gamma,i,j]: P \to Q$ is type 3', then $X([\Gamma,i,j])$ is an epimorphism in the category of linear relations.  Indeed, we can achieve any potentials on $j(Q)$ using a harmonic function which is constant on each component of the graph.  Since $j(Q) \subset i(P)$, we can then achieve whatever output current we want on $j(Q)$ by cancelling it with the input current on $P$.  Similarly, type 4' morphisms produce monomorphisms in the category of linear relations.  Thus, the same argument as in Lemma \ref{lem:rankwidth} shows that if $[G,i,j]$ admits a semi-elementary factorization, then $\rank X([\Gamma,i,j])$ is twice the width.

These two steps complete (1) implies (2).  Now we prove (2) implies (1).  Let $\F$ be an infinite field.  Suppose $[G,i,j]$ does not admit a semi-elementary factorization.  By Proposition \ref{prop:maxflowmincut}, we can factorize $[G,i,j]$ into $[G_1,i_1,j_1]: P \to R$ and $[G_2,i_2,j_2]: R \to Q$ such that $|R| = m([G,i,j])$.  Since $[G,i,j]$ does not admit a semi-elementary factorization, we know that either
\begin{itemize}
	\item $[G_1,i_1,j_1]$ does not admit a factorization into types 1, 2, and 3', or
	\item $[G_2,i_2,j_2]$ does not admit a factorization into types 1, 2, and 4'.
\end{itemize}
By symmetry, we can assume the first case holds.  By Lemma \ref{lem:rankinequality}, it suffices to construct a network $\Gamma_1$ on $G_1$ such that $\rank X([\Gamma_1,i_1,j_1]) < 2|R|$.  In fact, it suffices to arrange that $\ker \overline{X([\Gamma_1,i_1,j_1])}$ is nonzero, that is, there is nonzero data on $R$ compatible with zero data on $P$.

We now follow the same strategy as in Proposition \ref{prop:completelyreducible}.  If $G_1$ has any boundary spikes, boundary edges, or isolated boundary vertices on the input side, or if it is possible to ``break off'' components from the input side using boundary wedge-sums or disjoint unions, then we first remove them.  Precisely, we factor $[G_1,i_1,j_1]$ as
\[
[G^*, i^*, j^*] \circ [G_n',i_n',j_n'] \circ \dots \circ [G_1',i_1',j_1'],
\]
where the $[G_k',i_k',j_k']$ are morphisms of type 1, type 2, or type 3', such that it is no longer possible to factor out any more elementary morphisms from the beginning of $[G^*, i^*, j^*]$.  Then it suffices to find edge weights on $[G^*, i^*, j^*]$ such that $\ker \overline{X([\Gamma^*, i^*, j^*])} \neq 0$.

Let $P^*$ and $Q^*$ be the inputs and outputs of $[G^*,i^*,j^*]$ considered as literal subsets of $V(G^*)$.  Assume there are no self-looping edges.  Let $S$ be the set of edges which are contained in a cycle \emph{or} a path from $Q^*$ to $Q^*$ with no self-intersections.  We define $u$ to be zero on any component of $G^* \setminus S$ that contains a vertex of $P^*$, and set it to a different nonzero value on each of the other components.

Choose a collection of sets $S_j$ which are cycles or paths from $Q^*$ to $Q^*$ such that $S = \bigcup_j S_j \cup \overline{S}_j$.  As in Proposition \ref{prop:completelyreducible}, we can choose $w_j$ which is nonzero on $S_j \cup \overline{S}_j$ such that $\Delta_{w_j} u = 0$ on $V \setminus Q^*$.  Since $\F$ is infinite, we can also choose $\alpha_j$ such that $\sum_j \alpha_j w_j$ is nonzero on all of $S$ simultaneously.  We then set $w = 1_{G^* \setminus S} + \sum_j \alpha_j w_j$.  This ensures that $u|_{P^*} = 0$ and $\Delta u |_{V(G^*) \setminus Q^*} = 0$.  Then we can choose a compatible assignment of input and output currents which will make the data on $P^*$ identically zero.

It only remains to arrange that the data on $Q^*$ is not identically zero.

{\bf Case 1:}  Suppose $Q^* \subset P^*$.  Then since $[G^*,i^*,j^*]$ is not type 3', there must be some path connecting distinct vertices $q_1 \in Q^*$ and $q_2 \in Q^*$.  We can assume this path is one of the $S_j$'s.  Then $\Delta_{w_j} u(q_1) \neq 0$ since nonzero current flows along the path.  Since $\F$ is infinite, we can choose the $\alpha_j$'s such that $\Delta_w u(q_1)$ is still nonzero.  Then since the input current on $q_1 \in P^*$ must be zero, the output current is nonzero.

{\bf Case 2:} Suppose there exists $q \in Q^* \setminus P^*$.  Then we claim that $u(q) \neq 0$.  It suffices to show that any vertex of $P^*$ is in its own different component of $G^* \setminus S$, since $u$ was only zero on the components with vertices of $P^*$.  We claim that any edge with one endpoint in $P^*$ must be contained in a cycle or in a path from $Q^*$ to $Q^*$ and hence is in $S$.  Choose $e$ with $e_- = p \in P^*$ and $e_+ = r$.
\begin{itemize}
	\item Suppose $p \in P^* \cap Q^*$.  Then because it is not possible to factor out a type 3' network, we know that deleting $p$ does not create any components that are disconnected from $Q^*$.  Thus, there is some path from $r$ to $Q^*$ which does not use $p$.  This implies there is a path from $p$ to a different vertex in $Q^*$ which uses $e$.
	\item On the other hand, suppose that $p \in P^* \setminus Q^*$.  Since $e$ is not a boundary spike on the input side, there must be some other edge $e'$ with $e_-' = e_-$.  If $e_+' = e_+$, then they form a cycle already.  Otherwise, let $r' = e_+'$.  There exist paths from $r$ and $r'$ to $Q^*$.  Joining these paths with $e$ and $e'$ produces a path from $Q^*$ to $Q^*$.  If the path has any self-intersections, then we can choose a subset containing $e$ and $e'$ which is either a cycle or a path from $Q^*$ to $Q^*$ with no self-intersections.
\end{itemize}
This completes case 2 and hence the proof.
\end{proof}

\subsection{Unique Complete Connections}

The following theorem relates boundary data, connections, elementary factorizations, and scaffolds.  It characterizes, for instance, existence and uniqueness for mixed-data boundary value problems in terms of connections.  Let us call a connection from $P$ to $Q$ {\bf complete} if it uses all the vertices in $P$ and all the vertices in $Q$.

\begin{theorem} \label{thm:uniquefull}
Let $G$ be a finite $\partial$-graph with no self-loops.  Assume $\partial V = P \cup Q$ with $|P| = |Q| = n$.  Let $P' = P \setminus Q$ and $Q' = Q \setminus P$.  Let $[G,i,j]: P \to Q$ be the corresponding IO-graph morphism.  The following are equivalent:
\begin{enumerate}[a.]
	\item There is exactly one complete connection from $P$ to $Q$ and this connection uses all the interior vertices.
	\item There is a scaffold $S$ on $G$ such that $V \setminus S_- = P$ and $V \setminus S_+ = Q$.
	\item $[G,i,j]$ admits an elementary factorization into type 1 and type 2 networks.
	\item $\rank X([\Gamma,i,j]) = 2n$ for all edge weights and $G$ has no nontrivial harmonic subgraphs with one boundary vertex.
	\item For any $v \in \F^P$ and $w \in \F^{P'}$, there is a unique harmonic function $u$ with $u|_P = v$ and $\Delta u|_{P'} = w$ and every interior vertex has degree at least $2$.
\end{enumerate}
\end{theorem}

\begin{proof}
(b) and (c) are equivalent by Proposition \ref{prop:IOscaffolds}.

(e) $\implies$ (d).  Note that (e) implies $\rank X([\Gamma,i,j]) = 2n$.  Hence, $G$ has a semi-elementary factorization.  But (e) also implies that a harmonic function is uniquely determined by its boundary data, and hence $G$ is completely reducible.  This implies that any harmonic subgraph is also completely reducible.  The only completely reducible $\partial$-graphs with one boundary vertex are trees (graphs with no cycles).  Any nontrivial tree would have to have an interior vertex with degree $1$, which contradicts our assumption.  Thus, (d) holds.

(d) $\implies$ (c).  By Theorem \ref{thm:RC3}, $[G,i,j]$ admits a semi-elementary factorization.  Since $\rank X([\Gamma,i,j]) = 2n$ for any edge weights, we know the maximum size connection is $n$.  Thus, in the factorization every type 3' or type 4' IO-graph must have $n$ inputs and $n$ outputs.  This implies that any component of the type 3' or type 4' morphism must be a harmonic subgraph of $G$ with one boundary vertex.  Thus, it is trivial, so in fact, there are no nontrivial type 3' or type 4' morphisms, so (c) holds.

(c) $\implies$ (e).  We know that $\rank X([\Gamma,i,j]) = 2n$ by Theorem \ref{thm:RC3}, we know that for any $v \in \F^P$ and $w \in \F^{P'}$, there \emph{exists} a harmonic function $u$ with $u|_P = v$ and $\Delta u|_{P'} = w$, and the boundary data on $Q$ is uniquely determined by the boundary data on $P$.  But $G$ must be layerable, hence completely reducible, so any harmonic function is uniquely determined by its boundary data.  Thus, the first condition of (e) holds.  For the second condition, note that every interior vertex must have degree at least $2$ since it is contained in a path from $P$ to $Q$ along the edges in the type 1 networks.

(c) $\implies$ (a).  This follows from a straightforward induction on the number of type 1 and type 2 factors.

(a) $\implies$ (b).  Let $S$ be the set of oriented edges in the paths of the unique complete connection.  Since it is assumed to use all the interior vertices, we have $V \setminus S_- = P$ and $V \setminus S_+ = Q$.  Thus, there are no non-input or non-output interior edges.  The graph is also finite, so the only scaffold condition left to check is that there is no increasing path which forms a cycle.  The idea is that if we had such a loop, then we could construct a different connection between $P$ and $Q$ as shown in Figure \ref{fig:newconnection}.

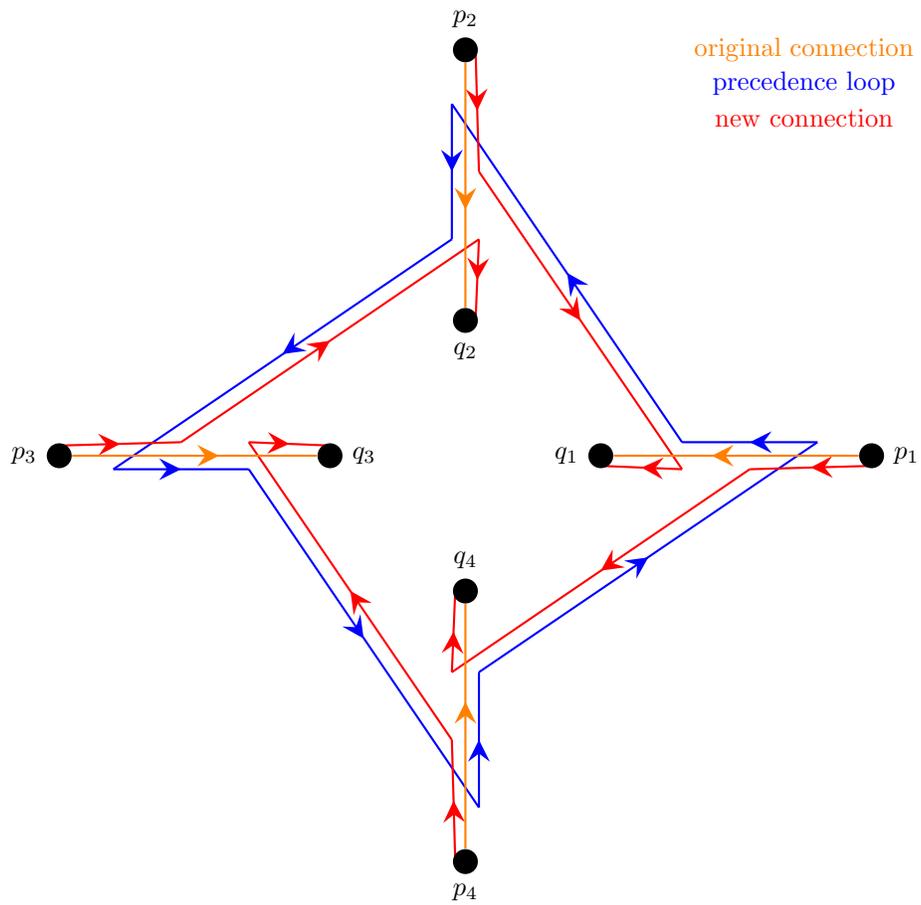
\begin{figure}
\caption{Proof of Theorem \ref{thm:uniquefull}: (a) $\implies$ (b)}  \label{fig:newconnection}

\begin{center}
\begin{tikzpicture}[scale = 0.9]

	\begin{scope}[thick,blue]
		\draw (5.2,0.2) -- (3.2,0.2) [arrow inside={end=stealth,opt={blue, scale=2}}{0.5}];
		\draw (3.2,0.2) -- (-0.2,5.2) [arrow inside={end=stealth,opt={blue, scale=2}}{0.5}];
		\draw (-0.2,5.2) -- (-0.2,3.2) [arrow inside={end=stealth,opt={blue, scale=2}}{0.5}];
		\draw (-0.2,3.2) -- (-5.2,-0.2) [arrow inside={end=stealth,opt={blue, scale=2}}{0.5}];
		\draw (-5.2,-0.2) -- (-3.2,-0.2) [arrow inside={end=stealth,opt={blue, scale=2}}{0.5}];
		\draw (-3.2,-0.2) -- (0.2,-5.2) [arrow inside={end=stealth,opt={blue, scale=2}}{0.5}];
		\draw (0.2,-5.2) -- (0.2,-3.2) [arrow inside={end=stealth,opt={blue, scale=2}}{0.5}];
		\draw (0.2,-3.2) -- (5.2,0.2) [arrow inside={end=stealth,opt={blue, scale=2}}{0.5}];
	\end{scope}

	\begin{scope}[thick,red]
		\draw (0.15,6) to (0.2,4.2) [arrow inside={end=stealth,opt={red, scale=2}}{0.5}];
		\draw (0.2,4.2) to (3.2,-0.2) [arrow inside={end=stealth,opt={red, scale=2}}{0.5}];
		\draw (3.2,-0.2) to (2,-0.15) [arrow inside={end=stealth,opt={red, scale=2}}{0.5}];
		\draw (-6,0.15) to (-4.2,0.2) [arrow inside={end=stealth,opt={red, scale=2}}{0.5}];
		\draw (-4.2,0.2) to (0.2,3.2) [arrow inside={end=stealth,opt={red, scale=2}}{0.5}];
		\draw (0.2,3.2) to (0.15,2) [arrow inside={end=stealth,opt={red, scale=2}}{0.5}];
		\draw (-0.15,-6) to (-0.2,-4.2) [arrow inside={end=stealth,opt={red, scale=2}}{0.5}];
		\draw (-0.2,-4.2) to (-3.2,0.2) [arrow inside={end=stealth,opt={red, scale=2}}{0.5}];
		\draw (-3.2,0.2) to (-2,0.15) [arrow inside={end=stealth,opt={red, scale=2}}{0.5}];
		\draw (6,-0.15) to (4.2,-0.2) [arrow inside={end=stealth,opt={red, scale=2}}{0.5}];
		\draw (4.2,-0.2) to (-0.2,-3.2) [arrow inside={end=stealth,opt={red, scale=2}}{0.5}];
		\draw (-0.2,-3.2) to (-0.15,-2) [arrow inside={end=stealth,opt={red, scale=2}}{0.5}];
	\end{scope}

	\node[text = orange] at (5,6) {original connection};
	\node[text = blue] at (5,5.5) {precedence loop};
	\node[text = red] at (5,5) {new connection};

	\node[circle,fill] (P1) at (0:6) [label = right: $p_1$] {};
	\node[circle,fill] (P2) at (90:6) [label = above: $p_2$] {};
	\node[circle,fill] (P3) at (180:6) [label = left: $p_3$] {};
	\node[circle,fill] (P4) at (270:6) [label = below: $p_4$] {};

	\node[circle,fill] (Q1) at (0:2) [label = left: $q_1$] {};
	\node[circle,fill] (Q2) at (90:2) [label = below: $q_2$] {};
	\node[circle,fill] (Q3) at (180:2) [label = right: $q_3$] {};
	\node[circle,fill] (Q4) at (270:2) [label = above: $q_4$] {};

	\draw[thick,orange] (P1) to (Q1) [arrow inside={end=stealth,opt={orange, scale=2}}{0.6}];
	\draw[thick,orange] (P2) to (Q2) [arrow inside={end=stealth,opt={orange, scale=2}}{0.6}];
	\draw[thick,orange] (P3) to (Q3) [arrow inside={end=stealth,opt={orange, scale=2}}{0.6}];
	\draw[thick,orange] (P4) to (Q4) [arrow inside={end=stealth,opt={orange, scale=2}}{0.6}];

\end{tikzpicture}
\end{center}

\end{figure}

To make this rigorous, consider the increasing loops with the minimal number of oriented edges not in $S$, and from those, choose one with the minimal number of oriented edges overall.  Let $\alpha_1, \dots, \alpha_n$ be the paths in the connection.  Then observe:
\begin{itemize}
	\item Any loop must contain some edges not in $S \cup \overline{S}$, since otherwise it would have to be contained in one of the $\alpha_m$'s, which is impossible.
	\item In the chosen loop the oriented edges must be distinct, since otherwise we could find a smaller loop.
	\item Call our loop $e_1$, $e_2$, \dots $e_K$.  Suppose there are some $i < j < k$ where $e_i$ and $e_k$ are edges in $S$ in the same path $\alpha_m$, and $e_j$ is not in $\alpha_m$ and that $e_i$ comes before $e_k$ in the path $\alpha_m$.  If we replace the segment $e_{i+1} \dots e_{k-1}$ of the loop with the segment of $\alpha_m$ from $e_i$ to $e_k$, then we get a loop with fewer edges not in $S \cup \overline{S}$.  Thus, this cannot happen in our chosen loop.  The same reasoning holds for any cyclic permutation of the indices $1,\dots,K$.  Thus, the loop must intersect each path in an ``interval''; that is, $I_m = \{k: e_k \in \alpha_m \}$ is of the form $\{1,\dots,k\}$ after some cyclic permutation of the indices.
\end{itemize}
Hence, our loop has the following form: It moves forward along some path of the connection (which we will call $\alpha_1$ after reindexing), then crosses by an edge not in $S \cup \overline{S}$ to some other path $\alpha_2$, and it continues in the same way until it crosses from some $\alpha_\ell$ back to $\alpha_1$.  The paths $\alpha_1$,\dots,$\alpha_\ell$ are distinct.  It follows that the vertices in our loop must be distinct and the loop looks essentially like the one portrayed in the Figure except that it might not visit every path of the connection.  If the remaining paths are $\alpha_{\ell+1}, \dots, \alpha_n$, then we construct our new complete connection as follows: $\alpha_j' = \alpha_j$ for $j = \ell + 1, \dots, n$.  For $j = 1,\dots, \ell$, $\alpha_j'$ follows $\alpha_j$ until it meets an endpoint of a plank from the loop, then it crosses along the loop over to $\alpha_{j-1}$, and it continues along $\alpha_{j-1}$ until it reaches $Q$ (indices written mod $\ell$).  This contradicts the uniqueness of the original complete connection, so in fact, we have a scaffold.
\end{proof}

The following much weaker corollary is surprising in itself:

\begin{corollary}
Suppose that $P', Q' \subset \partial V(G)$.  Suppose that there is exactly one complete connection from $P'$ to $Q'$ and that it uses all the interior vertices.  Then $G$ is layerable.
\end{corollary}

\section{Symplectic Properties} \label{sec:symplectic}

The relationship between electrical networks and symplectic forms is well-known, especially among physicists.  The results of this section draw on \cite{BF} and \cite{LP} (see also \cite{CS}).  We characterize the possible subspaces of $\F^n \times \F^n$ that arise as the boundary behavior of a networks, and the linear relations that arise as $X([\Gamma,i,j])$.  This is analogous to the characterization of response matrices as being symmetric with row sums zero, but slightly harder to prove.

We also characterize the electrical linear group defined in \S \ref{subsec:IOlayerstripping}.  We discuss the star-mesh transformation over arbitrary fields and use it to show that any Lagrangian subspace of $\F^n \times \F^n$ can be represented as the boundary behavior of a circular planar network with nonzero edge weights in $\F$, for a field $\F \neq \F_2$.

\subsection{Symplectic Vector Spaces and Relations}

We recall the following definitions from the theory of symplectic vector spaces.

A bilinear form $\omega$ on a vector space $V$ is called {\bf symplectic} if it is non-degenerate, that is,
\[
\omega(x,y) = 0 \text{ for all } y \implies x = 0
\]
and totally isotropic, that is,
\[
\omega(x,x) = 0 \text{ for all } x.
\]
A {\bf symplectic vector space} is a vector space $V$ equipped with a symplectic form.  The {\bf standard symplectic form} on $\F^n \times \F^n$ is
\[
\omega((x_1,x_2), (y_1,y_2)) = \ip{x_2, y_1} - \ip{x_1, y_2} = \begin{pmatrix} x_1^T & x_2^T \end{pmatrix} \begin{pmatrix} 0 & -1 \\ 1 & 0 \end{pmatrix} \begin{pmatrix} y_1 \\ y_2 \end{pmatrix},
\]
where $\ip{-,-}$ is the ``inner product'' $\ip{x,y} = \sum_j x_j y_j$.  We will use the same symplectic form on $\F^P \times \F^P$ for $P$ a finite set.  It turns out that any symplectic form can be written in this way after a change of basis and hence a symplectic vector space must be even-dimensional (reference).

If $V$ is a symplectic vector space of dimension $2n$, then we say a subspace $U$ is {\bf isotropic} if $\omega|_{U \times U} = 0$ and {\bf Lagrangian} if it is isotropic with dimension $n$.  A linear relation $T: V \rightsquigarrow W$ is called {\bf symplectic} if $\omega_V(v,v') = \omega_W(w,w')$ whenever $(v,w) \in T$ and $(v',w') \in T$.  Equivalently, $T$ is symplectic if it is an isotropic subspace of $V \times W$ with respect to the symplectic form
\[
\omega((v,w), (v',w')) = \omega_V(v,v') - \omega(w,w').
\]
In particular, a linear transformation $T: V \to W$ is symplectic if and only if it preserves the symplectic form.  We say a linear relation $T$ is {\bf Lagrangian} if it is a Lagrangian subspace of $V \times W$ with respect to the symplectic form given above.

\subsection{Symplectic Characterization of Boundary Behavior} \label{subsec:symplecticchar}

Our goal in this section is to prove
\begin{theorem} \label{thm:symplectic}
Let $T: (\F^m)^2 \rightsquigarrow (\F^n)^2$.  Then $T = X([\Gamma,i,j])$ for some IO-network morphism $[\Gamma,i,j]: [m] \to [n]$ if and only if $T$ is Lagrangian and contains the vector $((1,\dots,1,0,\dots,0), (1,\dots,1,0,\dots,0))$.
\end{theorem}

The forward direction of this theorem is proved in \cite{BF} modulo changes of notation.  However, the converse direction, which requires us to construct an electrical network for any Lagrangian relation containing the constant-potential vector, will be proved using layer-stripping and the electrical linear group.  As we shall see, the Lagrangian property has to do with the fact that the Laplacian is a symmetric matrix, and the property concerning $(1,\dots,1,0,\dots,0)$ has to do with the fact that constant functions are harmonic.

This theorem contains the following proposition as a special case when $n = 0$, which characterizes the electrical Grassmannian $EG_m$:
\begin{proposition} \label{prop:symplectic1}
Let $\Lambda \subset (\F^m)^2$.  Then $\Lambda$ is the boundary behavior of some network with $\partial V = [m]$ if and only if $\Lambda$ is a Lagrangian subspace and contains $(1,\dots,1,0,\dots,0)$.
\end{proposition}

Another important special case is when $m = n$ and $T$ is an invertible transformation:

\begin{proposition} \label{prop:symplectic2}
Let $T$ be an invertible linear transformation $T: (\F^n)^2 \to (\F^n)^2$.  Then $T = X([\Gamma,i,j])$ for some $[\Gamma,i,j]: [n] \to [n]$ if and only if $T$ is symplectic and preserves the vector $(1,\dots,1,0,\dots,0)$.
\end{proposition}

Before presenting the crux of the proof of Theorem \ref{thm:symplectic}, we first reduce to the special case stated in Proposition \ref{prop:symplectic1} by a tedious computation:
\begin{lemma}
Theorem \ref{thm:symplectic} foillows from Proposition \ref{prop:symplectic1}.  More precisely, each implication of Theorem \ref{thm:symplectic} follows from the corresponding implication in Proposition \ref{prop:symplectic1}.
\end{lemma}

\begin{proof}
Suppose that $T = X([\Gamma,i,j])$ for $[\Gamma,i,j]: [m] \to [n]$.  Set
\begin{align*}
\partial V &= i([m]) \cup j([n]) \\
P &= i([m]) \setminus j([n]) \\
Q &= j([n]) \setminus i([m]) \\
R &= i([m]) \cap j([n]).
\end{align*}
Suppose that $(x,y)$ and $(x',y')$ in $X([\Gamma,i,j])$ correspond to harmonic functions with boundary data $(\phi,\psi)$ and $(\phi',\psi')$ in $\Lambda(\Gamma)$.  Assuming that $\Lambda(\Gamma)$ is Lagrangian, we have
\[
\ip{\phi', \psi} - \ip{\phi, \psi'} = 0.
\]
Now consider
\begin{align*}
\omega((x,y), (x',y')) &= \omega(x,x') - \omega(y,y') \\
&= \ip{x_1',x_2} - \ip{x_1,x_2'} - \ip{y_1',y_2} + \ip{y_1,y_2'}
\end{align*}
We split the computation into the sets $P$, $Q$, and $R$ to obtain
\begin{multline*}
\ip{x_1'|_P, x_2|_P} + \ip{x_1'|_R, x_2|_R} - \ip{x_1|_P, x_2'|_P} - \ip{x_1|_R,x_2'|_R} \\
- \ip{y_1'|_Q,y_2|_Q} - \ip{y_1'|_R,y_2|_R} + \ip{y_1|_Q, y_2'|_Q} + \ip{y_1|_R, y_2'|_R},
\end{multline*}
where by a slight abuse of notation, we have identified $P$, $Q$, and $R$ with subsets of $[m]$ and $[n]$.  Recall that by definition of $X$, we have
\[
x_1|_R = y_1|_R = \phi|_R \qquad y_2|_R - x_2|_R = \psi|_R
\]
as well as
\[
x_1|_P = \phi|_P, \quad y_1|_Q = \phi|_Q, \quad x_2|_P = -\psi|_P, \quad y_2|_Q = \psi_Q.
\]
The same holds for $x'$, $y'$, $\phi'$, $\psi'$.  Substituting this shows that
\[
\omega((x,y), (x',y')) = -\ip{\phi', \psi} + \ip{\phi, \psi'} = 0.
\]
Thus, $X([\Gamma,i,j])$ is isotropic.

To compute the dimension of $X([\Gamma,i,j])$, note that $(x,y) \mapsto (\phi,\psi)$ is a surjective map $X([\Gamma,i,j]) \to \Lambda$.  The kernel consists of $(x,y)$ such that $x_1 = 0 = y_1$, $x_2|_P = 0$, $y_2|_Q = 0$, and $x_2|_R = y_2|_R$.  Thus, the kernel has dimension $R$.  Therefore, using the assumption that $\Lambda$ is Lagrangian
\[
\dim X([\Gamma,i,j]) = |R| + \dim \Lambda = |R| + |\partial V| = |R| + |P| + |Q| + |R| = m + n,
\]
and hence $X$ has the correct dimension.

Finally, because $(1,\dots,1,0,\dots,0)$ is in $\Lambda$, we know $X$ contains $(1,\dots,1,0,\dots,0) \times (1,\dots,1,0,\dots,0)$.

To prove the other implication of Theorem \ref{thm:symplectic}, suppose that $T$ is Lagrangian and contains $(1,\dots,1,0,\dots,0) \times (1,\dots,1,0,\dots,0)$.  Define $\Lambda \subset (\F^{m+n})^2$ by
\[
\Lambda = \{((x_1,y_1), (x_2,-y_2)): (x,y) \in T\}.
\]
Then $\Lambda$ is Lagrangian with respect to the standard symplectic form on $(\F^{m+n})^2$ and contains $(1,\dots,1,0,\dots,0)$.  We assume that $\Lambda$ can be represented as the boundary behavior of some electrical network $\Gamma$ with $\partial V = [m+n]$.  Then the same computation as before (with $R = \varnothing$) shows that $T$ is $X(\Gamma)$ with the obvious labelling of the boundary nodes.
\end{proof}

Now that that is out of the way, we prove the forward implication of Proposition \ref{prop:symplectic1}:
\begin{lemma}
If $\Lambda \subset (\F^n)^2$ is the boundary behavior of a network with $\partial V = [n]$, then
\begin{enumerate}[a,]
	\item $\Lambda$ is isotropic,
	\item $\dim \Lambda = n$.
	\item $\Lambda$ contains $(1,\dots,1,0,\dots,0)$.
\end{enumerate}
\end{lemma}

\begin{proof}
Let $\Gamma$ be a network with $\partial V = [n]$.

To prove (a), recall that the linear map $\Delta: \F^V \to \F^V$ is given by the matrix
\[
\Delta_{p,q} = \begin{cases} \sum_{e:e_+=p} w(e), & p = q \\ -\sum_{\substack{e: e_+ = p \\ e_- = q}} w(e), \end{cases}
\]
or equivalently
\[
\Delta_{p,q} = \begin{cases} \sum_{e:e_-=p} w(e), & p = q \\ -\sum_{\substack{e: e_- = p \\ e_+ = q}} w(e). \end{cases}
\]
The matrix $\Delta$ is symmetric, and hence
\[
\ip{u_1, \Delta u_2} = \ip{u_2, \Delta u_1}
\]
for any $u_1, u_2 \in \F^V$.  If $u_1$ and $u_2$ are harmonic on $V^\circ$, then
\[
\ip{u_1|_{\partial V}, \Delta u_2|_{\partial V}} = \ip{u_1, \Delta u_2} = \ip{u_2, \Delta u_1} = \ip{u_2|_{\partial V}, \Delta u_1|_{\partial V}}.
\]
Since $\Lambda$ consists of all pairs $\ip{u|_{\partial V}, \Delta u|_{|\partial V|}}$ for harmonic $u$, this implies that $\Lambda$ is isotropic.

To prove (b), let $\mathcal{U}$ be the space of harmonic functions on $\Gamma$.  Let $\Phi: \mathcal{U} \to \Lambda$ be given by $u \mapsto (u|_{\partial V}, \Delta u|_{\partial V})$.  Note that $\Phi$ is surjective by definition of $\Lambda$.  Let $\Delta_{V^\circ, V}$ be the submatrix of the Laplacian consisting of the rows corresponding to interior vertices.  Then
\[
\mathcal{U} = \ker \Delta_{V^\circ,V}.
\]
On the other hand, note that
\begin{align*}
\ker \Phi &= \{u \in \mathcal{U}: u|_{\partial V} = 0, \Delta u|_{\partial V} = 0\} \\
&\cong \{w \in \F^{V^\circ}: \Delta_{V,V^\circ} w = 0\} = \ker \Delta_{V,V^\circ},
\end{align*}
where $\Delta_{V,V^\circ}$ is the submatrix with columns indexed by the interior vertices.  By symmetry, $\Delta_{V,V^\circ}^T = \Delta_{V^\circ,V}$.  Therefore, when we apply the rank-nullity theorem to $\Phi$ and $\Delta_{V,V^\circ}$ and $\Delta_{V^\circ,V}$, we obtain
\begin{align*}
\dim \Lambda &= \dim \mathcal{U} - \dim \ker \Phi \\
&= \dim \ker \Delta_{V^\circ, V} - \dim \ker \Delta_{V,V^\circ} \\
&= (|V| - \rank \Delta_{V^\circ,V}) - (|V^\circ| - \rank \Delta_{V,V^\circ}) \\
&= |V| - |V^\circ| = |\partial V| = n.
\end{align*}
This completes (b), and (c) is trivial since the constant function $u \equiv 1$ is harmonic.
\end{proof}

The last lemma was straightforward for the most part.  The only subtlety is that sometimes $\dim U > \dim \Lambda$ because there can be harmonic functions with zero potential and zero current on the boundary, as remarked in \S \ref{subsec:completelyreducible}.

Now we consider the converse direction of Proposition \ref{prop:symplectic1}, showing that any Lagrangian subspace of $\F^{2n}$ containing $(1,\dots,1,0,\dots,0)$ can be realized as the boundary behavior of a network.  The first step is purely algebraic:

\begin{lemma} \label{lem:lagrangianpartition}
Suppose $V$ is a Lagrangian subspace of $\mathbb{F}^{2n}$.  For $S \subset [2n]$, let $\pi_S: \mathbb{F}^{2n} \to \mathbb{F}^S$ be the coordinate projection.  Then there is a partition of $[n]$ into two sets $S$ and $T$ such that
\begin{itemize}
	\item $\pi_S(x) = 0$ implies $\pi_{[n]}(x) = 0$ for $x \in V$.
	\item $\pi_{S \cup (n + T)}$ defines an isomorphism $V \to \mathbb{F}^{S \cup (n + T)}$.
\end{itemize}
\end{lemma}

In electrical language, the lemma says the following:

\begin{corollary}
If $\Gamma$ is a linear network over $\mathbb{F}$, then there is a partition of $\partial V$ into two sets $P$ and $Q$ such that potentials on $P$ and net currents on $Q$ uniquely determine the other boundary data.
\end{corollary}

\begin{proof}[Proof of Lemma \ref{lem:lagrangianpartition}]
Let $W = \{w \in \mathbb{F}^n: (0,w) \in V\}$.  If $(x,y) \in V$ and $w \in W$, then
\[
0 = \omega((x,y), (0,w)) = -\ip{x,w},
\]
and hence
\[
W \subset \pi_{[n]}(V)^\perp = \{w \in \mathbb{F}^n: \ip{w,x} = 0 \text{ for } x \in \pi_{[n]}(V)\}.
\]
However, note that $W \cong \ker(\pi_{[n]}|_V)$, hence by the rank-nullity theorem $\dim W + \dim \pi_{[n]}(V) = \dim V = n$.  We also know by the rank-nullity theorem that $\dim \pi_{[n]}(V) + \dim \pi_{[n]}(V)^\perp = n$ for any field.  Therefore, $W = \pi_{[n]}(V)^\perp$.

From basic linear algebra, we can choose $S \subset [n]$ such that $\pi_S: \mathbb{F}^n \to \mathbb{F}^S$ restricts to an isomorphism $\pi_{[n]}(V) \to \mathbb{F}^S$.  Let $T = [n] \setminus S$.  Since $W = \pi_{[n]}(V)^\perp$, this implies that $\pi_T: \mathbb{F}^n \to \mathbb{F}^T$ defines an isomorphism $W \to \mathbb{F}^T$ (details\footnote{$\pi_{[n]}(V) \cap (\mathbb{F}^T \times 0^S) = 0$ in $\mathbb{F}^n$, which implies $\pi_{[n]}(V) + (\mathbb{F}^T \times 0^S) = \mathbb{F}^n$ since $\dim \pi_{[n]}(V) = |S| = n - |T|$.  Hence taking orthogonal complements $\pi_{[n]}(V)^\perp \cap (\mathbb{F}^S \times 0^T) = 0$}).  This implies that $\pi_{S \cup (n + T)}: \mathbb{F}^{2n} \to \mathbb{F}^{S \cup (n + T)}$ defines an isomorphism $V \to \mathbb{F}^{S \cup (n + T)}$.  One way to see this is to by applying the five-lemma to the diagram
\[
\begin{tikzcd}
0 \arrow{d} \arrow{r} & W \arrow{d} \arrow{r} & V \arrow{d} \arrow{r} & \pi_{[n]}(W) \arrow{d} \arrow{r} & 0 \arrow{d} \\
0 \arrow{r} & \mathbb{F}^T \arrow{r} & \mathbb{F}^{S \cup (n + T)} \arrow{r} & \mathbb{F}^S \arrow{r} & 0.
\end{tikzcd}
\]
\end{proof}

\begin{lemma} \label{lem:lagrangian}
Let $\Lambda$ be a Lagrangian subspace of $\mathbb{F}^{2n}$ containing $(1,\dots,1,0,\dots,0)$.  Then $\Lambda$ is the boundary behavior of some linear network over $\mathbb{F}$.
\end{lemma}

\begin{proof}
Choose a partition of $[n]$ into two sets $S$ and $T$ as in the previous lemma.  By reindexing the coordinates, assume that $S = [\ell]$ for some $\ell \leq n$.  Let $m = n - \ell$.  Then we can choose a basis $x_1, \dots, x_n$ of $\Lambda$ such that
\[
\begin{pmatrix} | & \dots & | \\ x_1 & \dots & x_n \\ | & \dots & | \end{pmatrix} \text{ is of the form }
\begin{pmatrix}
I & 0 \\ * & 0 \\ * & * \\ 0 & I
\end{pmatrix},
\]
where the sizes of the blocks are
\[
\begin{pmatrix}
\ell \times \ell & \ell \times m \\ m \times \ell & m \times m \\ \ell \times \ell & \ell \times m \\ m \times \ell & m \times m
\end{pmatrix}.
\]
Then define $\Lambda' := \Xi_{\ell+1}(1) \dots \Xi_n(1)(\Lambda)$ and note that
\[
\Lambda' = \im \begin{pmatrix} I & E_{\ell+1,\ell+1} \\ 0 & I \end{pmatrix} \dots  \begin{pmatrix} I & E_{n,n} \\ 0 & I \end{pmatrix} \begin{pmatrix} I & 0 \\ * & 0 \\ * & * \\ 0 & I \end{pmatrix} = \im \begin{pmatrix} I & 0 \\ * & I \\ * & * \\ 0 & I \end{pmatrix} = \im \begin{pmatrix} I & 0 \\ 0 & I \\ * & * \\ * & I \end{pmatrix},
\]
which can be written as
\[
V' = \im \begin{pmatrix} I \\ L \end{pmatrix}
\]
with $n \times n$ blocks.

Since we have already proved the forward direction of Proposition \ref{prop:symplectic1} and hence Theorem \ref{thm:symplectic}, we know $\Xi_j(t)$ is symplectic and fixes $(1,\dots,1,0,\dots,0)$.  This can also be verified by direct computation.  In any case, $\Lambda'$ is a Lagrangian subspace that contains $(1,\dots,1,0,\dots,0)$.  This implies $L$ is symmetric and has row sums zero.  Thus, $L$ has the form
\[
L =  \sum_{i < j} -L_{i,j}(E_{i,i} - E_{i,j} - E_{j,i} + E_{j,j}),
\]
and this implies that
\[
\begin{pmatrix} I \\ L \end{pmatrix} = \prod_{i < j} \begin{pmatrix} I & 0 \\ -L_{i,j}(E_{i,i} - E_{i,j} - E_{j,i} + E_{j,j}) & I\end{pmatrix} \begin{pmatrix} I \\ 0 \end{pmatrix},
\]
or in other words,
\[
\Lambda' = \prod_{i < j} \Xi_{i,j}(-L_{i,j})(\mathbb{F}^n \times 0^n),
\]
so that
\[
\Lambda = \prod_{k=\ell+1}^n \Xi_k(-1) \prod_{i < j} \Xi_{i,j}(L_{i,j}) (\mathbb{F}^n \times 0).
\]
Using the ideas of \S \ref{subsec:IOlayerstripping}, $\Lambda$ is the boundary behavior of the network obtained by taking $n$ isolated boundary vertices, adjoining boundary edges of conductances $L_{i,j}$ between vertices $i$ and $j$ whenever $L_{i,j} \neq 0$, and then adjoining boundary spikes of conductance $-1$ to the vertices $\ell + 1$, \dots, $n$.
\end{proof}

This completes the proof of Proposition \ref{prop:symplectic1} and hence Theorem \ref{thm:symplectic}.  The proof of the last lemma leads to the following corollaries:

\begin{corollary} \label{cor:equivalentlayerable}
Any $\Lambda \in EG_n(\mathbb{F})$ can be expressed as the boundary behavior of a layerable network with at most $\frac{1}{2}n(n-1) + 1$ edges.
\end{corollary}

\begin{proof}
In the previous proof, the number of boundary edges added was the number of nonzero entries of $\Lambda$ above the diagonal.  Since
\[
\Lambda = \begin{pmatrix} * & * \\ * & I \end{pmatrix},
\]
with the last block being $m \times m$, the number of nonzero entries above the diagonal is at most $\frac{1}{2}\ell(\ell - 1) + \ell m$.  The number of boundary spikes adjoined was $m$, so recalling $\ell + m = n$, the total number of edges is at most
\[
m + \frac{1}{2}(n - m)(n - m - 1) + (n - m)m = \frac{1}{2}n(n - 1) - \frac{1}{2} m(m - 3) \leq \frac{1}{2} n(n - 1) + 1.
\]
\end{proof}

\begin{remark*}
In the simple case when the Dirichlet-to-Neumann map exists, one can represent $\Lambda$ by a network on a complete graph.  Thus, the number of edges needed generically should be $\frac{1}{2}n(n-1)$.  The corollary says that even in degenerate cases, we can get away with at most one more edge.
\end{remark*}

\begin{corollary}
Let $Y = \{(i,j) \in [n] \times [n]: i < j\}$.  For $S \subset n$, define $\digamma_S: \mathbb{F}^Y \to EG_n(\mathbb{F})$ by
\[
\digamma_S((t_{i,j})) = \prod_{k \in S} \Xi_k(-1) \prod_{i < j} \Xi_{i,j}(t_{i,j})(\mathbb{F}^n \times 0^n).
\]
Then the images $U_S = \digamma_S(\mathbb{F}^Y)$ cover $EG_n(\mathbb{F})$ and the transition maps $\digamma_S^{-1} \circ \digamma_{S'}$ are rational functions.  In particular, for $\mathbb{F} = \R$ or $\C$, $EG_n$ is a smooth real/complex manifold of dimension $n(n-1)/2$.
\end{corollary}

\begin{proof}
The fact that the $U_S$'s cover $EG_n$ follows from the previous proofs, and the transition functions are rational because they can be computed in terms of multiplying and inverting matrices.
\end{proof}

\begin{remark*}
It is perhaps not surprising that $EG_n$ is a smooth manifold given by these symplectic equations for real or complex edge weights.  What is remarkable is that the same characterization works for any field, even fields which have no nice topological or algebraic properties.  It seems that this could only be proved by an elementary and explicit argument such as the one given here.
\end{remark*}

\subsection{Characterization of the Electrical Linear Group} \label{subsec:ELchar}

Proposition \ref{prop:symplectic2} showed that an invertible matrix arises as $X([\Gamma,i,j])$ for some network if and only if it is symplectic and preserves $(1,\dots,1,0,\dots,0)$.  In particular, the group $EL_n$ generated by matrices of the form $\Xi_k$ and $\Xi_{j,k}$ is \emph{contained} in the group of symplectic matrices which fix $(1,\dots,1,0,\dots,0)$.  However, we will show in this section that in fact $EL_n$ is \emph{equal} to this group.  The proof once again is elementary, but a bit tedious.  We will construct explicit factorizations in terms of the generators $\Xi_j$ and $\Xi_{i,j}$.  The argument works for any field other than $\F_2$, the field with two elements.

For brevity, we write
\[
c_0 = (1,\dots,1,0,\dots,0)
\]
and
\[
\Omega = \begin{pmatrix} 0 & -1 \\ 1 & 0 \end{pmatrix}.
\]
We recall that $A$ is symplectic if and only if $A^T \Omega A = \Omega$.

\begin{theorem} \label{thm:ELsymplectic}
Suppose $\mathbb{F} \neq \mathbb{F}_2$.  If $A$ is symplectic and $A c_0 = c_0$, then $A \in EL_n(\mathbb{F})$.
\end{theorem}

\begin{proof}
We proceed by induction on $n$.  For $n = 1$, any symplectic matrix $A$ that fixes $c_0$ must be of the form
\[
A = \begin{pmatrix} 1 & t \\ 0 & 1 \end{pmatrix} = \Xi_1(t).
\]

For the induction step, it suffices to find $A_1, \dots, A_\ell \in EL_n(\mathbb{F})$ such that
\[
A_1 \dots A_\ell A = \begin{pmatrix} * & 0 & * & 0 \\ 0 & 1 & 0 & 0 \\ * & 0 & * & 0 \\ 0 & 0 & 0 & 1 \end{pmatrix},
\]
where each ``$*$'' is $(n - 1) \times (n - 1)$.  Heuristically, $A_1 \dots A_\ell A$ is the behavior of IO-network where the $n$th input vertex equals the $n$th output, and this vertex is isolated; we are thus reducing to the case of networks with $n - 1$ inputs and outputs.  If we can find such matrices $A_1$, \dots $A_\ell$, then the matrix $A'$ formed by deleting the $n$th and $2n$th row and column of $A_1 \dots A_\ell A$ must be symplectic and fix $c_0 \in \mathbb{F}^{2(n-1)}$.  So by the induction hypothesis $A' \in EL_{n-1}(\mathbb{F})$, which implies $A \in EL_n(\mathbb{F})$.

Our first goal is to find $A_1$, \dots, $A_m$ generators of $EL_n(\mathbb{F})$ such that $A_m \dots A_1 A$ fixes $e_{2n}$ (the last column is $e_{2n}$).  Heuristically, $A_m \dots A_1 A$ corresponds to an IO-network where the $n$th input vertex is the same as the $n$th output, but is not necessarily an isolated vertex.  Let $x = A e_{2n}$; it suffices to show that by multiplying by elements of $EL_n$ we can map $x$ to $e_{2n}$.  There are several cases:
\begin{enumerate}
	\item Suppose that the ``potential'' $x_n \neq 0$ and that the ``net currents'' $x_{n+1}, \dots, x_{2n-1} \neq 0$.  Let
	\[
	y = \left( \prod_{k=1}^{n-1} \Xi_k(-x_k / x_{n+k}) \right) x.
	\]
	Then $y_1, \dots, y_{n-1} = 0$, $y_n = x_n \neq 0$.  Next, let
	\[
	z = \left( \prod_{k=1}^{n-1} \Xi_{k,n}(-y_{n+k}/y_n) \right) y.
	\]
	Then $z_1, \dots, z_{n-1} = 0$ and $z_{n+1}, \dots, z_{2n-1} = 0$.  But $\omega(c_0,z) = \omega(c_0,x) = 1$, so $z_{2n} = 1$.  Thus, multiplying by $\Xi_n(-z_n)$ will make the $n$th entry zero, yielding $e_{2n}$.
	\item If $x_n = 0$ but $x_{n+1}, \dots, x_{2n-1}, x_{2n} \neq 0$, then we can multiply by $\Xi_n(1)$ to make $x_n \neq 0$, then proceed to Case 1.
	\item Suppose that some of the ``currents'' $x_{n+1}, \dots, x_{n+k}$ are zero, but the ``potentials'' $x_1, \dots, x_n$ are not all equal.  For each $j$ with $x_{n+j} = 0$, we can find a $k$ with $x_j \neq x_k$.  Then multiply by some $\Xi_{j,k}(t)$ to make it nonzero.  In order to guarantee that the ``net current'' at $k$ is still nonzero, we choose $t \neq 0$ and $t \neq -x_{n+k}/(x_k - x_j)$.  This is possible because $\mathbb{F}$ has at least three elements.  Once we have done this for every $j$, proceed to Case 2.
	\item Suppose that $x_1, \dots, x_n$ are all equal to some constant $t$.  Since the vector $c_0$ is fixed by $A$ and all matrices in $EL_n$, it is not possible that $x_{n+1}, \dots, x_{2n}$ are all zero.  Hence, there is some $x_{n+k} \neq 0$, and we can multiply by some $\Xi_k(1)$ to make the new $x_k \neq t$.  Then proceed to Case 3.
\end{enumerate}
Thus, if we let $A_1, \dots, A_m$ be the matrices used in the above operations and $B = A_m \dots A_1 A$, then $Be_{2n} = e_{2n}$.

Our next task is find $A_{m+1}$, \dots, $A_\ell$ such that $A_\ell \dots A_{m+1} B$ fixes both $e_{2n}$ and $e_n$.  Let $x = B e_n$, and consider the following cases:
\begin{enumerate}
	\item Suppose that the ``net currents'' $x_{n+1}, \dots, x_{2n}$ are all nonzero.  Observe
	\[
	x_n = \omega(e_{2n},x) = \omega(B e_{2n}, B e_n) = \omega(e_{2n}, e_n) = 1.
	\]
	Let
	\[
	y = \left( \prod_{k=1}^{n-1} \Xi_k(x_k/x_{n+k}) \right) x,
	\]
	so that $y_1, \dots, y_{n-1} = 0$ and $y_n = 1$.  Then let
	\[
	z = \left( \prod_{k=1}^{n-1} \Xi_{k,n}(-y_{n+k}) \right) y.
	\]
	Then $z_1 = y_1, \dots, z_n = y_n$, and $z_{n+1}, \dots, z_{2n-1} = 0$.  But $\omega(c_0,z) = \omega(c_0,e_n) = 0$, so $z_{2n} = 0$ as well.  Hence, $z = e_n$.
	\item If some of ``currents'' $x_{n+1}, \dots, x_{n+k}$ are zero, but the ``potentials'' $x_1, \dots, x_n$ are not all equal, we can multiply by $\Xi_{j,k}(t)$'s to make all the ``currents'' nonzero (as in the previous part of the proof).  Then proceed to Case 1.
	\item Suppose that $x_1, \dots, x_n$ are all equal to $1$.  One of the ``net currents'' must be nonzero; so in fact, at least two of them are nonzero.  Hence, we can multiply by $\Xi_k(1)$ for some $k \neq n$ to make the new $x_k \neq 1$.  Then proceed to Case 2.
\end{enumerate}
In all these cases, we never multiplied by a $\Xi_n(t)$ matrix.  Thus, if we let $A_{m+1}$, \dots, $A_\ell$ be the matrices used in the above operations, then each one fixes $e_{2n}$, and thus
\[
C = A_\ell \dots A_{m+1} B = A_\ell \dots A_1 A
\]
also fixes $e_{2n}$, besides fixing $e_n$.

Because $C^T \Omega C = \Omega$, we know $C^T = \Omega C^{-1} \Omega^{-1}$.  Since $C^{-1}$ fixes $e_n$ and $e_{2n}$, we know $C^T$ fixes $\Omega e_n = e_{2n}$ and $\Omega e_{2n} = -e_n$.  Thus, the $n$th and $2n$th rows of $C$ are $e_n$ and $e_{2n}$, and so are the $n$th and $2n$th columns.  Thus, $C$ has the desired form and the induction step is complete.
\end{proof}

\begin{remark*}
The theorem fails in the case of $\mathbb{F}_2$.  For instance, for $n = 2$,
\[
\begin{pmatrix} 1 & 0 & 1 & 1 \\ 0 & 1 & 1 & 1 \\ 0 & 0 & 1 & 0 \\ 0 & 0 & 0 & 1 \end{pmatrix} \not \in EL_2(\mathbb{F}_2)
\]
despite being symplectic and fixing $c_0$.  An easy way to see this is to compute the orbit of $e_4$ under the action of $EL_2(\mathbb{F}_2)$ on $\mathbb{F}_2^4$; the orbit has only four elements and does not contain $e_1 + e_2 + e_4$, which is the last column of the matrix of above.\footnote{I have not worked out precisely what happens for $\mathbb{F}_2$, but might do it later.  This would be a good problem for REU students.}
\end{remark*}

As with $EG_n(\mathbb{F})$, the construction in Theorem \ref{thm:ELsymplectic} provides parametrizations of $EL_n(\mathbb{F})$ for which the transition functions are rational.  For a given $A$, we parametrize a ``neighborhood'' using the parameters for Case 1 of each step, keeping the parameters in the other steps fixed.  From this, we work out that the ``dimension'' of $EL_n(\mathbb{F})$ is $n(2n - 1)$, which is the same as for $EG_{2n}(\mathbb{F})$.

The action of $EL_n(\mathbb{F})$ on $EG_n(\mathbb{F})$ is transitive; indeed, the proof of Lemma \ref{lem:lagrangian} showed that every element of $EG_n(\mathbb{F})$ is in the orbit of $\mathbb{F}^n \times 0$.  However, the action is not faithful:  There exist nontrivial elements of $EL_n$ which fix every element of $EG_n$.  These elements are the kernel of the homomorphism $\Upsilon$ from $EL_n$ to the group of bijections $EG_n \to EG_n$ given by $\Xi \mapsto F_\Xi$, where $F_\Xi: EG_n \to EG_n: L \mapsto \Xi(L)$.  The reader can verify that (for $\mathbb{F} \neq \mathbb{F}_2$) the kernel consists of matrices of the form
\[
\begin{pmatrix} I + \mathbf{1} \alpha^T & \mathbf{1} \beta^T + \beta \mathbf{1}^T \\ 0 & I - \alpha \mathbf{1}^T \end{pmatrix},
\]
where $\mathbf{1}$ is the vector with every entry $1$ and $\alpha, \beta \in \R^n$ with $\sum_{k=1}^n \alpha_k = 0$.

\subsection{Network Planarization}

Given a network, we want to find a circular planar network with the same boundary behavior.  This has long been a goal of electrical engineers, who desired to print out flat circuit components with certain behavior.  For instance, \cite{vanLier} suggests using the star-mesh transformation to find planar equivalents.  Thanks to \cite{CIM} Theorem 4 (and related results), we now know exactly what response matrices can occur for circular planar networks with positive linear conductances, which ought to be the end of the matter as far as positive edge weights are concerned.  Many non-planar networks with positive real edge weights cannot have the same boundary behavior as a circular planar network with positive edge weights.

However, if we allow negative edge weights, it is much easier to planarize a network.  The REU paper \cite{KS} conjectured that any real response matrix could be represented by a circular planar network with signed real conductances, and \cite{MG} and \cite{WJ} suggest using the star-mesh transformation with signed conductances.  This turns out to be true for all fields other $\mathbb{F}_2$, as we will prove in Theorem \ref{thm:planarequivalent} below.

We will use the electrical linear group and the star-mesh transformation.  We first review the star-mesh transformation described in \cite{JR} and \cite{DI}, generalizing to arbitrary fields.  The {\bf $n$-star} $\bigstar_n$ is the $\partial$-graph with $n$ boundary vertices $\{1,\dots,n\}$ and one interior vertex $0$, and edges from the interior vertex to each of the boundary vertices.  The {\bf $n$-mesh} $M_n$ is the graph with $n$ boundary vertices, no interior vertices, and edges between any two boundary vertices.

The star-mesh transformation replaces a network on $\bigstar_n$ with a network on $M_n$ with the same boundary behavior and vice versa (if possible).  Using the principle of subnetwork splicing described in \ref{subsec:subgraphs}, we can replace a star subnetwork in a larger network with a mesh subnetwork without affecting the boundary behavior.

The star-mesh transformation in the special case $n = 4$ produces relations between the generators of $EL_n$.  This reduces our original set of generators of $EL_n$ to a smaller set of generators corresponding to ``circular planar'' operations of adjoining boundary spikes and boundary edges between boundary vertices with \emph{adjacent indices} (Lemma \ref{lem:relations} and Proposition \ref{prop:generators}).  We already know that any boundary behavior can be represented by a layerable network, and Proposition \ref{prop:generators} allows us to replace any operation of adjoining a boundary edge with an equivalent operation that preserves network planarity and thus to prove Theorem \ref{thm:planarequivalent}.

\begin{lemma}[adaptation of \cite{JR}] \label{lem:stark}
Let $n \geq 3$.
\begin{itemize}
	\item Consider a network $\Gamma$ on $\bigstar_n$ whose $j$th edge weight is $a_j$.  Then $\Gamma$ has the same boundary behavior as a network on $M_n$ if and only if $\sigma = \sum_{j=1}^n a_j \neq 0$.  In this case, the edge weights on $M_n$ are given by $b_{i,j} = a_i a_j / \sigma$.
	\item Consider a network $\Gamma'$ on $M_n$ with edge weight $b_{i,j}$.  Then $\Gamma'$ has the same boundary behavior as a a network on $\bigstar_n$ if and only if $b_{i,j} b_{k,\ell} = b_{i,k} b_{j,\ell}$ for distinct $i,j,k,\ell$ and
	\[
	\sum_{j \neq i} b_{i,j} + \frac{b_{i,k} b_{i,\ell}}{b_{k,\ell}} \neq 0 \text{ for some } i, k, \ell.
	\]
	In this case, the edge weights on the star are given by
	\[
	a_i = \sum_{j \neq i} b_{i,j} + \frac{b_{i,k} b_{i,\ell}}{b_{k,\ell}},
	\]
	which is independent of the choice of $k$ and $\ell$.
\end{itemize}
\end{lemma}

\begin{proof}
Observe that for any edge weights on $M_n$, the Dirichlet problem has a unique solution, that is, there is a unique harmonic function that achieves any given boundary potentials.  If $\Gamma$ is a network on an $n$-star and $\sigma = \sum_j a_j = 0$, then any harmonic function $u$ must satisfy $\sum_j a_j u(j) = 0$.  Hence, the Dirichlet problem does not always have a solution.  Therefore, the star cannot have the same boundary behavior as a network where the Dirichlet problem always has some solution.

On the other hand, if $\sigma \neq 0$, then the Dirichlet problem has a unique solution given by $u(0) = \sum_{j=1}^n a_j u(j) / \sigma$.  Moreover, if we set $b_{i,j} = a_i a_j / \sigma$, then $\Delta u(j) = a_j(u(j) - u(0)) = \sum_i b_{i,j}(u(j) - u(i))$.  Thus, the star has the same behavior as a network on $M_n$, and the $b_{i,j}$'s are also uniquely determined.

To prove the second claim, suppose $\Gamma'$ is a network on $M_n$.  If the network has the same boundary behavior as some $n$-star, then the previous argument shows that $b_{i,j} = a_i a_j / \sigma$.  Thus, for distinct $i,j,k,\ell$,
\[
b_{i,j} b_{k,\ell} = \frac{a_i a_j a_k a_\ell}{\sigma^2} = b_{i,k} b_{j,\ell}.
\]
Also,
\[
\sum_{j \neq i} b_{i,j} + \frac{b_{i,k} b_{i,\ell}}{b_{k,\ell}} = \sum_{j \neq i} \frac{a_i a_j}{\sigma} + \frac{a_i a_k a_i a_\ell}{\sigma a_k a_\ell} = a_i \neq 0.
\]

Suppose conversely that $\Gamma'$ satisfies $b_{i,j} b_{k,\ell} = b_{i,k} b_{j,\ell}$ for distinct $i,j,k,\ell$ and
\[
\sum_{j \neq i} b_{i,j} + \frac{b_{i,k} b_{i,\ell}}{b_{k,\ell}} \neq 0 \text{ for some } i, k, \ell.
\]
Fix $i$ and choose distinct $k, \ell \neq i$, and let
\[
a_i = \sum_{j \neq i} b_{i,j} + \frac{b_{i,k} b_{i,\ell}}{b_{k,\ell}}.
\]
The ``quadrilateral rule'' $b_{i,j} b_{k,\ell} = b_{i,k} b_{j,\ell}$ guarantees that the right hand side is independent of $k$ and $\ell$.  By assumption at least one of the $a_i$'s is nonzero.  By extending $\mathbb{F}$ to a larger field if necessary, we can assume that there exists $c_i$ with
\[
c_i^2 = b_{i,k} b_{i,\ell} / b_{k,\ell} \text{ for distinct } k, \ell \neq i,
\]
and again this is independent of $k, \ell$.  Then
\[
c_i^2 c_j^2 = \frac{b_{i,k} b_{i,j}}{b_{j,k}} \frac{b_{j,k} b_{i,j}}{b_{i,k}} = b_{i,j}^2
\]
so that $c_i c_j = \pm b_{i,j}$.  By choosing $c_1$ first and then modifying $c_j$ for $j \neq 1$ if necessary, we can guarantee $c_1 c_j = b_{1,j}$ for $j \neq 1$.  Then for $i \neq 1$ we have
\[
c_i c_j = b_{1,i} b_{1,j} / c_1^2 = b_{i,j}
\]
as well.  Then
\[
a_i = \sum_{j \neq i} b_{i,j} + \frac{b_{i,k} b_{i,\ell}}{b_{k,\ell}} = \sum_{j \neq i} c_i c_j + c_i^2 = c_i \sum_{j=1}^n c_j.
\]
Since at least one $a_i$ is nonzero, we have $\sum_{j=1}^n c_j \neq 0$; hence, all the $a_i$'s are nonzero.  Moreover,
\[
\sigma = \sum_{i=1}^n c_i \sum_{j=1}^n c_j = \left( \sum_{i=1}^n c_i \right)^2 \neq 0.
\]
The network $\Gamma'$ is equivalent to the network on the star because
\[
\frac{a_i a_j}{\sigma} = \frac{\left( c_i \sum_{k=1}^n c_k \right) \left( c_j \sum_{k=1}^n c_k \right)}{\left( \sum_{k=1}^n c_k \right)^2} = c_i c_j = b_{i,j}. \qedhere
\]
\end{proof}

\begin{lemma} \label{lem:relations}
Let $\mathbb{F} \neq \mathbb{F}_2$.  For any distinct indices $i, j, k$, $\Xi_{i,k}(t)$ can be expressed in terms of $\Xi_{i,j}$'s, $\Xi_j$'s, and $\Xi_{j,k}$'s.
\end{lemma}

\begin{proof}
For simplicity in drawing pictures, we will assume $i = 1$, $j = 2$, and $k = 3$.  We can also assume $n = 3$, since for general $n$, one simply has to add more rows/columns to all the matrices, filling the new spaces with ones on the diagonal and zeroes elsewhere (this corresponds to adding isolated input/output boundary vertices to an IO-network for the indices larger than $3$).

We begin with an IO-network representing $\Xi_{1,3}(a)$ for given $a \neq 0$.  Here the inputs are red and the outputs green, and the inputs/outputs $1$, $2$, and $3$ are in order from top to bottom:
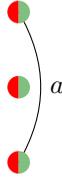
\begin{center}
	\begin{tikzpicture}
		\draw (0,1) to[bend left=30] node[auto] {$a$} (0,-1);
		
		\rgvertex{0}{1}
		\rgvertex{0}{0}
		\rgvertex{0}{-1}
	\end{tikzpicture}
\end{center}
For some parameter $b$ to be chosen later, add in a series with conductances $b$ and $-b$, representing $\Xi_2(1/b)$ and $\Xi_2(-1/b) = \Xi_2(1/b)^{-1}$:
\begin{center}
	\begin{tikzpicture}
		\draw (-2,0) to node[auto] {$-b$} (-1,0);
		\draw (-1,0) to node[auto] {$b$} (0,0) to (1,0);
		\draw (0,1) to node [auto] {$a$} (0,0) to (0,-1);
	
		\rvertex{-2}{0}
		\ivertex{-1}{0}
		\gvertex{1}{0}
		\rgvertex{0}{1}
		\rgvertex{0}{-1}
	\end{tikzpicture}
\end{center}
Next, add some cancelling parallel edges.  Two of them, for instance, correspond to inserting $\Xi_{1,2}(a)$ and its inverse $\Xi_{1,2}(-a)$ into our factorization in $EL_n$.  In the picture, the crossing edges in the middle are not labelled; their weights are shown in the previous picture.
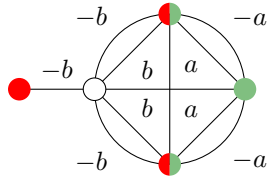
\begin{center}
	\begin{tikzpicture}
		\draw (-2,0) to node[auto] {$-b$} (-1,0);
		\draw (-1,0) to (1,0);
		\draw (0,1) to (0,-1);
		\draw (-1,0) to node[auto,swap] {$b$} (0,1);
		\draw (-1,0) to node[auto] {$b$} (0,-1);
		\draw (1,0) to node[auto] {$a$} (0,1);
		\draw (1,0) to node[auto,swap] {$a$} (0,-1);
		\draw (-1,0) to[bend left=45] node[auto] {$-b$} (0,1);
		\draw (-1,0) to[bend right=45] node[auto,swap] {$-b$} (0,-1);
		\draw (1,0) to[bend right=45] node[auto,swap] {$-a$} (0,1);
		\draw (1,0) to[bend left = 45] node[auto] {$-a$} (0,-1);
	
		\rvertex{-2}{0}
		\ivertex{-1}{0}
		\gvertex{1}{0}
		\rgvertex{0}{1}
		\rgvertex{0}{-1}
	\end{tikzpicture}
\end{center}
We want to choose $b$ so that the $4$-mesh subnetwork in the middle will be equivalent to a star.  Examining the formulas in Lemma \ref{lem:stark}, we choose $b \neq 0$ so that $a + 3b \neq 0$, which is possible because $\F$ has at least three elements.  Set
\[
c = 3a + a^2/b = (a + 3b)(a/b) \neq 0, \qquad d = a + 3b,
\]
and then the $4$-mesh is equivalent to a $4$-star with conductances $c$, $d$, $d$, $d$, and hence our network becomes
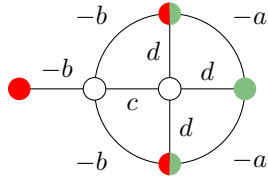
\begin{center}
	\begin{tikzpicture}
		\draw (-2,0) to node[auto] {$-b$} (-1,0);
		\draw (-1,0) to node[auto,swap] {$c$} (0,0) to node[auto] {$d$} (1,0);
		\draw (0,1) to node[auto,swap] {$d$} (0,0) to node[auto] {$d$} (0,-1);
		\draw (-1,0) to[bend left=45] node[auto] {$-b$} (0,1);
		\draw (-1,0) to[bend right=45] node[auto,swap] {$-b$} (0,-1);
		\draw (1,0) to[bend right=45] node[auto,swap] {$-a$} (0,1);
		\draw (1,0) to[bend left = 45] node[auto] {$-a$} (0,-1);
	
		\rvertex{-2}{0}
		\ivertex{-1}{0}
		\ivertex{0}{0}
		\gvertex{1}{0}
		\rgvertex{0}{1}
		\rgvertex{0}{-1}
	\end{tikzpicture}
\end{center}
This represents $\Xi_{1,3}(a)$ as the product of\footnote{The matrix at the top of the list is applied first, which means that it goes on the \emph{right} when we write the product out.}
\begin{align*}
&\Xi_2(-1/b) \\
&\Xi_{1,2}(-b) \Xi_{2,3}(-b) \\
&\Xi_2(1/c) \\
&\Xi_{1,2}(d) \Xi_{2,3}(d) \\
&\Xi_2(1/d) \\
&\Xi_{1,2}(-a) \Xi_{2,3}(-a),
\end{align*}
which completes the proof.
\end{proof}

\begin{proposition} \label{prop:generators}
Let $\mathbb{F} \neq \mathbb{F}_2$.  The electrical linear group is generated by $\Xi_j(t)$ for $j = 1$, \dots, $n$ and $\Xi_{j,j+1}(t)$ for $j = 1$, \dots, $n-1$ and $t \in \mathbb{F} \setminus \{0\}$.
\end{proposition}

\begin{proof}
For $k > j$,  we want to show that $\Xi_{j,k}(t)$ can be expressed in terms of $\Xi_j(t)$ for $j = 1$, \dots, $n$ and $\Xi_{j,j+1}(t)$ for $j = 1$, \dots, $n-1$.  By induction on $k - j$, it suffices to show $\Xi_{j,k}(t)$ can be expressed in terms of $\Xi_{k-1}$'s, $\Xi_{j,k-1}$'s and $\Xi_{k-1,k}$'s, which follows from the last lemma.
\end{proof}

\begin{theorem} \label{thm:planarequivalent}
For $\F \neq \F_2$, every element of $EG_n(\F)$ can be represented by a layerable circular planar network.
\end{theorem}

Recall that we defined $EL_n(\mathbb{F})$ using $\Xi_j(t)$ together with $\Xi_{j,k}(t)$ for all $j \neq k$.  The smaller set of generators in Proposition \ref{prop:generators} more closely resembles \cite{LP}'s definition of the electrical linear group.  If we view $EL_n(\F)$ as acting on $EG_n(\F)$ by adjoining boundary spikes and boundary edges to networks, the theorem says that it suffices to consider adjoining boundary edges between consecutively-indexed boundary vertices, rather than between any pair of boundary vertices.

\begin{proof}[Proof of Theorem \ref{thm:planarequivalent}]
By Corollary \ref{cor:equivalentlayerable}, any element of $EG_n(\mathbb{F})$ can be represented by a layerable network, and hence has the form
\[
A(\mathbb{F}^n \times 0^n) \text{ for some } A \in EL_n(\mathbb{F}).
\]
But $A$ can be represented as a product of the generators in Proposition \ref{prop:generators}, which implies that $A(\mathbb{F}^n \times 0^n)$ is the boundary behavior of a network obtained from a network of isolated boundary vertices by adjoining boundary spikes, and adjoining boundary edges between consecutively-indexed boundary vertices.  If we embed the original network with $n$ isolated boundary vertices in the disk with the boundary vertices indexed in CCW order, then at each step the modified network can still be embedded in the disk with the boundary vertices indexed in CCW order, so Theorem \ref{thm:planarequivalent} follows.
\end{proof}

\begin{remark*}
Our proof of Theorem \ref{thm:planarequivalent} is a terribly inefficient algorithm for constructing a circular planar network representing a given boundary behavior, in the sense that it adds too many unnecessary edges.  Future research may find a better method | perhaps by giving a circular planar version of the proof of Lemma \ref{lem:lagrangian} or Theorem \ref{thm:ELsymplectic}.
\end{remark*}

One might hope to show that any boundary behavior can be represented by a \emph{critical} circular planar network, but this is overly optimistic.  Consider the following network:
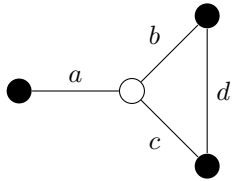
\begin{center}
	\begin{tikzpicture}
		\node[circle,draw] (4) at (0,0) {};
		\node[circle,fill] (1) at (-1.5,0) {};
		\node[circle,fill] (2) at (1,1) {};
		\node[circle,fill] (3) at (1,-1) {};
		\draw (1) to node[auto] {$a$} (4);
		\draw (4) to node[auto] {$b$} (2);
		\draw (4) to node[auto,swap] {$c$} (3);
		\draw (2) to node[auto] {$d$} (3);
	\end{tikzpicture}
\end{center}
Suppose that $a + b + c = 0$ and $1/b + 1/c + 1/d = 0$ (which can happen for most fields).  In this case, the boundary potentials do not uniquely determine the boundary currents, nor do the boundary currents determine the boundary potentials up to constants.  However, there does not exist a critical circular planar network, or indeed any network recoverable over positive linear conductances, which has this property and has only three boundary vertices.  For the Dirichlet problem to not have a unique solution, it must have an interior vertex, and the interior vertex must have degree $\geq 3$ for the network to be critical circular planar, since a series is not recoverable.  Since a recoverable network with $3$ boundary vertices cannot have more than $3$ edges by consideration of the number of variables, the only possibility is a $Y$.  However, in a $Y$, the Neumann problem has a unique solution.

This example also shows that not every network is equivalent to a network with $\leq \frac{1}{2} n (n - 1)$ edges, as we might hope, so the bound in Corollary \ref{cor:equivalentlayerable} is sharp in this case.

\section{Generalizations and Open Problems} \label{sec:concluding}

\subsection{Nonlinear Networks} \label{subsec:nonlinear}

Johnson's treatment of harmonic continuation \cite{WJ} was motivated by the question of how to recover networks with non-Ohmic resistors.  The current on each edge is given as a nonlinear function of the voltage, that is, $\gamma_e(du(e))$, where $\gamma_e: \R \to \R$.  For $\gamma_e$ to be physically reasonable, one would require that $\gamma_e(0) = 0$ and $\gamma_e$ is increasing.  However, as we shall see, the inverse problem can be solved whenever $\gamma_e(0) = 0$ and $\gamma_e$ is a bijection.

Another more algebraic generalization was described by Avi Levy and the author in \cite{torsion}, motivated by the algebraic-topological perspective on the graph Laplacian in \cite{DKM}.  We can take the edge weights to be units in a ring $R$, and consider potential and current functions taking values in $R$, or more generally in an $R$-module $M$.  The $\Z$-module of harmonic functions on the network could be studied using homological algebra.

A little reflection shows that all our harmonic continuation arguments work in very general situations, including the two described above.  It relied exclusively upon the following ingredients:
\begin{itemize}
	\item The currents on edges and the potentials on vertices can be added together.
	\item The current on an edge is a function by the voltage across the edge.
	\item Conversely, the voltage on an edge is a function of the current on the edge.
	\item Zero voltage corresponds to zero current.
\end{itemize}
This motivates the following definition:
\begin{definition}
Let $M$ be an abelian group, written additively.  A {\bf $BZ(M)$-network} is a $\partial$-graph together with a bijection $\gamma_e: M \to M$ such that
\[
\gamma_e(0) = 0 \text{ and } \gamma_{\overline{e}}(x) = -\gamma_e(-x),
\]
where the second condition guarantees that the current on $e$ is negative the current on $\overline{e}$ and that $\gamma_e$ is uniquely determined by $\gamma_{\overline{e}}$.
\end{definition}

\begin{remark*}
The set-up given here can be generalized even further.  For instance, Kenyon considers a vector bundle Laplacian where the potentials at each vertex take values in some vector space \cite{RK}.  There is a different vector space for each vertex, and each edge has an associated ``parallel transport'' isomorphism between the different vector spaces, which allows us to compare potentials on the two endpoints.  We omit this case for the sake of simpler notation and leave it to the reader to generalize to cases that interest them.
\end{remark*}

The arguments given here adapt almost word for word to show that
\begin{theorem}
If a $\partial$-graph is recoverable by scaffolds, then it is recoverable over $BZ(M)$ for any $M$, that is, the function $\gamma_e: M \to M$ for each $e$ is uniquely determined by $\Lambda(\Gamma)$.
\end{theorem}
In particular, we have reproved the main result of \cite{WJ} that critical circular planar networks are recoverable in the nonlinear case.

The category of IO-networks generalizes easily to the $BZ(M)$ case, and the IO boundary behavior $X$ is now a functor from the category of nonlinear IO-networks to the category of relations (not necessarily linear).  Now suppose we have an elementary factorization of $[\Gamma,i,j]: P \to Q$ consisting of $[\Gamma_k,i_k,j_k]: P_{k-1} \to P_k$ and choose $P_\ell$ with $|P_\ell| = m = m(P,Q)$.  Suppose there are $k$ input stubs and $\ell$ output stubs.  Then, starting in the middle with $P_\ell$ and working toward the beginning and the end of the factorization, we can use harmonic continuation to parametrize $X([\Gamma,i,j])$ by $(M^{P_\ell})^2 \times M^k \times M^\ell$.\footnote{However, the same argument does not work for semi-elementary factorizations.}

This implies that if there is an elementary factorization, then the rank-connection principle generalizes to the nonlinear case, provided we have a suitable notion of dimension.  For instance, if $M$ is finite, then $\{x: \exists y \text{ with } (x,y) \in X([\Gamma,i,j])\}$ has cardinality $|M|^{2m+k}$ and $\{x: (x,0) \in X([\Gamma,i,j])\}$ has cardinality $|M|^k$.  Thus, $m$ can be detected from $X([\Gamma,i,j])$ by looking at the size of these sets.  Similarly, if $M = \R$ and $\gamma_e: M \to M$ is a homeomorphism, then $\{x: \exists y \text{ with } (x,y) \in X([\Gamma,i,j])\}$ has dimension $2m + k$ as a topological manifold and $\{x: (x,0) \in X([\Gamma,i,j])\}$ has dimension $k$.  So again, $m$ can be detected from the boundary behavior.

Another consequence of elementary factorizations is
\begin{proposition} \label{prop:manifold}
Suppose $M = \R$ and consider networks where $\gamma_e: \R \to \R$ is a homeomorphism.  If $\Gamma$ is a finite layerable network with $n$ boundary vertices, then $\Lambda(\Gamma)$ is properly embedded topological submanifold of $\R^{2n}$ which is homeomorphic to $\R^n$.
\end{proposition}

\begin{proof}
Assume $[n] = \partial V$.  Since $\Gamma$ is layerable, as discussed in \S \ref{subsec:IOlayerstripping}, we can express $[\Gamma,i,j]: \varnothing \to [n]$ as
\[
[\Gamma_n,i_n,j_n] \circ \dots \circ [\Gamma_0,i_0,j_0],
\]
where $[\Gamma_0,i_0,j_0]: \varnothing \to [n]$ is a network consisting of isolated boundary vertices, and the other morphisms correspond to adding boundary spikes or boundary edges.  Since $\gamma_e$ is a homeomorphism, we can see that for $k \geq 1$, $X_k = X([\Gamma,i_k,j_k])$ defines a homeomorphism $\R^{2n} \to \R^{2n}$.  Now
\[
\Lambda(\Gamma) = X_n \circ \dots \circ X_1(\R^n \times 0),
\]
which proves the asserted claims.
\end{proof}

In the smooth case, we can refine this to

\begin{proposition} \label{prop:lagrangianmanifold}
Let $\Gamma$ be a finite network over $\R$.  Suppose that $\gamma_e: \R \to \R$ is a diffeormorphism.  Then
\begin{itemize}
	\item $\Lambda(\Gamma)$ is a smooth submanifold of $\R^{2n}$ which is diffeomorphic to $\R^n$.
	\item If $d_u\Gamma$ is the linear network with edge weights $\gamma_e'(u(e_+) - u(e_-))$, then the tangent space
	\[
	T_{(u|_{\partial V}, \Delta u|_{\partial V})}\Lambda(\Gamma) = \Lambda(d_u\Gamma).
	\]
	\item The tangent space to $\Lambda(\Gamma)$ at each point is a Lagrangian subspace of $\R^{2n}$, which means that $\Lambda(\Gamma)$ is a Lagrangian submanifold of $\R^{2n}$.
\end{itemize}
\end{proposition}

\begin{proof}
Proceeding as in the previous proof, we see that $X_k$ is a diffeomorphism, which establishes the first claim.  Moreover, by direct computation, the derivative $DX_k$ at the point corresponding to a potential function $u$ is given by the symplectic matrix corresponding to adjoining spikes or boundary edges of weight $\gamma_e'(u(e_+) - u(e_-))$.  Thus, $D(X_n \circ \dots \circ X_1)$ is the product of symplectic matrices corresponding to adding edges of weight $\gamma_e'(u(e_+) - u(e_-))$.  This establishes the second claim, and the third follows immediately from Theorem \ref{thm:symplectic}.
\end{proof}

The fact that $\Lambda(\Gamma)$ is a submanifold is nontrivial and does not hold for all networks.  Indeed, consider the following network:
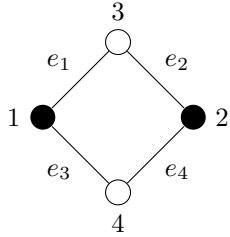
\begin{center}
	\begin{tikzpicture}
		\node[circle,fill] (1) at (-1,0) [label = left: $1$] {};
		\node[circle,fill] (2) at (1,0) [label = right: $2$] {};
		\node[circle,draw] (3) at (0,1) [label = above: $3$] {};
		\node[circle,draw] (4) at (0,-1) [label = below: $4$] {};
		\draw (1) to node[auto] {$e_1$} (3) to node[auto] {$e_2$} (2) to node[auto] {$e_4$} (4) to node[auto] {$e_3$} (1);
	\end{tikzpicture}
\end{center}
Define $\gamma_e: \R \to \R$ by $\gamma_e = \rho_e^{-1}$, where $\rho_{e_1}(t) = \rho_{e_3}(t) = t + \frac{1}{2} \sin t$ (the orientation of the edge does not matter since the function is odd), and $\rho_{e_2}(t) = \rho_{e_3}(t) = -t$.  These are bijective $C^\infty$ resistance functions with a $C^\infty$ inverse.

We can view $\rho_e$ as a \emph{resistance function} which gives the voltage on an edge as a function of the current.  The series with resistance functions $\rho_{e_1}$ and $\rho_{e_2}$ is equivalent to a single-edge with resistance $\rho_{e_1} + \rho_{e_2}$.  Thus, the network is equivalent to a parallel connection
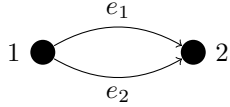
\begin{center}
	\begin{tikzpicture}
		\node[circle,fill] (1) at (-1,0) [label = left: $1$] {};
		\node[circle,fill] (2) at (1,0) [label = right: $2$] {};
		\draw[->] (1) to[bend left = 30] node[auto] {$e_1$} (2);
		\draw[<-] (2) to[bend left = 30] node[auto] {$e_2$} (1);
	\end{tikzpicture}
\end{center}
in which each edge has resistance function $\rho(t) = \frac{1}{2} \sin t$.  Though this is not a $BZ(M)$ network, it still makes sense to talk about harmonic functions as being given by a potential $u: V \to \R$ and a compatible current function $c: E \to \R$ with $c(\overline{e}) = -c(e)$.  Though the manipulations we are about to do can be understood without reducing the series to a single edge with a rather degenerate resistance function, we think that this is conceptually simpler.

Let $e_1$ and $e_2$ be the oriented edges shown in the picture.  Thus, a potential function $u$ has a compatible current function $c: E \to \R$ if and only if
\[
u_1 - u_2 = \tfrac{1}{2} \sin c_{e_1} = \tfrac{1}{2} \sin c_{e_2}.
\]
Now $\sin c_{e_1} = \sin c_{e_2}$ is equivalent to $c_{e_2} = c_{e_1} + 2 \pi n$ or $c_{e_2} = \pi - c_{e_1} + 2 \pi n$.  If $c_{e_1} = c_{e_2} + 2 \pi n$, then the net current $\psi_1 = c_{e_1} + c_{e_2} = 2c_{e_1} + 2 \pi n$ and $\psi_2 = -\psi_1$ and $u_1 - u_2$ must be $\frac{1}{2} \sin \psi_1/2$.  If $c_{e_2} = \pi - c_{e_1} + 2 \pi n$, then $\psi_1 = (2n + 1)\pi$ and $\psi_2 = -\psi_1$ and $u_1 - u_2$ could be any number in $[-1,1]$.  Thus,
\begin{align*}
L = &\{(\phi,\psi): \phi_1 - \phi_2 = \tfrac{1}{2} \sin \psi_1/2, \, \psi_1 = -\psi_2\} \\ &\cup \{(\phi,\psi): \phi_1 - \phi_2 \in [-1,1], \, \psi_1 = (2n + 1) \pi, \, \psi_2 = -\psi_1\}.
\end{align*}
This is not a smooth manifold in any neighborhood of the points where $\phi_1 - \phi_2 = \pm 1$ and $\psi_1 = (2n + 1) \pi$.

\subsection{More General Sufficient Conditions for Recoverability}

The condition of recoverability for scaffolds is not as general as possible.  One improvement on recoverability by scaffolds stems from the observation that Lemmas \ref{lem:hcuniqueness} and \ref{lem:hcexistence} only used ``partial scaffolds'' defined on subgraphs of $G$.  Moreover, in the proof of Lemma \ref{lem:recovery}, we did not use all of the scaffold $S$, but only the parts in the regions $\Gamma_1$ and $\Gamma_2$.  Thus, we can define a class of $\partial$-graphs that are {\bf recoverable through partial scaffolds}.  This condition would be more complicated to state and more unwieldy in many situations.  But it would still pull back through UHMs, since the scaffolds defined on subgraphs can still be pulled back functorially.

Similar to the generalization of elementary factorizations in \S \ref{subsec:semielementary}, we can adapt the definition of scaffolds to account for harmonic continuation steps that define a function to be constant on some subnetwork.  Moreover, we can use the fact that a boundary wedge sum of recoverable networks is recoverable, provided one of them is finite.  However, the preimage of a boundary wedge-sum under a UHM is not a boundary wedge-sum.  Similarly, the more general notion of harmonic continuation given by semi-elementary factorizations does not pull back functorially under UHMs.

We did not develop these more general conditions systematically because they were not needed for most graphs on surfaces or examples that one usually comes up with by hand.  Nonetheless, such conditions might have some hope of geometrically characterizing recoverability, though it not clear to the author how to prove this.

Moreover, the notion of ``recoverability'' is subtle.  A network can be recoverable for \emph{generic} edge weights over some algebraically closed field without being recoverable for \emph{all} edge weights.  It is known that some networks which are not even completely reducible are recoverable for positive real edge weights.  Given that the geometric characterizations in this paper correspond to algebraic conditions holding for \emph{all} edge weights in an infinite field, we expect that it is easier to give a geometric characterization for universal recoverability over an infinite field, but testing weaker forms of recoverability would require a different approach.

A major weakness of the machinery developed here for solving the inverse problem is that it is useful almost exclusively for proving \emph{positive} results.  We have not described how to prove a $\partial$-graph is \emph{not} recoverable, \emph{not} recoverable by scaffolds, or \emph{not} totally layerable.  It would be very useful to have some algebraic or combinatorial invariants (not directly related to the inverse problem) that could be used to prove negative results about recoverability, or about recoverability by scaffolds.

\subsection{Infinite Networks}

Recoverability for scaffolds makes sense for infinite networks, but much of our theory has not been fully fleshed out in the infinite case.  For instance,
\begin{itemize}
	\item What is the analogue of completely reducible $\partial$-graphs in the infinite case, and does a version of Proposition \ref{prop:completelyreducible} hold?
	\item Can we generalize Theorem \ref{thm:RC3} if we allow an infinite size connection to correspond to an infinite rank?
	\item Can we define elementary factorizations with infinitely many factors using categorical limits and prove a version of Proposition \ref{prop:IOscaffolds}?
\end{itemize}

Moreover, as mentioned in \S \ref{subsec:halfplanar}, there are several reasonable definitions of $\Lambda(\Gamma)$ in the infinite case.  Over arbitrary fields, the two feasible choices are (1) the boundary data of all harmonic functions and (2) the boundary data of finitely supported harmonic functions.  We have adopted the first definition, but the idea of recoverability by scaffolds works using the second definition as well, so long as we guarantee that harmonic continuation produces finitely supported functions.

It is not clear in general whether recoverability using (1) and (2) are equivalent or whether either one implies the other.

\end{document}